\documentclass[oneside,english,reqno]{amsart}

\makeatletter

\usepackage[top=2cm,bottom=2cm,left=3cm,right=3cm]{geometry}
\usepackage{graphicx} 
\usepackage{subcaption}
\usepackage{blindtext}
\usepackage[utf8]{inputenc} 
\usepackage{amsmath}
\usepackage{amsfonts}
\usepackage{mathtools}
\usepackage{graphicx}
\usepackage{listings}
\usepackage{epigraph}
\usepackage{color}
\usepackage{textcomp}
\usepackage{gensymb}
\usepackage{wrapfig}
\usepackage[OT2, T1]{fontenc}

\usepackage{rotating}
\usepackage[normalem]{ulem}
\useunder{\uline}{\ul}{}
\usepackage{float}
\usepackage{booktabs}
\usepackage{xspace}
\usepackage{csquotes}
\usepackage{mathrsfs}
\usepackage{stmaryrd}
\usepackage[english]{babel}
\usepackage{amsthm}
\usepackage{ragged2e}
\usepackage{dsfont}
\usepackage{geometry}
\usepackage{tabularx}
\usepackage{indentfirst}
\usepackage{caption}
\usepackage{eso-pic}
\usepackage{url}
\usepackage{afterpage}
\usepackage{parskip}
\usepackage{listings}
\usepackage{fancyhdr}
\usepackage{textcomp}
\usepackage{multirow}
\usepackage[utf8]{inputenc}
\usepackage{setspace}
\usepackage[safe]{tipa}
\usepackage{subcaption}
\usepackage[export]{adjustbox}
\usepackage{wrapfig}
\usepackage{array}
\usepackage[table]{xcolor}
\usepackage{tikz-cd}
\usepackage{pdfpages}
\usepackage{comment}
\usepackage{todonotes}
\newtheorem{theoremA}{Theorem}

\newtheorem{theoremB}{Theorem}

\newtheorem{theoremC}{Theorem}

\newtheorem{theoremD}{Theorem}

\DeclareSymbolFont{cyrletters}{OT2}{wncyr}{m}{n}
\DeclareMathSymbol{\Sha}{\mathalpha}{cyrletters}{"58}
\DeclareMathSymbol{\Yu}{\mathalpha}{cyrletters}{"10}

\newcommand{\dd}{\mathrm{d}}
\newcommand{\Mbar}{\overline{\mathcal{M}}}

\newcommand{\M}{\mathcal{M}}
\newcommand{\Mct}{\mathcal{M}^{\textnormal{ct}}}

\newcommand{\Z}{\mathcal{Z}}

\newcommand{\1}{\mathds{1}}

\newcommand{\st}{such that\xspace}
\newcommand{\Id}{\textnormal{Id}}

\newcommand{\wrt}{with respect to\xspace}

\newcommand{\End}{\textnormal{End}}
\newcommand{\Hom}{\textnormal{Hom}}

\newtheorem{theorem}{Theorem}[section]

\newtheorem{proposition}[theorem]{Proposition}
\newtheorem{lemma}[theorem]{Lemma}
\newtheorem{corollary}[theorem]{Corollary}

\newtheorem{definition}[theorem]{Definition}
\newtheorem{remark}[theorem]{Remark}

\makeatletter
\newcommand{\vast}{\bBigg@{3}}
\newcommand{\Vast}{\bBigg@{4}}
\makeatother

\date{}

\makeatother

\PassOptionsToPackage{table}{xcolor}

\usepackage[colorlinks,allcolors={blue}, hidelinks]{hyperref}

\begin{document}
\global\long\def\DR#1#2{{\rm DR}_{#1}\left(#2\right)}%
\global\long\def\DRR#1#2{{\rm DR}_{#1}^{1}\left(#2\right)}%

\title{Reconstruction of F-cohomological field theories on moduli of compact type}

\author{Ga{\"e}tan Borot}
\address{G.~Borot:\newline Humboldt-Universit{\"a}t zu Berlin, Institut f{\"u}r Mathematik und Institut f{\"u}r Physik, Unter den Linden 6, 10099 Berlin, Germany.}%
\email{gaetan.borot@hu-berlin.de}

\author{Silvia Ragni}
\address{S.~Ragni:\newline Humboldt-Universit{\"a}t zu Berlin, Institut f{\"u}r Mathematik, Unter den Linden 6, 10099 Berlin, Germany.}
\email{silvia.ragni.1@hu-berlin.de}

\author{Paolo Rossi}
\address{P.~Rossi:\newline Dipartimento di Matematica "Tullio Levi-Civita". 
Universit\`a degli Studi di Padova, 
Via Trieste 63,
35121 Padova, Italy.}
\email{paolo.rossi@math.unipd.it}

\begin{abstract}
We prove an analogue of Givental--Teleman reconstruction for compact-type F-cohomological field theories. We apply it to reconstruct the restriction of the extended $r$-spin classes to the extended direction and deduce relations between $\kappa$-classes (both in compact type).
\end{abstract}

\maketitle
\tableofcontents{}

\newpage

\section{{\Large Introduction}}
\label{Sec1}
\medskip

Cohomological field theories (CohFTs) were first introduced in \cite{kontsevich_gromov-witten_1994} to formalise the universal properties of Gromov--Witten theory. They are families of cohomology classes on the moduli spaces of stable, marked curves, compatible with their boundary stratification and equivariant with respect to permutations of the marked points, whose degree $0$ part form a two-dimensional topological field theory (TFT) --- equivalently, a Frobenius algebra \cite{abrams_two-dimensional_1996}. Their top degree part (resp. intersection with polynomials in $\psi$-classes) is encoded in a generating series called potential (respectively, ancestor potential), whose genus $0$ part encodes the structure of a Frobenius manifold \cite{dubrovin_geometry_1996}: an analytic manifold whose tangent space at each point is a Frobenius algebra, together with certain integrability conditions for such algebra bundles. In \cite{givental_semisimple_2001}, Givental introduced a group acting on ancestor potentials and 
used it to conjecturally reconstruct them from the Frobenius manifold in semi-simple cases.  The fact that this action preserves the property of being the ancestor potential of a CohFT was later justified in \cite{Lee1,Lee2,FSZ}. In \cite{teleman_structure_2012}, Teleman lifted this action
to the level of CohFTs\footnote{The history of this action is convoluted. It seems to have been discovered independently by Kazarian and by Kontsevich, but first appears in print in \cite{teleman_structure_2012}.} and proved Givental's conjecture. Thanks to Givental--Teleman theory, CohFTs have turned very useful in algebraic geometry, \textit{e.g.} to compute Gromov--Witten invariants or to find relations in the tautological ring \cite{pandharipande_relations_2015}. They also have been used in combination with the geometry of the double ramification cycle to construct integrable hamiltonian hierarchies \cite{BuryakDR,buryak_new_2017}.

F-cohomological field theories (F-CohFTs) are variants of CohFTs where we only require compatibility with the boundary stratification of the moduli of stable curves with compact Jacobian (i.e. stable curves whose dual graph is a tree), and reduce the permutation invariance to single out one special marked point. They were introduced in \cite{buryak_extended_2021}, although a precursor example (the notion of partial CohFT) can be found in Fan--Jarvis--Ruan--Witten theories \cite{BCFG}. Combined with the double ramification cycle they give rise to non-hamiltonian integrable hierarchies \cite{ABLRint}. The degree $0$ part of a F-CohFT is not a Frobenius algebra anymore, but a F-TFT, that is a commutative associative algebra equipped with a distinguished vector $\alpha$ corresponding to the degree $0$ part of the F-CohFT on  the moduli of pointed elliptic curves. The genus $0$ part of the potential of a F-CohFT determines a flat F-manifold. This is a weaker version of Frobenius manifold, in particular lacking a flat metric compatible with the product, and dating back to \cite{hertling_weak_1999,manin_f_2005,LPR09}. Flat F-manifolds also appear in open Gromov--Witten theory \cite{Basa}. The authors of \cite{arsie_semisimple_2023} introduced a variant of the Givental group, called \emph{F-Givental group} in \cite{borot_symmetries_2024}, acting on F-CohFTs. They proved that its action is transitive on flat F-manifold potentials in the semi-simple case; moreover, they showed that, given a semi-simple flat F-manifold and choice of $\alpha$, it is possible to produce a F-CohFT whose associated flat F-manifold is the original one.

In fact, in higher genus the action of the F-Givental group on F-CohFTs is far from being transitive and the reconstruction fails in general. For instance, the shifted extended $2$-spin F-CohFT of \cite{buryak_extended_2021} gives rise to a semi-simple flat F-manifold, but the F-CohFT constructed from the latter in \cite{arsie_semisimple_2023} does not agree with the shifted extended $2$-spin classes.

The goal of the paper is to show that the lack of transitivity of the F-Givental group action can be repaired and a complete analogue of the Givental--Teleman theory in the F-world exists, provided one works in restriction to the moduli of compact type. In fact this compact-type Givental--Teleman theory applies to the more general notion of ctF-CohFT, which is analogous to a F-CohFT, but only defined on moduli spaces $\M^{\textnormal{ct}}$ of compact type without requiring extendability to a F-CohFT on the full compactification $\Mbar$. Of course,  the restriction of a F-CohFT from $\Mbar$ to $\M^{\textnormal{ct}}$ gives a ctF-CohFT. On the other hand, when trying to extend a ctF-CohFT to a F-CohFT two things might go wrong: either the cohomology classes themselves do not admit an extension, or such extensions only satisfy the F-CohFT axioms after restriction to compact type. We will see an example of the second phenomenon in this paper. On ctF-CohFT side the key assumption for the Givental--Teleman theory to work with is the invertibility of $\alpha$. For CohFTs invertibility is equivalent to semi-simplicity, but for (ct)F-CohFTs they are independent properties.

\begin{theoremA}
\label{thm:transfree}
    The F-Givental group acts freely and transitively on the set of invertible ctF-CohFTs with given underlying F-TFT: if $\Omega$ is an invertible ctF-CohFT on a vector space $V$ and $\omega$ is the underlying F-TFT, then there exist unique $R \in \textnormal{End}(V)\llbracket z \rrbracket$ and $T \in z^2V\llbracket z\rrbracket$ \st 
\[
RT\omega=\Omega.
\] 
Furthermore, if the unit $\1$ is flat, then $T(z) = z(\1 - R^{-1}(z)[\1])$. 
\end{theoremA}

In particular, any invertible ctFCohFTs admits a canonical extension to $\overline{\mathcal{M}}$ which is a FCohFT, simply by considering $RT\omega$ on $\overline{\mathcal{M}}$ instead of $\mathcal{M}^{\textnormal{ct}}$. Nevertheless, there is no a priori uniqueness of extensions of ctF-CohFTs as F-CohFTs, and we stress that Theorem~\ref{thm:transfree} gives no information on invertible F-CohFTs outside $\mathcal{M}^{\textnormal{ct}}$. 

In \cite{borot_symmetries_2024} additional linear symmetries of F-CohFTs which do not commute with the F-Givental group action were described, but they leave invariant the restriction of the initial F-CohFT to $\mathcal{M}^{\textnormal{ct}}$. It would be interesting to know if the full F-CohFTs could be constructed from its F-TFT by taking into account the action of this larger group. In the semi-simple case, the ancestor potential of the F-CohFT $RT\omega$ is computed by F-topological recursion \cite{borot_symmetries_2024}. Yet, it does not mean that the ancestor potential of $\Omega$ will necessarily be, because higher-genus potential are sensitive to the classes on $\Mbar$.

After Theorem~\ref{thm:transfree} a primary question is to understand how much of this unique F-Givental group element can be reconstructed from the genus $0$ potential, \textit{i.e.} from the underlying flat F-manifold structure. Here the stronger semi-simplicity assumption comes handy. The genus $1$ data of the vector field $\alpha$ is a priori not contained in the flat F-manifold. But in the semi-simple case, we show in Lemma~\ref{lem:alphaHR1} that it must be proportional to the inverse square of a known vector potential for the rotation coefficients of the flat F-manifold \cite{arsie_darbouxegorov_2013}, and thus determined by its value at one point of the manifold.

\begin{theoremB}
\label{thm:rec}
    Let $\Omega$ be an invertible semi-simple ctF-CohFT on $V$ and $(M,\nabla,\cdot)$ be the associated germ of flat F-manifold near $0$ in $V$. Denote $(\partial_i)_{i = 1}^N$ the canonical basis of vector fields, $(u^i)_{i = 1}^N$  canonical coordinates, and $U = \textnormal{diag}(u^1,\ldots,u^N)$. Then $R(z)$ from Theorem~\ref{thm:transfree} is such that the columns of $R(z)H^{-1}e^{U/z}$ expressed in the canonical basis form a basis of flat sections for the deformed connection $\nabla - z^{-1} \cdot$. This determines $R(z)$ from the flat F-manifold up to pre-composition with $\exp(\sum_{k \geq 1} D_k z^k)$, where $D_k$ are represented by constant diagonal matrices in the canonical basis.  Furthermore,  $\Upsilon(z) := R(z)[\1 - z^{-1}T(z)]$ is uniquely determined by the flat F-manifold, see \eqref{vaceqn}.
    \end{theoremB}    

We discuss the comparison with the construction of $R$-matrices from the flat F-manifold in \cite{arsie_semisimple_2023} around \eqref{dRRossi}-\eqref{Rcorres}. The odd diagonal ambiguities $D_{2k + 1}$ are well-known in F-CohFTs and already appear in CohFTs. By Mumford's formula they correspond on the moduli of compact type to multiplication of the F-TFT by 
\[
\exp\bigg(\sum_{k \geq 1}  D_{2k - 1} \frac{(2k)!}{B_{2k}} \textnormal{ch}_{2k - 1}(\mathbb{H})_{|\textnormal{ct}}\bigg),
\]
where $\mathbb{H}$ is the Hodge bundle and $B_{k}$ are the Bernoulli numbers. The Chern characters $\textnormal{ch}_{2k - 1}(\mathbb{H})$ vanish in genus zero so cannot be detected by the flat F-manifold. In CohFTs the even ambiguities are ruled out by the symplectic condition, but they are allowed in (ct)F-CohFTs. They correspond to multiplication of the F-TFT by $\exp(\sum_{k \geq 1} D_{2k} \theta_{2k})$, where
\[
\theta_{2k} := \kappa_{2k} + \psi_{1}^{2k} - \sum_{i = 2}^{1+n} \psi_i^{2k} + \sum_{\delta} \sum_{m + m' = 2k - 1} (-1)^{m'} \psi^m (\psi')^{m'} \cup [\delta].
\]
Here the sum ranges over all boundary divisors $\delta$ of separating type, $\psi$ is associated to the node in the component containing the first marked point, and $\psi'$ to the opposite node. One can easily check that $\theta_{2k} \in H^*(\Mbar_{0,n})$ for $n > 0$ pulls back to $\theta_{2k} \in H^*(\Mbar_{0,n+1})$ by the forgetful morphism, implying by successive pullbacks from $H^*(\Mbar_{0,3})$ that $\theta_{2k}$ always vanishes in genus $0$. In passing this gives an explicit formula in genus $0$ for $\kappa_{2k}$  in terms of $\psi$- and $\kappa$-classes. We do not know if those classes come from a natural geometric construction (like the Hodge classes did).

\begin{theoremC}
\label{thm:recHom} Let $\Omega$ be a conformal, invertible, semi-simple ctF-CohFT. Then $\Omega$ is uniquely determined by the 1-jet of the conformal flat F-manifold at the origin and the value of $\alpha$ at the origin, see \eqref{dzdzR}.
\end{theoremC}

These results are directly relevant for the double ramification hierarchies obtained from F-CohFTs, as their flows only involve integrals over the DR cycle of the product of the F-CohFT with $\psi$-classes and the top Hodge class $\lambda_g$, and the latter vanishes on the complement $\M^{\textnormal{ct}}$, so that only the restriction of the F-CohFT to compact type can contribute. Notice that the DR hierarchy can also be defined in the case of ctF-CohFTs $\Omega$, under the assumption that an extension $\overline{\Omega}$ of $\Omega$ to $\Mbar$ exists (even without $\overline{\Omega}$ satisfying the axioms of a F-CohFT on $\Mbar$). Notice that $\lambda_g \overline{\Omega}$ does not depend on the choice of extension, in case the extension is not unique. In particular, for invertible ctF-CohFTs $R$ and $T$ of Theorem~\ref{thm:transfree} determine uniquely the flows of the DR hierarchy. In general the F-Givental group can be seen as a group of universal, integrability-preserving transformations connecting different integrable systems of the same rank. Even the constant diagonal ambiguity of Theorem~\ref{thm:rec} alone produces systems with essentially different dispersive expansions.

The precise definitions will be given in the text. Theorem~\ref{thm:transfree} is proved in Section~\ref{S3}, Theorem~\ref{thm:rec} in Section~\ref{sec:resdiff}-\ref{sec:proofnu} and Theorem~\ref{thm:recHom} in Section~\ref{sec:recHomo}. The strategy for the proofs follows closely the one invented by Teleman in \cite{teleman_structure_2012} for CohFTs. We propose a slightly different and essentially self-contained exposition of the arguments, \textit{e.g.} fixing some arbitrary choices by means of hyperbolic geometry, adding some explanations and details, etc. We hope that our presentation can facilitate the navigation of an interested reader in Teleman's original paper too. Theorem~\ref{thm:rec} requires computations with flat F-manifold structures having some new features compared to those for CohFTs and Frobenius manifold structures, in particular relating $\alpha$ and the Christoffel symbols (see Lemma~\ref{lem:alphaHR1}). We give a careful comparison between flat sections of the deformed flat connection and the differential equations for the R-element of the F-Givental group following from the analysis of the ctF-CohFT, and explain how those equations compare to \cite{arsie_semisimple_2023}. 

We illustrate Theorem~\ref{thm:recHom} by reconstructing a one-dimensional ctF-CohFT coming from the extension of Witten $r$-spin class \cite{JKV,BCT,buryak_extended_2021} and deriving vanishing results for certain polynomials in $\kappa$-classes in $H^*(\M^{\textnormal{ct}})$.

\begin{theoremD}
\label{thm:vani}
Let $r\geq2$
 be an integer, 
 $g \geq 0$ and 
 $n \geq 1$ 
 such that 
 $2g - 2 + n > 0$. 
Define $P_{m}^{(r)}(\boldsymbol{\kappa}) \in H^{2m}(\M_{g,n}^{\textnormal{ct}})$ 
by the formulae
\[
\sum_{m \geq 0} (rm - 1)!^{(r)} z^m = \exp\bigg(\sum_{m \geq 1} s_m z^m\bigg),\qquad \exp\bigg(-\sum_{m \geq 1} s_m \kappa_m\bigg) = 1+\sum_{m \geq 1} P_{m}^{(r)}(\boldsymbol{\kappa}) 
\]
involving the $r$-fold factorial $(rm - 1)!^{(r)} = (rm - 1)(rm - r - 1) \cdots (r - 1)$.  Then 
\begin{equation}
P_{m}^{(r)}(\boldsymbol{\kappa}) = \left\{\begin{array}{lll} 0 &&  \textnormal{if}\,\, (r - 1)(2g - 2 + n) < rm,  \\
(-1)^{g}r^{m-g} c^{r,\star}_{g,n}(\partial_t^{\otimes n-1})_{|\textnormal{ct}} && \textnormal{if}\,\, (r - 1)(2g - 2 + n) = rm,
\end{array}\right.
\end{equation}
where $c^{r,\star}$ is the restriction to the $r^\textnormal{th}$ subspace of the extended $r$-spin class of \cite{buryak_extended_2021}, \textit{cf.} Section~\ref{Sec:extrspin}.
\end{theoremD}

This result is the combined consequence of Proposition~\ref{prop:1FCoh}, Corollary~\ref{vanishingcor} and Lemma~\ref{crstarlema}. Pixton has described in \cite{Pixtonthesis} a generating set for all relations in the $\kappa$-ring of $\M^{\textnormal{ct}}$, i.e. the subring of the Chow ring consisting in polynomials in $\kappa$-classes. It is an immediate consequence of \cite[Theorem 2]{Pan-kappa} that the cycle class map from the $\kappa$-ring to the cohomology ring of $\M^{\textnormal{ct}}_{g,n}$ is injective for $n>0$, which implies in particular that there is no extra cohomological relation in the $\kappa$-ring that was not already a Chow relation. Hence Pixton's relations span all cohomological relation in the $\kappa$-ring of compact type with at least one marked point. In particular all of our relations from Theorem \ref{thm:vani} are in the span of Pixton's relations. We have double-checked this explicitly for $m,r \leq 4$ and, in particular, for $r = 2$ our relations are part of Pixton's spanning set itself\footnote{The observation of a relation between the extended $2$-spin F-CohFT and a subset of Pixton's relations is due to A. Buryak, and comes from the explicit computation in \cite{ABLR2} of the F-Givental group element $R(z)$ from its homogeneous flat F-manifold and the analysis of the restriction to the $2^\textnormal{nd}$ subspace of the corresponding F-Givental action.}. It would still be interesting to constructively derive our relations from Pixton's spanning set in all generality (i.e. to have control over the explicit linear combination of Pixton relations which produces ours, for all $r$ and $m$).

\noindent \textbf{Conventions.} Algebras are not assumed unital unless specified otherwise. We denote $[n]$ the set of integers between $1$ and $n$. If $\mathcal{Y} \subseteq \mathcal{X}$ and $\alpha$ is a cohomology class on $\mathcal{X}$, we denote $\alpha_{|\mathcal{Y}}$ its restriction to $\mathcal{Y}$, \textit{i.e.} the pullback of $\alpha$ by the natural inclusion $\mathcal{Y} \hookrightarrow \mathcal{X}$.

\noindent \textbf{Acknowledgements.} We thank A. Giacchetto for useful discussions on theories without flat unit, as well as D. Klompenhouwer, S. Perletti, S. Shadrin and S. Velstra for comments. We are grateful to A. Buryak for discussing and pointing out a relation between the extended $2$-spin theory and certain relations of Pixton's among $\kappa$-classes on compact type. We thank the anonymous referees for insightful remarks. S.R. is funded by the Deutsche Forschungsgemeinschaft RTG 2965 --- Project number 512730679. P.R. is supported by the University of Padova and is affiliated to the INFN under the national project MMNLP and to the INdAM group GNSAGA.

\vspace{0.5cm}

\section{{\Large Review of ctF-CohFTs}}

\medskip

\subsection{Definition and properties}

\label{F-CohFT, definition section}

For $2g - 1 + n > 0$, the Deligne--Mumford moduli space of stable curves of genus $g$ with marked points labelled $1,1+1,\ldots,1+n$ is denoted $\Mbar_{g,1+n}$. The moduli space $\mathcal{M}_{g,1+n}^{\textnormal{ct}} \subseteq \Mbar_{g,1+n}$ of stable curves of compact type is the locus of stable curves with compact Jacobian, or equivalently, stable curves in which all nodes are separating. The marked point labelled $1$ will often play a special role, stressed in the notation $1 + n$. Let
\begin{equation}
\label{fmor}
f : \Mct_{g,1+n+1} \longrightarrow \Mct_{g,1+n}.
\end{equation}
be the morphism forgetting the last marked point. Let
\begin{equation}
\label{glmor}
\textnormal{gl} : \Mct_{g_1,1 + n_1 + 1}\times\Mct_{g_2,1 + n_2} \longrightarrow \Mct_{g,1+n}
\end{equation}
with the implicit equalities $g = g_1 + g_2$ and $n = n_1 + n_2$ be the morphism gluing the last marked point of a stable curve in $\Mct_{g_1,1+n_1+1}$ with the first marked point of a stable curve in $\Mct_{g_2,1+n_2}$. The permutation group in $n$ elements is denoted $\mathfrak{S}_n$. An element $\sigma \in \mathfrak{S}_n$ acts as an automorphism of $\Mct_{g,1+n}$ still denoted $\sigma$ by permutation of the marked points labelled $1+1,\ldots,1+n$. In what follows we will mostly work $\Mct_{g,1+n}$, but the maps \eqref{fmor} and \eqref{glmor}, as well as the automorpshims $\sigma$ for $\sigma\in \mathfrak{S}_n$ have an obvious extension to the spaces $\Mbar_{g,1+n}$, which we denote in the same way.

\begin{definition}
    Let $V$ be a finite dimensional $\mathbb{C}$-vector space. A \emph{ctF-CohFT} is a collection
    \[
    \Omega_{g,1+n} \in \Hom\big(V^{\otimes n},V \otimes H^{\mathrm{even}}(\Mct_{g,1+n})\big)
    \]
    indexed by integers $g,n \geq 0$ such that $2g - 1 + n > 0$ and satisfying the following properties for any $v_1, \dots, v_n \in V$.
    \begin{itemize}
        \item It is $\mathfrak{S}$-equivariant:
        \begin{equation}
        \label{equivariance}
        \forall \sigma \in \mathfrak{S}_n\qquad \Omega_{g,1+n}(v_{\sigma(1)}\otimes\cdots\otimes v_{\sigma(n)})= \sigma^*\Omega_{g,1+n}( v_1\otimes\cdots\otimes v_n).
        \end{equation}
        \item It is compatible with any of the gluing morphisms \eqref{glmor}:
        \begin{equation}
        \label{compamor}
        \textnormal{gl}^*\Omega_{g,1+n}(v_1\otimes\dots\otimes v_n) =\ \Omega_{g_1,1+n_1+1}\big(v_1\otimes\cdots\otimes v_{n_1}\otimes \Omega_{g_2,1+n_2}(v_{n_1+1}\otimes\dots\otimes v_{n})\big).
        \end{equation}
    \end{itemize}
    A ctF-CohFT admits a \emph{flat unit} if there exists $\1 \in V$ such that, for any $g,n \geq 0$ and $v,v_1,\ldots,v_n$ we have
\[
\Omega_{0,1+2}(v \otimes \1) = v , \qquad f^*\Omega_{g,1+n}(v_1\otimes\cdots\otimes v_n)=\,\Omega_{g,1+n+1}(v_1\otimes\dots\otimes v_n\otimes\1).
\]
A F-TFT is a ctF-CohFT concentrated in cohomological degree $0$.
\end{definition}

\begin{definition}
A F-CohFT is a collection 
\[
    \Omega_{g,1+n} \in \Hom\big(V^{\otimes n},V \otimes H^{\mathrm{even}}(\Mbar_{g,1+n})\big)
    \]
indexed by integers $g,n \geq 0$ such that $2g - 1 + n > 0$ and satisfying axioms \eqref{equivariance} and \eqref{compamor} for the $\Mbar$-version of the corresponding morphisms, and similarly for the notion of admitting a flat unit.
\end{definition}

Notice that, while a F-CohFT gives rise to a ctF-CohFT by restriction, extending a ctF-CohFT to a F-CohFT is not always possible, both because not all cohomology classes are extendable from $\Mct_{g,1+n}$ to $\Mbar_{g,1+n}$, and because even in the presence of such extension, axioms \eqref{equivariance} and \eqref{compamor} could hold on compact type but not on the full compactification. In cohomological degree zero, however, this extension problem is trivial and the extension unique, which is why we do not distinguish between F-TFTs and ctF-TFTs.

As for CohFTs, one could replace the even cohomology groups appearing in the definition of (ct)F-CohFT with full cohomology groups at the cost of having to deal with $\mathbb{Z}_2$-graded objects and the corresponding Koszul signs. We avoid this for simplicity. 
The analogue of the well-known equivalence between $0$-th cohomology parts of CohFTs (\textit{i.e.} TFTs) and Frobenius algebras is as follows.

\begin{lemma}\cite{buryak_extended_2021}
The cohomological-degree $0$ part of a ctF-CohFT $\Omega$, denoted $\omega$, is a F-TFT uniquely determined by the commutative associative algebra structure on $V$ given by
\[
\forall v_1,v_2 \in V\qquad v_1\cdot v_2:= \omega_{0,1+2}(v_1\otimes v_2).
\]
together with the distinguished element $\alpha:= \omega_{1,1} \in V$. We used the identification $H^0(\Mbar_{g,1+n}) \cong \mathbb{C}$ to consider $\omega_{g,1+n}$ as an element of $\Hom(V^{\otimes n},V)$.  Conversely, any commutative associative algebra structure on $V$ together with a choice of distinguished element comes from a unique F-TFT.
\end{lemma}

At various stages we will make extra assumptions on ctF-CohFTs.

\begin{definition}
A ctF-CohFT is \emph{invertible} if $V$ is unital and $\alpha$ has an inverse for the product $\cdot$.
\end{definition}
 If $V$ is semi-simple, then $V$ is automatically unital. If a ctF-CohFT on $V$ has a flat unit, then $V$ is a unital algebra. But, if $V$ is unital, the unit may not satisfy the flat unit axiom.

The panorama of known examples of (ct)F-CohFTs that are not CohFTs is not currently as ample as the one of CohFTs. The most studied F-CohFTs come from modifications of CohFTs. For instance, the FJRW partial CohFTs of \cite{BCFG} arise from reductions of larger CohFTs and the extended $r$-spin classes of \cite{buryak_extended_2021} are limits of families of CohFTs. In \cite{buryak_rossi_classification_2025,buryak_xu_yang_2026} the relation of rank-$1$ F-CohFTs with integrable systems is explored, with classification purposes, in terms of the F-Givental group action recalled in Section~\ref{sec:Fgro}. Gromov-Witten theory with non-compact targets \cite{lu_GW_for_non_compact_targets_2006} and open Gromov-Witten theory (see \textit{e.g.} \cite{Solomonopen,Zongopen}) might provide a geometric source of further interesting examples.

\subsection{Stratification of moduli spaces}
\label{F-Givental group section}

Recall the stratification of $\Mbar_{g,1+n}$ by stable graphs of type $(g,1+n)$: vertices correspond to connected components of the normalisation of a stable curve and remember their respective genera, (unoriented) edges correspond to nodes and half-edges to images of the nodes in the normalisation, leaves correspond to marked points. The leaves are labelled from $1$ to $1 + n$ and the sum of genera at the vertices is $g$ minus the first Betti number of the graph.

The moduli space of compact type $\mathcal{M}_{g,1+n}^{\textnormal{ct}} \subseteq \Mbar_{g,1+n}$ is the union of strata corresponding to \emph{stable trees}, \textit{i.e.} stable graphs with first Betti number $0$. Stable trees can be canonically rooted at the leaf labelled $1$, and their edges receive a canonical orientation flowing from the leaves labelled $1+1,\ldots,1+n$ (ingoing, drawn at the top) towards the root leaf (outgoing, drawn at the bottom). If $v$ is a vertex in a stable tree, we denote $g(v)$ the genus it carries, and $n(v)$ its number of ingoing edges, \textit{i.e.} the valency of $v$ minus $1$. As is shown in Figure~\ref{fig:stabtreeaut}, stable trees can have non-trivial automorphisms\footnote{This corrects \cite[Remark 4.2]{arsie_semisimple_2023} or \cite[above Theorem  3.2]{borot_symmetries_2024}. Nevertheless, all arguments in these articles are valid once the adequate automorphism factors are added. In fact, the stable trees admitting non-trivial automorphisms do not contribute to the DR hierarchy associated to F-CohFTs \cite{ABLRint}, as subtrees without ingoing leaves come with a factor $\textnormal{DR}_g(\mathbf{0})\lambda_g = (-1)^{g}\lambda_g^2 = 0$ \cite{DRJanda}.}. Yet, this possibility is rather limited: automorphisms only originate from the permutations of isomorphic subtrees without ingoing leaves and which are attached to a common vertex.

\begin{figure}[h]
\begin{center}
\includegraphics[width=0.15\textwidth]{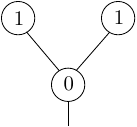}
\caption{\label{fig:stabtreeaut} A stable tree of genus $2$ with an automorphism group of order $2$.}
\end{center}
\end{figure}

For the definition of the F-Givental group we need to consider strata associated to stable trees.

\begin{definition}
\label{def:inclu}Let $T_{g,1+n}$ the set of stable trees of type $(g,1+n)$. If $\Gamma \in T_{g,1+n}$, we denote
\[
\Mct_{\Gamma} := \prod_{\textnormal{vertex}\,\,v} \Mct_{g(v),1+n(v)}.
\]
Replacing each vertex $v$ of $\Gamma$ with a stable curve of genus $g(v)$ with $1 + n(v)$ marked points, contracting edges of $\Gamma$ to nodes, and labelling the remaining marked points as they were labelled in $\Gamma$, we obtain a proper morphism $\textnormal{gl}_{\Gamma} : \Mct_{\Gamma} \longrightarrow \Mct_{g,1+n}$. We call $\mathcal{S}_{\Gamma} := \textnormal{gl}_{\Gamma}(\Mct_{\Gamma})$ the stratum associated to $\Gamma$. The number of automorphisms of $\Gamma$ is denoted $\# \textnormal{Aut}(\Gamma)$ and it coincides with the degree of $\textnormal{gl}_{\Gamma}$ as a map from $\Mct_{\Gamma}$ onto its image.
\end{definition}
Since for any stable tree $\Gamma$ the morphism $\textnormal{gl}_{\Gamma}$ is a composition of gluing morphisms like \eqref{glmor}, ctF-CohFTs are compatible with the restriction to $\Mct_{\Gamma}$. This means that for a ctF-CohFT $\Omega$, we can express $\textnormal{gl}_\Gamma^* \Omega_{g,1+n}$ by multiplying the classes $\Omega_{g(v),1+n(v)}$ associated to the vertices $v$ of $\Gamma$ and composing the multilinear maps along the tree. Since $\textnormal{gl}_{\Gamma}^*$ is an isomorphism in cohomological degree zero,  F-TFTs can be calculated in all topologies (use the stable trees of Figure~\ref{fig:FTFT}).

\begin{lemma}
\label{lem:FTFT}
Let $\omega$ be a F-TFT. Denoting $\alpha := \omega_{1,1}$, for any $g,n \geq 0$ such that $2g - 1 + n > 0$ and $v_1,\ldots,v_n \in V$ we have $\omega_{g,1+n}(v_1 \otimes \cdots \otimes v_n) = \alpha^{g} \cdot v_1 \cdots v_n$.
\end{lemma}

\begin{figure}[h!]
\begin{center}
\includegraphics[width=0.5\textwidth]{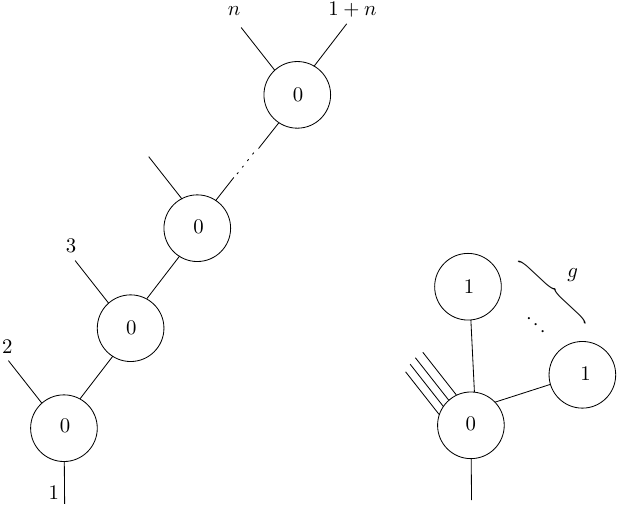}
\caption{\label{fig:FTFT}The evaluation of F-TFT follows from compatibility with the restriction to the stratum on the left ($g = 0$) and on the right ($g \geq 1$). }
\end{center}
\end{figure}

\subsection{F-Givental group}
\label{sec:Fgro}
We now review the F-Givental group and its action on ctF-CohFTs \cite{arsie_semisimple_2023}. The proofs are omitted, as they are completely analogue to the ones for the Givental group action on CohFTs that can be found  \textit{e.g.} in \cite{pandharipande_relations_2015}.

\begin{definition}
    \label{Taction}Consider a ctF-CohFT $\Omega$ on $V$. Take $T(z)\in z^2 V \llbracket z \rrbracket$. The \emph{translation of $\Omega$ by $T$} is the collection of classes $T\Omega$ defined by 
    \[
    (T\Omega)_{g,1+n}(v_1\otimes \cdots \otimes v_n)= \sum_{m\geq 0}\dfrac{1}{m!}(f_m)_*\,\Omega_{g,1+n+m}\big(v_1 \otimes \dots \otimes v_n \otimes T(\psi_{1+n+1}) \otimes \dots \otimes T(\psi_{1+n+m})\big),
    \]
    where $f_m: \Mct_{g,1+n+m} \rightarrow \Mct_{g,1+n}$ forgets the $m$ last marked points (due to the condition $T(z) = O(z^2)$, the sum over $m$ has only finitely many non-zero terms).
    \end{definition}

Equivalently, the translation can be formulated in terms of $\kappa$-classes.

\begin{lemma}
\label{expkappa}Assume that the algebra $(V,\cdot)$ has a unit $\1$. Let $T(z) = \sum_{m \geq 2} t_m z^m$ for $t_2,t_3,\ldots \in V$ and define
\[
\hat{T}(z) := \bigg(\1 - \frac{T(z)}{z}\bigg)^{-1} := \exp\bigg(\sum_{m \geq 1} \hat{t}_m z^m\bigg).
\]
Then, we have
\[
\exp\bigg(\sum_{m \geq 1} \hat{t}_m \kappa_m\bigg) = \sum_{m \geq 0} \frac{1}{m!} f_{m*}\bigg(\prod_{i = 1}^{m} T(\psi_{1+n+i})\bigg)
\]
\end{lemma}
\begin{proof}
See \textit{e.g.} in \cite[Proposition 6.13]{teleman_structure_2012},  or by a different method \cite[Lemma 6.3.2]{Eynardbook}.
\end{proof}

\begin{proposition}
    If $\Omega$ is a ctF-CohFT, so is $T\Omega$.  Translations form an abelian group \wrt the sum, and Definition~\ref{Taction} is a left group action on ctF-CohFTs.
\end{proposition}

Let $R(z)\in \End(V) \llbracket z \rrbracket$ be a group-like element, \textit{i.e.} $R(z)=\Id_V+ O(z)$. We define the associated edge weight
\begin{equation}
\label{Wweight}
E_R(z,z'):=\frac{\Id_V -R^{-1}(z)R(-z')}{z+z'}\in \End(V)\llbracket z,z' \rrbracket.
\end{equation}
There are two noteworthy differences with the definition of R-elements and edge weights in the Givental group. First, $R(z)$ do not need to satisfy a symplectic condition. Second, there are two distinct factors involving $R$ in the formula for edge weight: $R^{-1}(z)$ is associated to ingoing edges, while $R(-z)$ is associated to outgoing edges.

\begin{definition}
\label{Reaction}
    Consider a ctF-CohFT $\Omega$ on $V$ and let $R(z)$ be a group-like element of $\End(V)\llbracket z \rrbracket$. The R-transformation of $\Omega$ is the collection of classes $R\Omega$ defined by
 \[
 R\Omega_{g,1+n} = \sum_{\Gamma \in T_{g,1+n}} \frac{1}{\# \textnormal{Aut}(\Gamma)}\, \textnormal{gl}_{\Gamma *}\Omega_{\Gamma}
 \]
 where $\Omega_{\Gamma} \in \Hom\big(V^{\otimes n},V \otimes H^{\textnormal{even}}(\Mct_{\Gamma})\big)$ is obtained in the following way. We first place
 \begin{itemize}
 \item $R^{-1}(\psi_{1+i})$ at the $(1+i)$-th ingoing leaf, for each $i \in [n]$;
 \item $\Omega_{g(v),1+n(v)}$ at each vertex $v$;
 \item $E_R(\psi,\psi')$  at each oriented edge $v' \rightarrow v$, where $\psi,\psi'$ are the psi-classes associated to the image of the node in the components corresponding to $v,v'$;
 \item $R(-\psi_1)$ at the root.
 \end{itemize}
Then, we tensor classes in $H^{\textnormal{even}}(\Mct_{\Gamma}) \cong \bigotimes_{v} H^{*}(\Mct_{g(v),1+n(v)})$ and compose multilinear maps involving $V$ along edges of the stable tree, following the orientation.
\end{definition}

\begin{figure}[t]
\begin{center}
\includegraphics[width=0.65\linewidth]{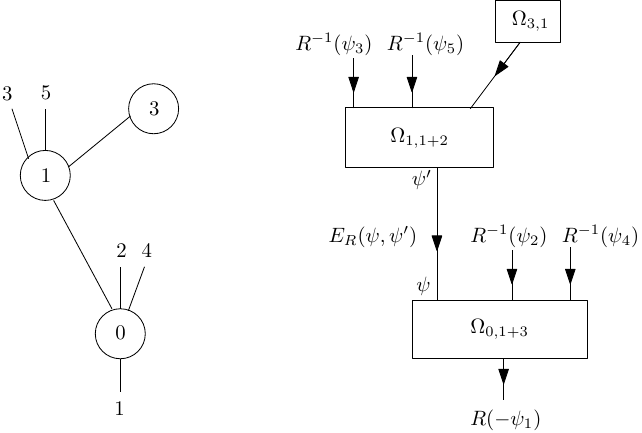}
\caption{A stable tree in $T_{4,1+4}$ and the corresponding $\textnormal{Cont}$.}
\end{center}
\end{figure}

\begin{proposition}
    If $\Omega$ is a ctF-CohFT, so is $R\Omega$. Group-like elements in $\textnormal{End}(V)\llbracket z \rrbracket$ form a group for the composition in $V$ and multiplication in $z$, and Definition~\ref{Reaction} is a left group action on ctF-CohFTs.
\end{proposition}

\begin{proposition}
\label{afterbeforeprop}
    Take $R(z)$ be a group-like element of $\End(V)\llbracket z \rrbracket$ and let\footnote{A stands for 'after', B for 'before'.} $T_{\textnormal{A}}(z), T_{\textnormal{B}}(z)$ be two elements of $z^2V\llbracket z \rrbracket$  related by $T_{\textnormal{A}}(z)=R(z)[T_{\textnormal{B}}(z)]$. Then, for every ctF-CohFT $\Omega$ we have
    \[
    T_{\textnormal{A}}R\Omega= RT_{\textnormal{B}}\Omega,
    \]
    where we mean applying first the $R$-action on $\Omega$ and then translation by $T_{\textnormal{A}}$ action, or applying first translation by $T_{\textnormal{B}}$ and then the $R$-action.
\end{proposition}

In other words, translations and $R$-transformations combine together in a semi-direct product of groups, which acts on ctF-CohFTs. We call this group the \emph{F-Givental group}. If $V = \mathbb{C}$ and we act on the trivial F-CohFT $\mathbf{1}$ given by the fundamental class in every $(g,n)$, we have the concise formula
\begin{equation}
\label{RT1}
(RT\mathbf{1})_{g,1+n} = \exp\bigg(\sum_{k \geq 1} \hat{t}_k \kappa_{k} + r(-\psi_1) - \sum_{i = 2}^{n+1} r(\psi_i) + \sum_{\Gamma}  \textnormal{gl}_{\Gamma *} \frac{r(\psi) - r(-\psi')}{\psi + \psi'}\bigg),
\end{equation}
where $R(z) = e^{r(z)}$ and the sum ranges over stable trees $\Gamma \in T_{g,1+n}$ with a single edge. This comes from the treatment of self-intersections of boundary divisors, see \textit{e.g.} \cite[Lemma 3.10]{MVpaper}.

\begin{proposition}
    Let $\Omega$ be a ctF-CohFT on $V$ with a flat unit $\1\in V$. Let $R(z)$ be a group-like element of $\End(V)\llbracket z \rrbracket$ and define
    \[
    T_{\textnormal{A}}(z)=z\big(R(z)[\1]-\1\big) \qquad \textnormal{and} \qquad  T_{\textnormal{B}}(z)=z\big(\1-R^{-1}(z)[\1]\big).
    \]
    Then $R.\Omega := T_{\textnormal{A}}R\Omega=RT_{\textnormal{B}}\Omega$ is a ctF-CohFT on $V$ with the same flat unit $\1$.
    \end{proposition}

\begin{remark} For CohFTs and TFTs, the compatibility property holds not only for stable trees but also for stable graphs. In other words, CohFTs are algebras over the modular operad $H^*(\Mbar_{g,n})$, while ctF-CohFTs are algebras over the graded \emph{operad} $\bigoplus_{g \geq 0} H^*(\Mct_{g,1+n})$. Restricting the latter to the genus $0$ part gives the so-called hypercommutative algebras, and the $R$-action on them first appeared (at the infinitesimal level) in \cite[Section 6]{SM}.
\end{remark}

\vspace{0.5cm}

\section{{\Large Geometry of invertible ctF-CohFTs}}

\medskip
\label{S3}
\label{topological field theory}

\subsection{Variants of ctF-CohFTs related to other moduli spaces} \label{sec:variantFCohFT} For the proof of Theorem~\ref{thm:transfree} it is crucial to work not only with or $\Mct_{g,1+n}$, but also with the moduli of smooth curves $\M_{g,1+n}$ and certain bundles over it with a more differential (rather than algebraic) geometric perspective. Here some arbitrary choices have to be made, which we fix using hyperbolic structures. 

\begin{definition}
\label{defpi}For $2g - 1 + n > 0$, let $\mathcal{M}_{g,1+n}^{\circ}$ be the moduli space of hyperbolic structures on a smooth real surface of genus $g$ with $n$ unit-length geodesic boundaries labelled $1,\ldots,1+n$. Let
\[
\pi : \mathcal{M}^{\bullet}_{g,1+n} \longrightarrow \mathcal{M}_{g,1+n}^{\circ}
\]
be the $(\mathbb{S}_1)^{1 + n}$-bundle whose fibers parametrise the choice of an origin point on each boundary.
\end{definition}
It is well-known that in the smooth category, $\mathcal{M}_{g,1+n}^{\circ}$ is isomorphic to the moduli space $\mathcal{M}_{g,1+n}$ parametrising smooth complex curves with marked points and the bundle $\mathcal{M}_{g,1+n}^{\bullet}$ is isomorphic to the bundle over $\mathcal{M}_{g,1+n}$ whose fibers parametrise tangent vectors at each marked point modulo rescaling by a positive real number. We denote
\begin{equation}
\label{thetaiso}
\theta : \mathcal{M}_{g,1+n}^{\circ} \longrightarrow \mathcal{M}_{g,1+n}
\end{equation}
the isomorphism equipping the bordered hyperbolic surfaces with their associated complex structure, and gluing along each boundary a punctured complex disc to get a smooth complex curve with marked points. We define $\psi$- and $\kappa$-classes on $\M_{g,1+n}^{\circ}$ by pulling back the corresponding classes on $\M_{g,1+n}$.

We can construct the analogue of the gluing morphism \eqref{glmor} for these two new moduli spaces. Observe each element of $\mathcal{M}_{g,1+n}^{\bullet}$ is represented by real surfaces $\Sigma$ in which each boundary component is canonically identified to $\mathbb{S}_1$, using the hyperbolic length of paths along the boundary and issuing from its origin point. For each integer decomposition $g = g_1 + g_2$ and $n = n_1 + n_2$, we have a smooth gluing map
\begin{equation}
\label{glmorbull}
\textnormal{gl}^{\bullet} : \mathcal{M}_{g_1,1 + n_1 + 1}^{\bullet} \times \mathcal{M}_{g_2,1+n_2}^{\bullet} \longrightarrow \mathcal{M}_{g,1+n}^{\bullet}.
\end{equation}
It is obtained by gluing the last boundary component of a hyperbolic surface in $\mathcal{M}_{g_1,1 + n_1 + 1}^{\bullet}$ with the first boundary component of a hyperbolic surface in $\mathcal{M}_{g_2,1+n_2}^{\bullet}$ matching their common canonical identification to $\mathbb{S}_1$, and forgetting the origin point. The result is a hyperbolic surface because we glued geodesic boundaries of hyperbolic surfaces. We define $\kappa$-classes on $\M_{g,1+n}^{\bullet}$ by pulling back with $\pi$ the corresponding classes on $\M_{g,1+n}^{\circ}$. Doing the same with $\psi$-classes yields zero, as we discuss in Section~\ref{sec:free}.

We introduce the locus $\mathcal{N} \subset \mathcal{M}_{g,1+n}^{\circ}$ of hyperbolic surfaces admitting a geodesic of length $\leq 1$ separating it into two components of genus $g_1$ and $g_2$, the first one containing the boundaries labelled $1,\ldots,1 + n_1$. The locus $\partial \mathcal{N} \subset \mathcal{M}_{g,1+n}^{\circ}$ consists of surfaces where the same splitting happens with a geodesic of length exactly $1$. Cutting along this geodesic defines a map (Figure~\ref{fig:numap})
\begin{equation}
\label{numap}
\nu : \partial \mathcal{N} \longrightarrow \mathcal{M}_{g_1,1+n_1 + 1}^{\circ} \times \mathcal{M}_{g_2,1+n_2}^{\circ}.
\end{equation}
This is a $\mathbb{S}_1$-bundle, as the twist was forgotten in the cutting. The cohomology of this bundle will be discussed in Section~\ref{sec:free}.

\begin{figure}[h!]
\begin{center}
\includegraphics[width=0.5\textwidth]{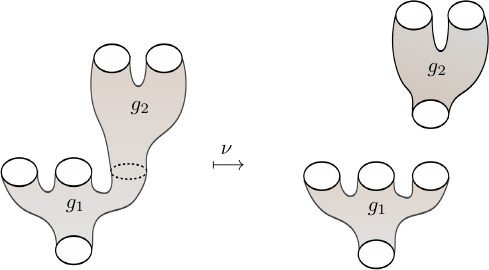}
\caption{\label{fig:numap} The map $\nu$ of \eqref{numap}, where $g_1,g_2$ indicate the genus of each component.}
\end{center}
\end{figure}

\begin{figure}[h!]
\begin{center}
\includegraphics[width=0.3\textwidth]{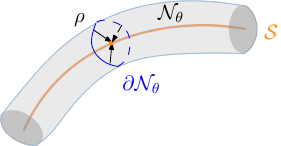}
\caption{Tubular neighborhood and circle bundle \eqref{eq:rhomap}. \label{fig:rhomap}}
\end{center}
\end{figure}

Back to the moduli spaces of complex curves, let $\mathcal{S} \subset \M_{g,1+n}^{\textnormal{ct}}$ be the stratum
\begin{equation}
\label{Sopenst}
\mathcal{S} := \textnormal{gl}(\M_{g_1,1+n_1+1} \times \M_{g_2,1+n_2}).
\end{equation}
The moduli spaces of bordered surfaces allow us defining a thickening $\mathcal{N}_{\theta} := \mathcal{S} \cup \theta(\mathcal{N}) \subset \M_{g,1+n}^{\textnormal{ct}}$: this is a tubular neighborhood of $\mathcal{S}$ admitting a smooth strong deformation retraction $r : \mathcal{N}_{\theta} \rightarrow \mathcal{S}$. The restriction of $r$ to $\partial \mathcal{N}_{\theta} = \theta(\partial \mathcal{N})$ is a $\mathbb{S}_1$-bundle
\begin{equation}
\label{eq:rhomap}
\rho : \partial \mathcal{N}_{\theta} \longrightarrow \mathcal{S}.
\end{equation}
By their geometric construction, the bundles $\rho$ and $\nu$ are related by the commutative diagram
\begin{equation}
\label{comu}
\rho \circ \theta = \textnormal{gl} \circ (\theta_1^{-1},\theta_2^{-1}) \circ \nu,
\end{equation}
where $\theta_1$ and $\theta_2$ are isomorphisms like \eqref{thetaiso} for the two factors. The thickening will be used for cohomological computations in the following way.
\begin{lemma}
\label{lemaphi} If $\phi \in H^*(\M_{g,1+n}^{\textnormal{ct}})$, then $\rho^*(\phi_{|\mathcal{S}}) = \phi_{|\partial \mathcal{N}_{\theta}}$.
\end{lemma}
\begin{proof} Since $r$ induces an isomorphism in cohomology, we have $r^*(\phi_{|\mathcal{S}}) = \phi_{|\mathcal{N}_{\theta}}$. Restricting to $\partial \mathcal{N}_{\theta}$ gives $\rho^*\phi_{|\mathcal{S}}$ in the left-hand side and $\phi_{|\partial \mathcal{N}_{\theta}}$ in the right-hand side.
\end{proof}

\begin{definition}
A \emph{free-boundary F-CohFT} on $V$ is a collection
\[
\Omega_{g,1+n} \in \textnormal{Hom}\big(V^{\otimes n},V \otimes H^{\textnormal{even}}(\mathcal{M}_{g,1+n}^{\circ})\big)
\]
indexed by integers $g,n \geq 0$ such that $2g - 1 + n \geq 0$, which is $\mathfrak{S}$-equivariant and is compatible with any of the maps \eqref{numap},  \textit{i.e.} for any $v_1,\ldots,v_n \in V$
\begin{equation}
\label{compascirc}
\Omega_{g,1+n}(v_1 \otimes \cdots \otimes v_n)_{|\partial N} = \nu^*\Omega_{g_1,1+n_1 + 1}\big(v_1 \otimes \cdots \otimes v_{n_1} \otimes \Omega_{g_2,1+n_2}(v_{n_1 + 1} \otimes \cdots \otimes v_n)\big).
\end{equation}
A \emph{pinned-boundary F-CohFT} on $V$ is a collection
\[
\Omega_{g,1+n} \in \textnormal{Hom}\big(V^{\otimes n},V \otimes H^{\textnormal{even}}(\mathcal{M}_{g,1+n}^{\bullet})\big)
\]
indexed by integers $g,n \geq 0$ such that $2g - 1 + n \geq 0$, which is $\mathfrak{S}$-equivariant and is compatible with the map \eqref{glmorbull}, \textit{i.e.} for any $v_1,\ldots,v_n \in V$:
\begin{equation}
\label{compasbull}
(\textnormal{gl}^{\bullet})^* \Omega_{g,1+n}(v_1 \otimes \cdots \otimes v_n) = \Omega_{g_1,1+n_1 + 1}\big(v_1 \otimes \cdots \otimes v_{n_1} \otimes \Omega_{g_2,1+n_2}(v_{n_1 + 1} \otimes \cdots \otimes v_{n})\big).
\end{equation}
\end{definition}

\begin{proposition}
\label{propF1} If $\Omega$ is a ctF-CohFT on $V$, then $\Omega^{\circ} := \theta^*(\Omega_{|\M})$ is a free-boundary F-CohFT and $\Omega^{\bullet} := \pi^* \Omega^{\circ}$ is a pinned-boundary F-CohFT\footnote{$\Omega^{\circ}$ and $\Omega^{\bullet}$ are called $\Z$ with lower and upper indices respectively in \cite{teleman_structure_2012}.} on $V$. 
\end{proposition}
\begin{proof}
Let $\Omega$ be a ctF-CohFT. Take $g_i,n_i \geq 0$ for $i = 1,2$ such that $2g_i - 1 + n_i > 0$, set $g = g_1 + g_2$ and $n = n_1 + n_2$, and take $w^{(1)} \in V^{\otimes n_1}$ and $w^{(2)} \in V^{\otimes n_2}$. Introduce $\Omega_{g,1+n}^{\circ}  := \theta^*(\Omega_{g,1+n}{}_{|\M_{g,1+n}})$. We examine its compatibility with any of the maps \eqref{numap}:
\begin{equation*}
\begin{split}
\Omega^{\circ}_{g,1+n}(w^{(1)} \otimes w^{(2)})_{|\partial N} & = \theta^* \Omega_{g,1+n}(w^{(1)} \otimes w^{(2)})_{|\partial N_{\theta}} \\
& = (\rho \circ \theta)^* \Omega_{g,1+n}(w^{(1)} \otimes w^{(2)})_{|\mathcal{S}} \\
& = (\rho \circ \theta)^* \textnormal{gl}_{*} \Omega_{g_1,1+n_1 + 1}\big(w^{(1)} \otimes \Omega_{g_2,1+n_2}(w^{(2)})_{|\M_{g_2, 1+n_2}}\big)_{|\M_{g_1, 1+n_1+1}}.
\end{split}
\end{equation*}
In the second line we used Lemma~\ref{lemaphi} and in the third line the compatibility property of the ctF-CohFT $\Omega$. The commutative diagram \eqref{comu} then yields
\begin{equation*}
\begin{split}
\Omega^{\circ}_{g,1+n}(w^{(1)} \otimes w^{(2)})_{|\partial N} & = \nu^* \theta_{1}^*\Omega_{g_1,1+n_1 + 1}\big(w^{(1)} \otimes \theta_{2}^*\Omega_{g_2,1+n_2}(w^{(2)})\big) \\
& = \nu^*\Omega_{g_1,1+n_1 + 1}^{\circ}\big(w^{(1)} \otimes \Omega_{g_2,1+n_2}^{\circ}(w^{(2)})\big).
\end{split}
\end{equation*}
This proves that $\Omega^{\circ}$ is a free-boundary F-CohFT.

Now let $\Omega^{\bullet}_{g,1+n} := \pi^*\Omega^{\circ}_{g,1+n}$. We examine its compatibility with the gluing map \eqref{glmorbull}. Notice that the image of $\pi \circ \textnormal{gl}^{\bullet}$ consists of hyperbolic surfaces obtained by gluing along a geodesic boundary of length $1$, therefore is included in $\partial \mathcal{N}$. Then, recalling the compatibility properties of $\Omega^{\circ}$, we compute
\begin{equation*}
\begin{split}
(\textnormal{gl}^{\bullet})^* \Omega_{g,1+n}^{\bullet}(w^{(1)} \otimes w^{(2)}) & = (\pi \circ \textnormal{gl}^{\bullet})^* \Omega^{\circ}_{g,1+n}(w^{(1)} \otimes w^{(2)}) \\
& = (\pi \circ \textnormal{gl}^{\bullet})^* \Omega_{g,1+n}^{\circ}(w^{(1)} \otimes w^{(2)})_{|\partial \mathcal{N}} \\
& = (\pi \circ \textnormal{gl}^{\bullet})^* \nu^* \Omega^{\circ}_{g_1,1+n_1+1}\big(w^{(1)} \otimes \Omega^{\circ}_{g_2,1+n_2}(w^{(2)})\big).
\end{split}
\end{equation*}
By their geometric construction we have $\nu \circ \pi \circ \textnormal{gl}^{\bullet} = (\pi_1,\pi_2)$ where $\pi_i$ are the bundle projections of Definition~\ref{defpi} on each of the two factors. Thus
\begin{equation*}
\begin{split}
(\textnormal{gl}^{\bullet})^* \Omega_{g,1+n}^{\bullet}(w^{(1)} \otimes w^{(2)}) & = \pi_1^*\Omega^{\circ}_{g_1,1+n_1+1}\big(w^{(1)} \otimes \pi_2^*\Omega^{\circ}_{g_2,1+n_2}(w^{(2)})\big) \\
 & = \Omega^{\bullet}_{g_1,1+n_1+1}\big(w^{(1)} \otimes \Omega^{\bullet}_{g_2,1+n_2}(w^{(2)})\big).
\end{split}
\end{equation*}
This proves that $\Omega^{\bullet}$ is a pinned-boundary F-CohFT.
\end{proof}

In some sense $\Omega^{\bullet}$ and $\Omega^{\circ}$ of Proposition~\ref{propF1} only capture information about the ctF-CohFT $\Omega$ close to the boundary divisors in $\M^{\textnormal{ct}}$. Following Teleman's strategy, one first seeks to reconstruct them from the F-TFT (Sections~\ref{sec:pinned}-\ref{sec:free}) and in a second step, one tries to extend this reconstruction to the whole $\M^{\textnormal{ct}}$ (Sections~\ref{sec:patch}-\ref{sec:proofthmA}).

\subsection{Stability theorems and cohomological results}
\label{sec:stab}
We review the structural properties of the cohomology of moduli spaces of curves which play a crucial role in  \cite{teleman_structure_2012} and for us. Pick $P \in \M_{0,1+2}^{\bullet}$ and $\Yu \in \M_{1,1+1}^{\bullet}$ and introduce the maps (Figures~\ref{fig:gammamap} and \ref{fig:phimap})
\begin{equation}
\label{varphieq}\begin{split}
\gamma & = \textnormal{gl}^{\bullet}(P,-) : \M_{g,1+n}^{\bullet} \longrightarrow \M_{g,1+n+1}^{\bullet}, \\
\varphi & = \textnormal{gl}^{\bullet}(\Yu,-) : \M_{g,1+n}^{\bullet} \longrightarrow \M_{g+1,1+n}^{\bullet},
\end{split}
\end{equation}
which increase the number of boundaries or the genus by one, respectively. Since the moduli spaces are path-connected the map they induce in (co)homology do not depend on the choices made for the surfaces $P$ or $\Yu$.

\begin{figure}[h!]
\begin{center}
\includegraphics[width=0.55\textwidth]{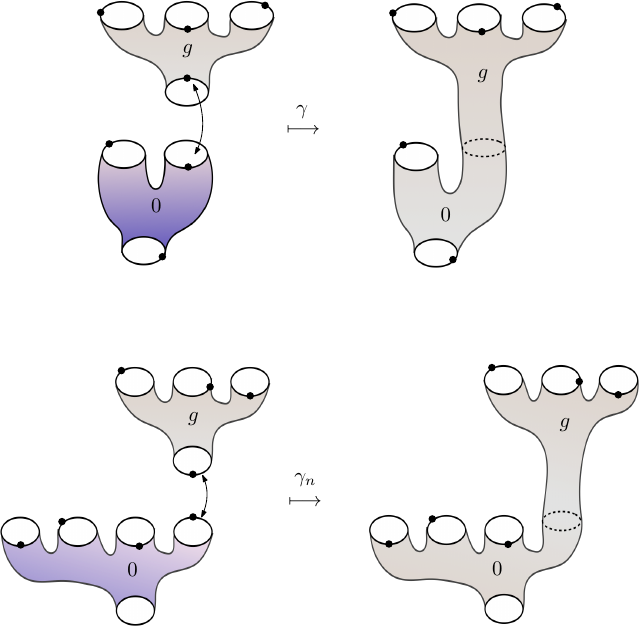}
\caption{Gluing fixed spheres with $1 + n \geq 3$ boundaries (in purple) \label{fig:gammamap}}
\end{center}
\end{figure}

\medskip
\medskip

\begin{figure}[h!]
\begin{center}
\includegraphics[width=\textwidth]{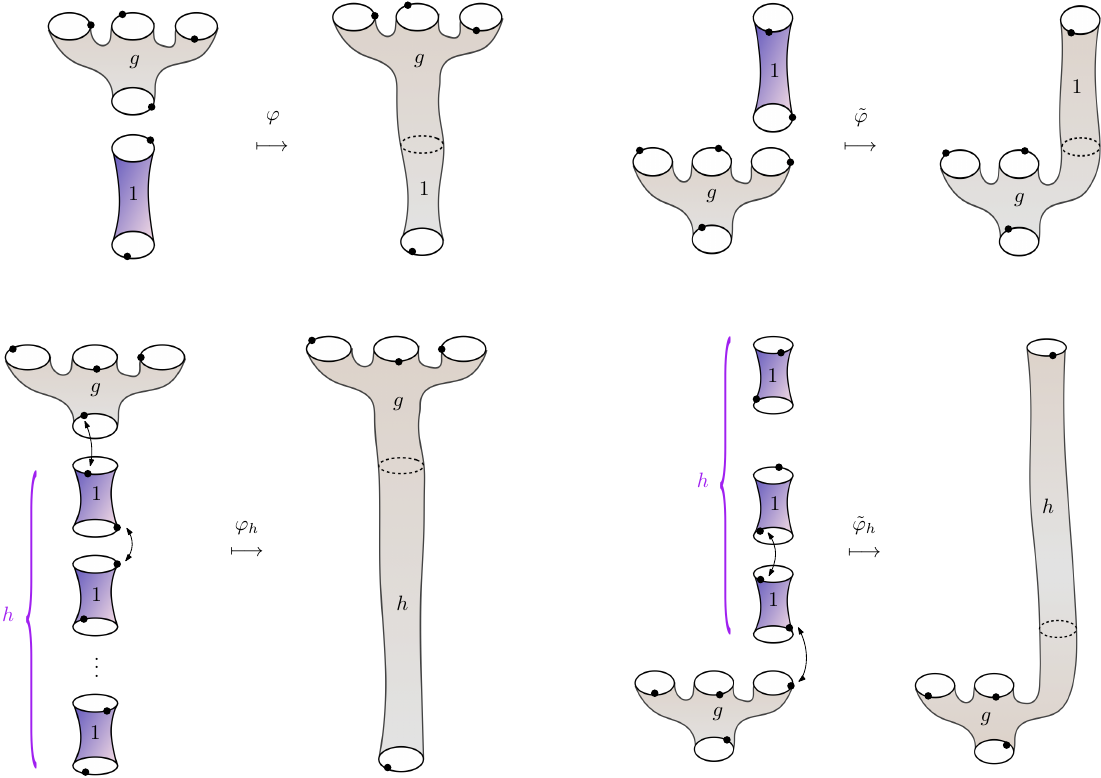}
\caption{Gluing fixed tori with two boundaries (in purple) \label{fig:phimap}}
\end{center}
\end{figure}

\begin{theorem} \cite{harer_stability_1985,ivanov_subgroups_1992} \label{thm:Harerstab}
Let $g,n \geq 0$ such that $2g - 1 + n > 0$. The map $\gamma^*$ in cohomology degree $\leq \frac{2g}{3}$  and the map $\varphi^*$ in cohomology degree $\leq \frac{2}{3}(g - 1)$ are isomorphisms.
\end{theorem}

From this result, Looijenga \cite[Proposition 2.1]{looijenga_stable_1994}  deduced that $H^k(\M^{\bullet}_{g,1+n}) \simeq H^k(\M_g)$ and $H^k(\M_{g,1+n}) \simeq H^k(\M_g)[\psi_1,\ldots,\psi_{1+n}]$ for $g,n \geq 0$ and degree in the stability range $0 \leq k \leq \frac{2g}{3}$.  Together with Mumford's conjecture proved by Madsen and Weiss \cite{MW}, this provides a concise description of the stable cohomology of the moduli spaces.
\begin{theorem}
\label{Mumfordwithpsiclasses}
   For $g,n \geq 0$ such that $2g - 1 + n > 0$ we have isomorphisms in degree $\leq \frac{2g}{3}$ :
   \[
   H^*(\M^\bullet_{g,1+n})\simeq \mathbb{C}[\kappa_1, \kappa_2, \ldots],\qquad H^*(\M_{g,1+n})\simeq \mathbb{C}[\psi_1,\ldots, \psi_{1+n},\kappa_1, \kappa_2, \ldots].
   \]
\end{theorem}

For an overview of stability results, see \cite{laptev_stable_2005,Wahl}. In the context of free- or pinned-boundary (F-)CohFTs, these results allow gluing surfaces of high genus to reach the stability range where we better control what happens in cohomology. This is the first main idea of Teleman \cite{teleman_structure_2012}.

The stable cohomology $H^*(\M_{\infty,1+n}^{\bullet})$ can be defined as an inverse limit of the cohomology rings $H^*(\M^{\bullet}_{g,1+n})$ using the system of morphisms $\varphi$, and Theorem~\ref{Mumfordwithpsiclasses} identifies it with the free ring with generators $\kappa_m$ in degree $2m$ for each $m \geq 1$. By definition of inverse limits, every polynomial in these generators admits a restriction for every finite $g$ to an element of $H^*(\M_{g,1+n}^{\bullet})$, where the generators are interpreted as the actual $\kappa$-classes on $\M_{g,1+n}^{\bullet}$ (\textit{i.e.} the pullback via $\pi$ of the kappa classes on $\M_{g,1+n}^{\circ}$), which now have relations. By the same reason and nilpotency of $H^*(\M_{g,1+n}^{\bullet})$, formal series in the generators (elements of the completion of the stable cohomology) also restrict to elements in $H^*(\M_{g,1+n})$. While we keep the same notation for elements in the stable cohomology ring and for their finite-genus restriction, the context makes clear in which ring we are working. A similar remark applies to the stable cohomology $H^*(\M_{\infty,1+n})$.

Finally, we record elementary properties (see \textit{e.g.} \cite{AC94}) of $\psi$ and $\kappa$ classes with respect to the gluing morphism $\textnormal{gl} : \Mct_{g_1,1+n_1 + 1} \times \Mct_{g_2,1+n_2} \rightarrow \Mct_{g,n}$ and the forgetful morphism $f: \Mct_{g,1+n+1}\longrightarrow \Mct_{g,1+n}$. The latter has sections $p_i : \Mct_{g,1+n} \rightarrow \Mct_{g,1+n+1}$ indexed by $i \in [n]$ and following the $i^{\text{th}}$ marked point --- this relies on the canonical identification of the universal curve $\overline{\mathcal{C}}_{g,1+n}$ with $\Mct_{g,1+n+1}$. 
 
\begin{lemma}
 \label{lem39}  Let $i \in [n]$, $k \geq 0$ and $m \geq 1$. If we denote with tilde the classes on $\Mct_{g,n+1}$ and without tilde those on $\Mct_{g,n}$, we have
\begin{equation}
\label{forgetfulpsiteleman}
  \tilde{\psi}_i^k = f^* \psi_i^k + p_{i*} \psi_i^{k - 1}, \qquad \tilde{\kappa}_m = f^*\kappa_m + \tilde{\psi}^m_{n + 1}.
  \end{equation}
  If we denote with $\kappa^{(i)}$ the $\kappa$-classes associated to the $i$-th factor in $\Mct_{g,1+n_1 + 1} \times \Mct_{g,1+n_2}$, we have
  \begin{equation}
  \label{rpsi}
   \textnormal{gl}^*\kappa_m = \kappa_m^{(1)} + \kappa_m^{(2)}.
\end{equation}
\end{lemma}
If we restrict to $\M_{g,1+n} \cong \M_{g,1+n}^{\circ}$ these relation holds without the second term in the first equality of \eqref{forgetfulpsiteleman}. If we pullback to $\M_{g,1+n}^{\bullet}$ they hold with all $\psi$-classes set to zero.

\subsection{Calculating pinned-boundary F-CohFTs}
\label{sec:pinned}
The large genus analysis of pinned-boundary F-CohFTs will allow us determining them completely from the underlying F-TFT and some elements in the stable cohomology. The results in this section are valid for any pinned-boundary F-CohFT, but for logical clarity we formulate them for the particular one $\Omega^{\bullet}$ associated to a given ctF-CohFT $\Omega$ by Proposition~\ref{propF1}.

\begin{proposition}
\label{prop:Zkappa} Let $\Omega$ be an invertible ctF-CohFT, $\omega$ the associated F-TFT and $\alpha = \omega_{1,1}$. Then, there exists an element $\hat{T}(\boldsymbol{\kappa}) \in V \otimes \mathbb{C}\llbracket \kappa_1,\kappa_2,\ldots \rrbracket$ defined in the completion of the stable cohomology by
\[
\hat{T}(\boldsymbol{\kappa}) = \lim_{h \rightarrow \infty} \alpha^{-h} \cdot \Omega_{h,1}^{\bullet}.
\]
This element is of the form $\hat{T}(\boldsymbol{\kappa}) = \exp(\sum_{m \geq 1} \hat{t}_m \kappa_m)$ for some $\hat{t}_1,\hat{t}_2,\ldots \in V$. Furthermore, for any $g,n \geq 0$ such that $2g - 1 + n > 0$ and $w \in V^{\otimes n}$ we have
\begin{equation}
\label{omen0}
\Omega_{g,1+n}^{\bullet}(w) = \hat{T}(\boldsymbol{\kappa}) \cdot \omega_{g,1+n}(w).
\end{equation}
\end{proposition}

If we compare with Lemma~\ref{expkappa}, the multiplication of $\omega_{g,1+n}(w)$ by $\hat{T}(\boldsymbol{\kappa})$ coincides with the action of a translation on the F-TFT, leading us to the following definition.

\begin{definition}
\label{defT}
If $\Omega$ is an invertible ctF-CohFT, we call $T(z) := z\big(\1 - \exp(-\sum_{m \geq 1} \hat{t}_m z^m)\big)$, that we see as an element of the F-Givental group.
\end{definition}

Before starting the proof we examine the behavior of $\Omega^{\bullet}$ under the genus-increasing map
\[
\varphi_h  : \M_{g,1+n}^{\bullet} \longrightarrow \M_{g+h,1+n}^{\bullet}
\]
consisting in gluing $h$ times a copy of $\Yu$ to the first boundary, \textit{i.e.} by iterating maps $\varphi = \textnormal{gl}^{\bullet}(\Yu,-)$ like in \eqref{varphieq}.

\begin{lemma}
\label{ghnlem} For any $g,h \geq 0$ we have $\varphi_h^* \Omega_{g+h,1}^{\bullet} = \alpha^{h} \cdot \Omega_{g,1}^{\bullet}$.
\end{lemma}
\begin{proof}
By induction the statement reduces to checking $h = 1$. Compatibility of $\Omega^{\bullet}$ with the gluing map $\textnormal{gl}^{\bullet} : \M_{1,1+1}^{\bullet} \times \M_{g,1}^{\bullet} \rightarrow \M_{g+1,1}^{\bullet}$ yields $(\textnormal{gl}^{\bullet})^*\Omega^{\bullet}_{g+1,1} = \Omega^{\bullet}_{1,1+1}[\Omega^{\bullet}_{g,1}]$. Restricting to the locus $\{\Yu\} \times \M_{g,1}^{\bullet}$ replaces $\Omega^{\bullet}_{1,1+1}$ with $\omega_{1,1+1}$, which is the multiplication by $\alpha$.
\end{proof}

\begin{proof}[Proof of Proposition~\ref{prop:Zkappa}]
Taking $n = 0$ and $g \geq 1$ in Lemma~\ref{ghnlem} shows that $\varphi_h^*(\alpha^{-(g + h)} \cdot \Omega_{g+h,1}^{\bullet}) = \alpha^{-g} \cdot \Omega_{g,1}^{\bullet}$ is independent of $h$. Therefore, as $h \rightarrow \infty$ the sequence $\alpha^{-h} \cdot \Omega_{h,1}^{\bullet}$ admits a limit in the completion of $V \otimes H^*(\M_{\infty,1}^{\bullet})$, that is $V \otimes \mathbb{C}\llbracket \kappa_1,\kappa_2,\cdots\rrbracket$. Denoting $\hat{T}(\boldsymbol{\kappa})$ this limit, we have
\begin{equation}
\label{Omegabul11}\Omega_{g,1}^{\bullet} = \alpha^g \cdot \hat{T}(\boldsymbol{\kappa}),
\end{equation}
Since $\omega_{g,1} = \alpha^g$, this is the $n = 0$ case of \eqref{omen0}.

Let $g_1,g_2 \geq 1$ and set $g = g_1 + g_2$. Choose $P \in \M_{0,1+2}^{\bullet}$ and define the map $\mu : \M_{g_1,1}^{\bullet} \times \M_{g_2,1}^{\bullet} \longrightarrow \M_{g,1}^{\bullet}$ which glues the two last boundaries of $P$ to two surfaces with a single boundary and respective genus $g_1$ and $g_2$ (Figure~\ref{fig:mumap}). The compatibility of the pinned-boundary F-CohFT $\Omega^{\bullet}$ with $\mu$  yields
\[
\mu^* \Omega_{g,1}^{\bullet} = \Omega_{g_1,1}^{\bullet} \cdot \Omega_{g_2,1}^{\bullet}.
\]
Combining with \eqref{Omegabul11} we get $\mu^* \hat{T}(\boldsymbol{\kappa}) = \hat{T}(\boldsymbol{\kappa}^{(1)}) \cdot \hat{T}(\boldsymbol{\kappa}^{(2)})$ as an equality of classes on $\M_{g_1,1}^{\bullet} \times \M_{g_2,1}^{\bullet}$. Since this is valid for all $g_1,g_2 \geq 1$ and $\mu^* \boldsymbol{\kappa} = \boldsymbol{\kappa}^{(1)} + \boldsymbol{\kappa}^{(2)}$, we can send $g_1,g_2$ to infinity and get
\begin{equation}
\label{addZ}
\hat{T}(\boldsymbol{\kappa}^{(1)} + \boldsymbol{\kappa}^{(2)}) = \hat{T}(\boldsymbol{\kappa}^{(1)}) \cdot \hat{T}(\boldsymbol{\kappa}^{(2)}),
\end{equation}
The restriction of \eqref{Omegabul11} to cohomological degree $0$ forces $\hat{T}(\boldsymbol{0})$ to be the unit for the product $\cdot$, and then \eqref{addZ} implies the exponential form of $\hat{T}(\boldsymbol{\kappa})$.
 
 \begin{figure}[h!]
\begin{center}
\includegraphics[width=0.5\textwidth]{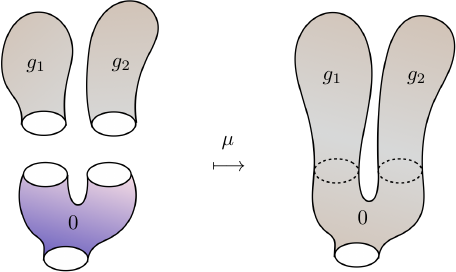}
\caption{Gluing two surfaces to a fixed pair of pants (in purple) \label{fig:mumap}}
\end{center}
\end{figure}

Now take $g,n,h \geq 1$ and consider the more general map $\mu : \M_{h,1}^{\bullet} \times \M_{g,1+n}^{\bullet} \longrightarrow \M_{g+h,1+n}^{\bullet}$ which glues the second boundary of $P$ to a surface of genus $h$ with a single boundary, and the third boundary of $P$ to the first boundary of a surface of genus $g$ with $1 + n$ boundaries. The compatibility of $\Omega^{\bullet}$ with this map yields
\[
\mu^* \Omega_{g+h,1+n}^{\bullet} = \Omega_{h,1}^{\bullet} \cdot \Omega_{g,1+n}^{\bullet}.
\]
Since the degree $0$ part of $\Omega_{h,1}^{\bullet}$ is $\omega_{h,1} = \alpha^h$ hence invertible, $\Omega_{h,1}^{\bullet}$ is also invertible and we have
\begin{equation}
\label{tounders}
\Omega_{g,1+n}^{\bullet} = (\Omega_{h,1}^{\bullet})^{-1} \cdot \mu^* \Omega_{g+h,1+n}^{\bullet}.
\end{equation}
We want to the relate the right-hand side to $\hat{T}(\boldsymbol{\kappa})$ by choosing the added genus $h$ sufficiently large to apply the stability theorems of Section~\ref{sec:stab}. By \eqref{Omegabul11} we already know that for any $h$ the first factor is
\begin{equation}
\label{Omegunag}
(\Omega_{h,1}^{\bullet})^{-1} = \alpha^{-h} \cdot \hat{T}^{-1}(\boldsymbol{\kappa}^{(1)}),
\end{equation} 
where the exponent ${}^{(1)}$ insists that the $\kappa$-classes are those attached to the first factor space $\M_{h,1}^{\bullet}$. To understand the second factor in \eqref{tounders}, we first choose $\Sigma \in \M_{0,1+n+1}^{\bullet}$ and use the compatibility of $\Omega^{\bullet}$ with the gluing map $\gamma_n := \textnormal{gl}^{\bullet}(\Sigma,-) : \M_{g+h,1}^{\bullet} \longrightarrow \M_{g+h,1+n}^{\bullet}$. Like in the proof of Lemma~\ref{ghnlem}, this gives for any $w \in V^{\otimes (n+1)}$
\begin{equation}
\label{gammanuns}
\gamma_n^*\Omega^{\bullet}_{g+h,1+n}(w) = \omega_{0,1+n + 1}(w \otimes \Omega^{\bullet}_{g+h,1}) = \Omega^{\bullet}_{g+h,1} \cdot \omega_{0,1+n}(w) = \alpha^{g+h} \cdot \hat{T}(\boldsymbol{\kappa}) \cdot \omega_{0,1+n}(w).
\end{equation}
Since the left-hand side of \eqref{tounders} can only contain cohomology classes of degree $\leq 6g - 4 + 2n$, we only need to understand $\Omega_{g+h,1+n}^{\bullet}$ for this degree range. For $h$ large enough this is in the stability range for Theorem~\ref{thm:Harerstab}, guaranteeing that $\gamma_n^*$ is an isomorphism. Together with \eqref{Omegunag} and \eqref{gammanuns}, the multiplicativity property \eqref{addZ} of $\hat{T}$ and the fact that pulling back by $\mu$ or $\gamma_n$ preserve $\kappa$-classes, this implies
\[
\Omega_{g,1+n}^{\bullet}(w) = \alpha^{g} \cdot \hat{T}(\boldsymbol{\kappa}^{(2)}) \cdot \omega_{0,1+n}(w).
\]
where ${}^{(2)}$ stresses that the $\kappa$-classes are associated to a second factor space in the gluing of a genus $h$ surface with a genus $g$ surface to obtain a genus $g + h$ surface, so that $\boldsymbol{\kappa} = \boldsymbol{\kappa}^{(1)} + \boldsymbol{\kappa}^{(2)}$. After Lemma~\ref{lem:FTFT} we recognise $\omega_{g,1+n}(w) = \alpha^{g} \cdot \omega_{0,1+n}(w)$, and thus the claimed \eqref{omen0} once the exponents are dropped from the notation of the $\kappa$-classes.
\end{proof}

\subsection{Calculating free-boundary F-CohFTs}
\label{sec:free}

We carry out a similar analysis to extract, from a given free-boundary F-CohFT, R-elements of the F-Givental group that can reconstruct it from the underlying F-TFT. Again, the result holds for any free-boundary F-CohFT but is stated for the particular one $\Omega^{\circ}$ associated to a given ctF-CohFT $\Omega$ by Proposition~\ref{propF1}.  

Here it is necessary to work with moduli spaces for surfaces with a mix of pinned and free boundaries. We call (Figure~\ref{fig:pimap})
\[
\pi_{\textnormal{in}} : \M_{g,1+n}^{\bullet \textnormal{in}} \longrightarrow \M_{g,1+n}^{\circ}\qquad \textnormal{and}\qquad \pi_{\textnormal{out}} : \M_{g,1+n}^{\bullet \textnormal{out}} \longrightarrow \M_{g,1+n}^{\circ}
\]
the $(\mathbb{S}_1)^{n}$-bundle (resp. the $\mathbb{S}_1$-bundle) whose fiber above $\Sigma \in \M_{g,1+n}^{\circ}$ parametrises choices of origins on each of the last $n$ boundaries of $\Sigma$ (resp. the choice of an origin on the first boundary).

\begin{figure}[h!]
\begin{center}
\includegraphics[width=0.33\textwidth]{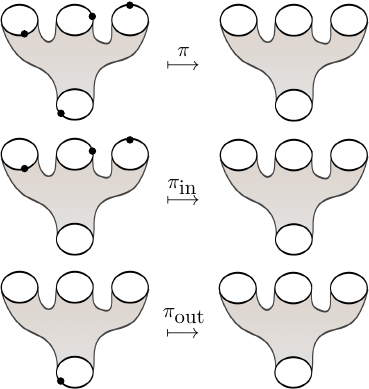}
\caption{Pinning boundaries \label{fig:pimap}}
\end{center}
\end{figure}

\begin{proposition}
\label{prop:Rkappapsi}
Let $\Omega$ be an invertible ctF-CohFT. Then, there exist two elements $R_{\textnormal{in}}(z,\boldsymbol{\kappa})$ and $R_{\textnormal{out}}(z,\boldsymbol{\kappa})$ in $\End(V) \otimes \mathbb{C}\llbracket z,\kappa_1,\kappa_2,\ldots\rrbracket$ defined by the formulae, for any $v \in V$
\begin{equation}
\label{RoutRin}
\begin{split}
R_{\textnormal{out}}(\psi_1,\boldsymbol{\kappa})[v] & := \lim_{h \rightarrow \infty} \pi_{\textnormal{in}}^* \Omega_{h,1+1}^{\circ}(\alpha^{-h} \cdot v), \\
R_{\textnormal{in}}(\psi_2,\boldsymbol{\kappa})[v] & := \lim_{h \rightarrow \infty} \alpha^{-h} \cdot \pi_{\textnormal{out}}^* \Omega_{h,1+1}^{\circ}(v).
\end{split}
\end{equation}

These elements satisfy $R_{\textnormal{out}}(-z,\boldsymbol{0}) \circ R_{\textnormal{in}}(z,\boldsymbol{0}) = \textnormal{Id}_{V\llbracket z\rrbracket}$ and $R_{\textnormal{out}}(0,\boldsymbol{0}) = R_{\textnormal{in}}(0,\boldsymbol{0}) = \Id_{V}$, and
\begin{equation}
\label{thetwotoR}
\forall v \in V\qquad R_{\textnormal{in}}(z,\boldsymbol{\kappa})[v] = \hat{T}(\boldsymbol{\kappa}) \cdot R_{\textnormal{in}}(z,\boldsymbol{0})[v],\qquad R_{\textnormal{out}}(z,\boldsymbol{\kappa})[v] = R_{\textnormal{out}}(z,\boldsymbol{0})[\hat{T}(\boldsymbol{\kappa})  \cdot v].
\end{equation}

Besides, we have for any $g,n \geq 0$ such that $2g - 1 + n > 0$ and any $v_1,\ldots,v_n \in V$
\begin{equation}
\label{Omegacununu}
\Omega_{g,1+n}^{\circ}(v_1 \otimes \cdots \otimes v_n) = R_{\textnormal{out}}(\psi_1,\boldsymbol{0})\bigg[\hat{T}(\boldsymbol{\kappa}) \cdot \omega_{g,1+n}\bigg(\bigotimes_{i = 1}^{n} R_{\textnormal{in}}(\psi_{1+i},\boldsymbol{0})[v_i]\bigg)\bigg].
\end{equation}
\end{proposition}

\begin{definition}
\label{defR}
Given an invertible ctF-CohFT $\Omega$, we define $R(z) \in \End(V)\llbracket z \rrbracket$ by
\[
R(z) = R_{\textnormal{in}}^{-1}(z,\boldsymbol{0}) = R_{\textnormal{out}}(-z,\boldsymbol{0}).
\]
\end{definition}
This convention for the definition of $R(z)$ allows a fair distribution of signs in its relation to $R_{\textnormal{in}}(z)$ and $R_{\textnormal{out}}(z)$. It is a group-like element, \textit{i.e.} we have $R(0) = \textnormal{Id}_V$. 

\begin{corollary}
\label{coformsmooth}
In the situation of Proposition~\ref{prop:Rkappapsi}, we have for any $g,n \geq 0$ such that $2g - 1 + n > 0$ and $v_1,\ldots,v_n \in V$ the equality
\begin{equation}
\label{Omegacirquesmooth}
\Omega^\circ_{g,1+n}(v_1 \otimes \cdots \otimes v_n) = R(-\psi_1) \bigg[\hat{T}(\boldsymbol{\kappa}) \cdot \alpha^{g} \cdot \prod_{i = 1}^{n} R^{-1}(\psi_{1+i})[v_i]\bigg].
\end{equation}
\end{corollary}
\begin{proof} We use \eqref{Omegacununu} and the value of the F-TFT.
\end{proof}
We recognise in this formula a bit of the F-Givental group on the F-TFT $\omega$, \textit{cf.} Definition~\ref{Reaction}. More precisely, it coincides to the restriction of $RT\omega$ to the moduli of smooth complex curves $\M_{g,1+n}$.

The proof of Proposition~\ref{prop:Rkappapsi} follows a pattern similar to Proposition~\ref{prop:Zkappa} and for this reason we give less details. But, there is an extra ingredient we should insist on, namely the description of the cohomology of circle bundles from Gysin--Leray sequence.

\begin{theorem}\cite[Proposition 14.33]{bott_differential_2010} \label{LerayGysin}
Let $\mathcal{X} \rightarrow \mathcal{B}$ be a smooth $\mathbb{S}_1$-bundle. Then
\[
\forall k \geq 0\qquad H^k(\mathcal{X}) = H^k\big((H^*(\mathcal{B})[\eta],\mathrm{d})\big),
\]
where $\eta$ is a generator of cohomological degree $1$ and $\mathrm{d}$ is the unique differential such that $\mathrm{d}\eta = c_1(\mathcal{X})$ and $\mathrm{d}(H^*(\mathcal{B})) = 0$. In other words
\begin{equation}
\label{HSS}
H^k(\mathcal{X}) = \frac{H^k(\mathcal{B})}{\textnormal{Im}\big( c_1(\mathcal{X}) \cup - \big)} \oplus \textnormal{Ker}\big(c_1(\mathcal{X}) \cup -\big).\eta,
\end{equation}
where the cup product acts on $H^{k - 2}(\mathcal{B})$ for the first summand, and on $H^{k - 1}(\mathcal{B})$ for the second.
\end{theorem}

In particular, since $-\psi_i  \in H^2(\M_{g,1+n}^{\circ})$ is the first Chern class of the $\mathbb{S}_1$-bundle whose fibers parametrise the choice of an origin on the $i$-th boundary, the $\psi_i$-class on the $\mathbb{S}_1$-bundle itself is killed. This justifies that the stable cohomology is  $\mathbb{C}[\psi_2,\ldots,\psi_{1+n},\kappa_1,\kappa_2,\cdots]$  for  $\M_{g,1+n}^{\bullet \textnormal{out}}$ while it is $\mathbb{C}[\psi_1,\kappa_1,\kappa_2,\cdots]$ for $\M_{g,1+n}^{\bullet \textnormal{in}}$, and accounts for the difference in the stable cohomologies of the two types of moduli spaces in Theorem~\ref{Mumfordwithpsiclasses}.

It will appear in the proof of Proposition~\ref{prop:Rkappapsi} that $R(-z)$ (resp. $R^{-1}(z)$) is the class we need to attach to a pinned outgoing (resp. incoming) boundary to make it free, with $z$ replaced by the $\psi$-class associated to this boundary. For a pinned boundary the $\psi$-class is killed and we are left with $R(z=0) = \Id_{V}$. In particular, pullback by $\pi_{\textnormal{in}}$ is pinning the incoming boundaries, only leaving a $\psi$-class associated to an outgoing boundary. This explains the exchanged role of $\textnormal{in}$ and $\textnormal{out}$ between $R$s and $\pi^*$s in Proposition~\ref{prop:Rkappapsi}.

\begin{proof}[Proof of Proposition~\ref{prop:Rkappapsi}]
The method of proof of Lemma~\ref{ghnlem} shows  for any $v \in V$, $g \geq 0$ and $h \geq 1$
\begin{equation*}
\begin{split}
 \varphi^*_h \pi^*_{\textnormal{in}}\Omega^{\circ}_{g+h,1+1}(v) & =  \pi^*_{\textnormal{in}}\Omega^{\circ}_{g,1+1}(\alpha^h \cdot v), \\
\tilde{\varphi}^*_h \pi^*_{\textnormal{out}} \Omega^{\circ}_{g+h,1+1}(v) & =  \alpha^h \cdot \pi^*_{\textnormal{out}}\Omega^{\circ}_{g,1+1}(v).
\end{split}
\end{equation*}
Here, $\tilde{\varphi}_h$ is obtained by $h$ iterations of maps like $\tilde{\varphi} = \textnormal{gl}^{\bullet}(-,\Yu)$ which glue $\Yu$ to the incoming boundary, while $\varphi_h$ was obtained by iterations of maps like $\textnormal{gl}^{\bullet}(\Yu,-)$ which glue $\Yu$ to the outcoming boundary. The pullbacks $\pi_{\textnormal{in}}^*$ and $\pi_{\textnormal{out}}^*$ have the effect of pinning the corresponding boundaries, which is necessary before being able to glue them. This justifies the existence of the limits $R_{\textnormal{in}}$ and $R_{\textnormal{out}}$ in \eqref{RoutRin}, and we have for any $g \geq 1$ and $v \in V$
\begin{equation}
\label{RinRoutstable}
\begin{split}
\pi^*_{\textnormal{in}} \Omega_{g,1+1}^{\circ}(v) & = R_{\textnormal{out}}(\psi_1,\boldsymbol{\kappa})[\alpha^{g} \cdot v], \\
\pi^*_{\textnormal{out}}\Omega_{g,1+1}^{\circ}(v) & = \alpha^{g} \cdot R_{\textnormal{in}}(\psi_2,\boldsymbol{\kappa})[v].
\end{split}
\end{equation}
Pulling this back to $\M_{g,1+n}^{\bullet}$ replaces $\psi_1,\psi_2$ by $0$, and evaluating in the stable cohomology to $\boldsymbol{\kappa} = \boldsymbol{0}$ and comparing with Proposition~\ref{prop:Zkappa} gives $R_{\textnormal{in}}(0,\boldsymbol{0}) = R_{\textnormal{out}}(0,\boldsymbol{0}) = \Id_V$.

\begin{figure}[h!]
\begin{center}
\includegraphics[width=0.45\textwidth]{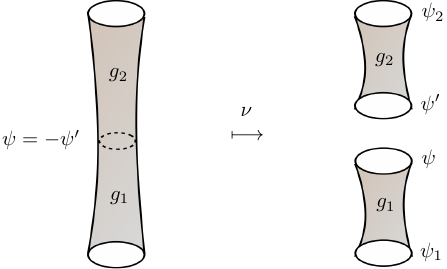}
\caption{\label{fig:nupsimap} Geometry of the map $\nu$.}
\end{center}
\end{figure}

Let $g_1,g_2 \geq 1$ and set $g = g_1 + g_2$. On the space $\M_{g_1,1+1}^{\circ} \times \M_{g_2,1+1}^{\circ}$ we have the classes $\psi_{1} := \psi_1^{(1)}$ and $\psi := \psi_2^{(1)}$ from the first factor, and $\psi' := \psi^{(2)}_1$ and $\psi_2 := \psi_2^{(2)}$ from the second factor (Figure~\ref{fig:nupsimap}). We apply Theorem~\ref{LerayGysin} to the $\mathbb{S}_1$-bundle $\nu : \partial \mathcal{N} \rightarrow \M_{g_1,1+1}^{\circ} \times \M_{g_2,1+1}^{\circ}$ which has first Chern class $-(\psi + \psi')$. By Theorem~\ref{Mumfordwithpsiclasses}, for fixed  $k$ when $g_1$ is large enough, we have an isomorphism between $H^k(\M_{g_1,1+1}^{\circ} \times \M_{g_2,1+1}^{\circ})$ and the degree $k$ part of $H^{\textnormal{even}}(\M_{g_2,1+1}^{\circ})[\psi_1,\psi,\kappa_1^{(1)},\kappa_2^{(1)},\ldots]$. In this ring the multiplication by the class $- (\psi + \psi')$ is injective (simply by decomposing on monomials in $\psi$) so the second summand in \eqref{HSS} is absent. The same argument would hold if $g_2$ (instead of $g_1$) were large enough. So, for $k$ fixed and $g_1$ \emph{or} $g_2$ large enough, we have
\begin{equation}
\label{delNHk} H^k(\partial \mathcal{N}) \simeq \frac{H^k(\M_{g_1,1+1}^{\circ} \times \M_{g_2,1+1}^{\circ})}{(\psi + \psi')}.
\end{equation}
The compatibility property \eqref{compascirc} for $\Omega^{\circ}$ yields
\[
\Omega_{g,1+1}^{\circ}{}_{|\partial \mathcal{N}} = \nu^*(\Omega_{g_1,1+1}^{\circ} \circ \Omega_{g_2,1+1}^{\circ}),
\]
After lifting to $\pi^{-1}(\partial \mathcal{N}) \subset \M_{g,1+1}^{\bullet}$ (like in the proof of Proposition~\ref{propF1}) we get
\begin{equation*}
\Omega_{g,1+1}^{\bullet}|_{\pi^{-1}(\partial \mathcal{N})} = \pi_{\textnormal{out}}^*\Omega_{g_1,1+1}^{\circ} \circ  \pi_{\textnormal{in}}^* \Omega_{g_2,1+1}^{\circ} |_{\psi' = - \psi}.
\end{equation*}
We know $\Omega_{g,1+1}^{\bullet}$ from \eqref{omen0}, and restricting to $\pi^{-1}(\partial \mathcal{N})$ just replaces $\boldsymbol{\kappa} = \boldsymbol{\kappa}^{(1)} + \boldsymbol{\kappa}^{(2)}$ to compare with the right-hand side (this is the pullback of \eqref{delNHk} via $\pi$). The right-hand side itself can be rewritten with help of \eqref{RinRoutstable}. Using commutativity of the product, we get
\[
\forall v \in V\qquad \alpha^{g_1} \cdot \hat{T}(\boldsymbol{\kappa}^{(1)} + \boldsymbol{\kappa}^{(2)}) \cdot \alpha^{g_2} \cdot v = \alpha^{g_1} \cdot R_{\textnormal{in}}(\psi,\boldsymbol{\kappa}^{(1)}) \circ R_{\textnormal{out}}(-\psi,\boldsymbol{\kappa}^{(2)})[\alpha^{g_2} \cdot v].
\]
By invertibility of $\alpha$ we deduce
 \begin{equation}
 \label{beforeev}
 \forall v \in V\qquad \hat{T}(\boldsymbol{\kappa}^{(1)} + \boldsymbol{\kappa}^{(2)}) \cdot v = R_{\textnormal{in}}(\psi,\boldsymbol{\kappa}^{(1)}) \circ R_{\textnormal{out}}(-\psi,\boldsymbol{\kappa}^{(2)})[v].
 \end{equation}
As the stable cohomology is freely generated, we can evaluate this to $\boldsymbol{\kappa}^{(1)} = \boldsymbol{\kappa}^{(2)} = \boldsymbol{0}$. Due to the exponential form of $\hat{T}(\boldsymbol{\kappa})$ from Proposition~\ref{prop:Zkappa}, we get
\begin{equation}
\label{inverseR}
\textnormal{Id}_{V} = R_{\textnormal{in}}(\psi,\boldsymbol{0}) \circ R_{\textnormal{out}}(-\psi,\boldsymbol{0}).
\end{equation}
If we only evaluate \eqref{beforeev} to $\boldsymbol{\kappa}^{(2)} = \boldsymbol{0}$ or to $\boldsymbol{\kappa}^{(1)} = \boldsymbol{0}$ and simply call $\boldsymbol{\kappa}$ the other one (which is a free generator in the stable cohomology), we get for any $v \in V$
 \begin{equation}
 \label{Rnotzer}
 \begin{split}
 R_{\textnormal{in}}(\psi,\boldsymbol{\kappa})[v] & = \hat{T}(\boldsymbol{\kappa}) \cdot R_{\textnormal{out}}^{-1}(-\psi,\boldsymbol{0})[v] = \hat{T}(\boldsymbol{\kappa}) \cdot R_{\textnormal{in}}(\psi,\boldsymbol{0})[v], \\
 R_{\textnormal{out}}(\psi',\boldsymbol{\kappa})[v] & = R_{\textnormal{in}}^{-1}(-\psi',\boldsymbol{0})\big[\hat{T}(\boldsymbol{\kappa}) \cdot v\big] = R_{\textnormal{out}}(\psi',\boldsymbol{0})\big[\hat{T}(\boldsymbol{\kappa}) \cdot v\big],
 \end{split}
 \end{equation}
For the rightmost equalities we made use of \eqref{inverseR}. This reproduces \eqref{thetwotoR}.
 
 \begin{figure}[h!]
\begin{center}
\includegraphics[width=\textwidth]{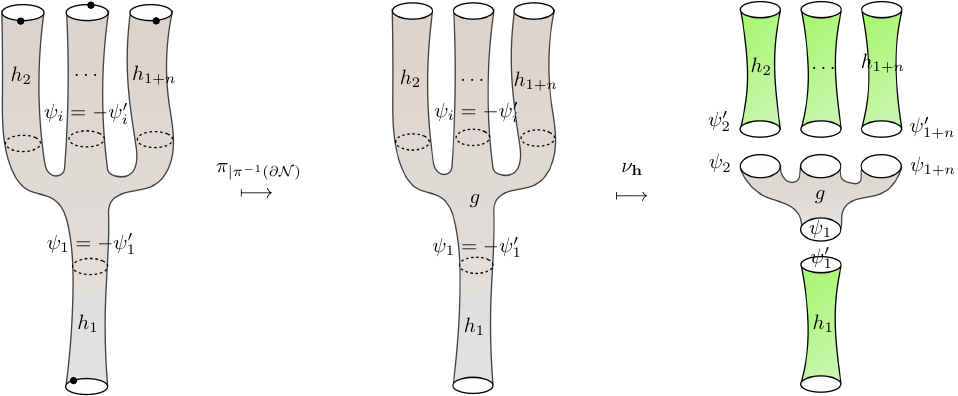}
\caption{\label{fig:mapnuh} Geometry of the map $\nu_{\boldsymbol{h}}$. The components in green have large genera $h_1,\ldots,h_{1+n}$. Note the unusual convention that $\psi_1'$ is associated to an ingoing (instead of outgoing) edge, which is responsible for the minus signs in the formulae.}
\end{center}
\end{figure}
 
Eventually, let $g,n \geq 0$ such that $2g - 1 + n > 0$ and $v_1',\ldots,v_n' \in V$. We take a $(1+n)$-tuple of positive integers $\boldsymbol{h}$, set $h = h_1 + \cdots + h_{1+n}$ and consider the $(\mathbb{S}_1)^{1 + n}$-bundle map (Figure~\ref{fig:mapnuh})
\begin{equation}
\label{eq_nuh}
\nu_{\boldsymbol{h}} : \partial \mathcal{N} \longrightarrow \M_{h_1,1+1}^{\circ} \times \M_{g,1+n}^{\circ} \times \prod_{i = 1}^{n} \M^{\circ}_{h_{1+i},1+1} \qquad \textnormal{with}\quad \partial \mathcal{N} \subset \M_{g+h,1+n}^{\circ}
\end{equation}
obtained by iterations of the maps $\nu$ like in \eqref{numap}. The compatibility of $\Omega^{\circ}$ with \eqref{eq_nuh} yields 
\[
\Omega_{g+h}^{\circ}(v_1' \otimes \cdots \otimes v_n')_{|\partial \mathcal{N}}  = \nu^*_{\boldsymbol{h}} \bigg(\Omega_{h_1,1+1}^{\circ} \circ \Omega_{g,1+n}^{\circ} \bigg[ \bigotimes_{i = 1}^{n} \Omega_{h_{1+i},1+1}^{\circ}[v_i']\bigg]\bigg).
\]
Pulling back by $\pi$ to the moduli with pinned boundaries, we get
\begin{equation}
\label{oibycglb}
\Omega_{g+h}^{\bullet}(v_1' \otimes \cdots \otimes v_n')|_{\pi^{-1}(\partial \mathcal{N})} = \pi_{\textnormal{out}}^*\Omega^\circ_{h_1,1+1} \circ \Omega_{g,1+n}^{\circ}\bigg[ \bigotimes_{i = 2}^{1+n}  \pi_{\textnormal{in}}^* \Omega_{h_{i},1+1}^{\circ}[v_i']\bigg]\bigg|_{\substack{\psi_j' = - \psi_j \\ j \in [n + 1]}},
\end{equation}
where we call $\psi_i$ the class on $\M_{g,1+n}^{\circ}$ coming from the $i$-th boundary, $\psi'_{1+i}$ the class in $\M_{h_{1+i},1+1}^{\circ}$ coming from the boundary identified with this $i$-th boundary in $\partial \mathcal{N}$. When looking at those equalities for an arbitrary but fixed cohomology degree, we can take $h_1,\ldots,h_{1+n}$ large enough so that  statements like \eqref{delNHk} are available. Then, on $\pi^{-1}(\partial \mathcal{N})$ we have the relations $\psi_i + \psi'_i = 0$ for every $i$ coming from various instances of pullback of \eqref{delNHk} via $\pi$, and the kappa classes $\tilde{\boldsymbol{\kappa}}$ from the ambient $\M_{g+h,1+n}^{\bullet}$ decompose as
\begin{equation}
\label{sumkapp}
\tilde{\boldsymbol{\kappa}} = \boldsymbol{\kappa} + \boldsymbol{\kappa}^{(1)} + \cdots + \boldsymbol{\kappa}^{(1+n)},
\end{equation}
where $\boldsymbol{\kappa}^{(i)}$ are the kappa classes pulled back from $\M_{h_{i},1+1}^{\circ}$.  As before, we know the left-hand side of \eqref{oibycglb} from Proposition~\ref{prop:Zkappa}
\[
\Omega_{g+h}^{\bullet}(v_1' \otimes \cdots \otimes v_n')_{|\pi^{-1}(\partial \mathcal{N})} = \hat{T}(\tilde{\boldsymbol{\kappa}}) \cdot \omega_{g+h,1+n}(v_1' \otimes \cdots \otimes v_n') = \hat{T}(\tilde{\boldsymbol{\kappa}}) \cdot \alpha^{h} \cdot \omega_{g,1+n}(v_1' \otimes \cdots \otimes v_n'),
\]
where we should substitute \eqref{sumkapp}. In the right-hand side of \eqref{oibycglb}, using \eqref{RinRoutstable} together with \eqref{Rnotzer} and the relations $\psi_i' = - \psi_i$, the first and last factors become
\begin{equation}
\begin{split}
\pi_{\textnormal{out}}^* \Omega_{h_1,1+1}^{\circ}[v'] & = \alpha^{h_1} \cdot R_{\textnormal{in}}(\psi_1',\boldsymbol{\kappa}^{(1)})[v'] = \alpha^{h_i} \cdot \hat{T}(\boldsymbol{\kappa}^{(1)}) \cdot R_{\textnormal{in}}(\psi'_1,\boldsymbol{0})[v'] \\
& = \alpha^{h_i} \cdot \hat{T}(\boldsymbol{\kappa}^{(1)}) \cdot R_{\textnormal{out}}^{-1}(\psi_1,\boldsymbol{0})[v_1'], \\
\pi_{\textnormal{in}}^* \Omega_{h_{1+i},1+1}^{\circ}(v_i') & = R_{\textnormal{out}}(\psi_{1+i}',\boldsymbol{\kappa}^{(1+i)})[\alpha^{h_{1+i}} \cdot v_i'] = R_{\textnormal{out}}(\psi_{1+i}',\boldsymbol{0})\big[\hat{T}(\boldsymbol{\kappa}^{(1+i)}) \cdot \alpha^{h_{1+i}} \cdot v'_i\big] \\
& = R_{\textnormal{in}}^{-1}(\psi_{1+i},\boldsymbol{0})\big[\hat{T}(\boldsymbol{\kappa}^{(1+i)}) \cdot \alpha^{h_{1+i}} \cdot v_i'\big].
\end{split}
\end{equation}
Now take $v_1,\ldots,v_n \in V$ and use the previous identities for $v_i' = \alpha^{-h_{1+i}} \cdot \hat{T}(-\boldsymbol{\kappa}^{(1+i)}) \cdot R_{\textnormal{in}}(\psi_{1+i},\boldsymbol{0})$, that is $\pi_{\textnormal{in}}^* \Omega_{h_{1+i},1+1}^{\circ}(v_i') = v_i$. Thanks to the exponential form of $\hat{T}$ and \eqref{sumkapp}, we extract from \eqref{oibycglb}  the equality
\begin{equation}
\begin{split}
& \quad \alpha^{h_1} \cdot \hat{T}(\boldsymbol{\kappa}^{(1)}) \cdot R_{\textnormal{out}}^{-1}(\psi_1,\boldsymbol{0})\bigg[\Omega^\circ_{g,1+n}(v_1 \otimes \cdots \otimes v_n)\bigg] \\
& = \hat{T}(\tilde{\boldsymbol{\kappa}}) \cdot \omega_{g+h,1+n}\bigg(\bigotimes_{i = 1}^{n} \alpha^{-h_{1+i}} \cdot \hat{T}(-\boldsymbol{\kappa}^{(1+i)}) \cdot R_{\textnormal{in}}(\psi_{1+i},\boldsymbol{0})[v_{i}]\bigg) \\
& =  \hat{T}(\boldsymbol{\kappa} + \boldsymbol{\kappa}^{(1)}) \cdot \alpha^{h_1} \cdot \omega_{g,1+n}\bigg(\bigotimes_{i = 1}^{n} R_{\textnormal{in}}(\psi_{1+i},\boldsymbol{0})[v_i]\bigg).
\end{split}
\end{equation}
Thus, we can isolate
\[
\Omega^\circ_{g,1+n}(v_1 \otimes \cdots \otimes v_n) = R_{\textnormal{out}}(\psi_1,\boldsymbol{0})\bigg[\hat{T}(\boldsymbol{\kappa}) \cdot \omega^\circ_{g,1+n}\bigg(\bigotimes_{i = 1}^{n} R_{\textnormal{in}}(\psi_{1+i},\boldsymbol{0})[v_i]\bigg)\bigg].
\]
\end{proof}

\begin{proposition}
\label{prop:RTflatunit}
If $\Omega$ is an invertible ctF-CohFT, take $R(z) \in \End(V)\llbracket z \rrbracket$ from Proposition~\ref{prop:Rkappapsi} and Definition~\ref{defR}, and $T(z) \in  z^2V\llbracket z \rrbracket$ from Proposition~\ref{prop:Zkappa} and Definition~\ref{defT}.  Set
\[
\hat{T}(z) := \bigg(\1 - \frac{T(z)}{z}\bigg)^{-1},\qquad \Upsilon(z) := R(z)[\hat{T}^{-1}(z)]
\]
Then, pulling back by $f_{\theta} = \theta^{-1} \circ f \circ \theta$ coming from the forgetful morphism $f : \M_{g,1+n+1} \rightarrow \M_{g,1+n}$ takes the form
\[
f_{\theta}^* \Omega^\circ_{g,1+n}(v_1 \otimes \cdots \otimes v_n) = \Omega^\circ_{g,1+n + 1}\big(v_1 \otimes \cdots \otimes v_n \otimes \Upsilon(\psi_{1+n+1})\big).
\]
If furthermore $\Omega$ has flat unit $\1$, then  $\Upsilon(z) = \1$ and $T(z) = z\big(\1 - R^{-1}(z)[\1]\big)$.
\end{proposition} 

\begin{proof} We transport the identity \eqref{Omegacirquesmooth} to $\M_{g,1+n}$ using the isomorphism $\theta_*$ and examine the behavior of its right-hand side under the forgetful morphism $f : \M_{g,1+n+1} \rightarrow \M_{g,1+n}$ between moduli of \emph{smooth} curves. Restricting Lemma~\ref{lem39} to the latter, we have $f^*\psi_i = \psi_i$ for $i \in [n + 1]$ and  $f^*\kappa_m = \kappa_m - \psi_{1+n+1}^{m}$ for $m \geq 0$ (here kept the same notations for classes on the two spaces). In particular, $f^*$ preserves the structure of the leaf factors involving $R(\psi_i)$ and the exponential form of $\hat{T}(\boldsymbol{\kappa})$ implies that 
\[
f^*\big(\hat{T}(\boldsymbol{\kappa})\big) = \hat{T}(\boldsymbol{\kappa}) \cdot \hat{T}^{-1}(\psi_{1+n+1})\qquad \textnormal{with}\quad \hat{T}(z) = \exp\bigg(\sum_{m \geq 1} \hat{t}_m z^{m}\bigg).
\]
Noticing that the F-TFT in genus $0$ is the iteration of the commutative product $\cdot$ and comparing again with \eqref{Omegacirquesmooth} for one more boundary, we obtain
\begin{equation*}
\begin{split}
f^*\Omega_{g,1+n}^{\circ}(v_1 \otimes \cdots \otimes v_n) & = R(-\psi_1)\bigg[\hat{T}(\boldsymbol{\kappa}) \cdot \omega_{0,1+n}\bigg(\bigotimes_{i = 1}^{n} R^{-1}(\psi_{1+i})[v_i]\bigg) \cdot \hat{T}^{-1}(\psi_{1 + n + 1})\bigg] \\
& = R(-\psi_1)\Bigg[\hat{T}(\boldsymbol{\kappa}) \cdot \omega_{0,1+n+1}\Bigg(\bigg(\bigotimes_{i = 1}^{n} R^{-1}(\psi_{1+i})[v_i]\bigg) \otimes \hat{T}^{-1}(\psi_{1+n+1})\Bigg)\Bigg] \\
& = \Omega_{g,1+n+1}^{\circ}\big(v_1 \otimes \cdots \otimes v_n \otimes R(\psi_{1+n+1})[\hat{T}^{-1}(\psi_{1+n+1})]\big).
\end{split}
\end{equation*}
If the ctF-CohFT satisfies the flat unit axiom, this should also be equal to $\Omega_{g,1+n+1}^{\circ}(v_1 \otimes \cdots \otimes v_n \otimes \1)$.  As $g \rightarrow \infty$ it induces an identity in the stable cohomology where the $\psi$-classes are free. Taking $n = 1$ it forces $R(z)[\hat{T}^{-1}(z)] = \1$. Taking into account Lemma~\ref{expkappa}, this makes $T(z) = z(\1 - R^{-1}(z)[\1])$.
\end{proof}

\subsection{Unique patching of cohomology classes on strata}

\label{sec:patch}
For the proof of Theorem~\ref{thm:transfree} in the next Section, we need a technical result showing that classes (not necessarily ctF-CohFTs) are uniquely determined by their restriction to each strata of $\M^{\textnormal{ct}}_{g,1+n}$ at least in some stability range. The original argument for $\Mbar_{g,n}$ is Teleman's second main idea \cite[Section 5]{teleman_structure_2012}, see also the review \cite{chiarello_telemans_2016}. Recall the description of strata $\mathcal{S}_{\Gamma} = \textnormal{gl}(\Mct_{\Gamma}) \subset \M_{g,1+n}^{\textnormal{ct}}$ in terms of stable trees $\Gamma \in T_{g,1+n}$ in Definition~\ref{def:inclu}.

We start by stating an elementary property coming from the relative cohomology long exact sequence, which characterizes the obstruction to have unique gluing.

\begin{lemma}
\label{MVlem} 
Let $M$ be a smooth orbifold, $S$ be a closed, smooth suborbifold of codimension $k$, and $\phi\in H^*(M)$ such that $\phi_{|S} = 0$ and $\phi_{|M \setminus S} = 0$. Then
\[
\phi=\sigma\left(j_*\textnormal{Ker}( c_{k}(N)\cup-))\right),
\]
where $c_{k}(N) \in H^{2k}(S)$ is the top Chern class of the normal bundle of $S$, the map $j_*:H^{d-2k}(S) \xrightarrow{\sim} H^d(M, M\setminus S)$ is Thom isomorphism and $\sigma$ is the map in the relative cohomology sequence: 
\[
\cdots \longrightarrow H^d(M, M\setminus S) \mathop{\longrightarrow}^{\sigma} H^d(M) \longrightarrow H^{d + 1}(M\setminus S)\longrightarrow H^{d+1}(M, M\setminus S)\longrightarrow\cdots
\]
\end{lemma}
A proof can we found, for instance, in \cite[Lemma 13]{chiarello_telemans_2016}. We want to apply  Lemma~\ref{MVlem} to show that, if a class vanishes on some strata, under a certain degree of cohomology it vanishes also on the union of those strata. To this end we will in fact work directly with union of strata sharing the same topological type for the root (this allows controlling the stability range) and we define a total order so that strata will be added step by step in a descending order.

\begin{definition}
\label{Def:tautype}
If $\Gamma$ is a stable tree in $T_{g,1+n}$, we call \emph{root type} the triple $\tau = (g_1,\ell_1,k_1)$, where $g_1$ is the genus of the vertex carrying the root, $\ell_1$ its number of ingoing leaves, and $k_1$ its number of ingoing edges. We call $\M_{g,1+n}^{\tau}$ the union of strata $\mathcal{S}_{\Gamma}$ over $\Gamma \in T_{g,1+n}$ of root type $\tau$, and $\M_{g,1+n}^{\textnormal{ct},\tau}$ its closure in $\M_{g,1+n}^{\textnormal{ct}}$ (\textit{cf.} Figure~\ref{fig:Mtau}).

For fixed $(g,n)$, we equip the set of root types of stable trees in $T_{g,1+n}$ with the lexicographic order, namely $\tau' \prec \tau$ if $g_1' < g_1$, or if $g_1 = g_1'$ and $l_1' < l_1$, or if $(g_1,l_1) = (g_1',l_1')$ and $k_1' < k_1$.
\end{definition} 

 Note that, if $\M_{g,1+n}^{\tau'} \cap \M_{g,1+n}^{\textnormal{ct},\tau} \neq \emptyset$, then  $\tau' \preceq \tau$: a stable curve can only degenerate into a stable curve of lower root type. The converse implication does not hold, see the example of Figure~\ref{fig:Mtau13}. We also denote
\[
\M_{g,1+n}^{\succeq \tau} = \bigsqcup_{\tau' \succeq \tau} \M_{g,1+n}^{\tau'}.
\]

\begin{figure}[h!]
\begin{center}
\includegraphics[width=0.9\textwidth]{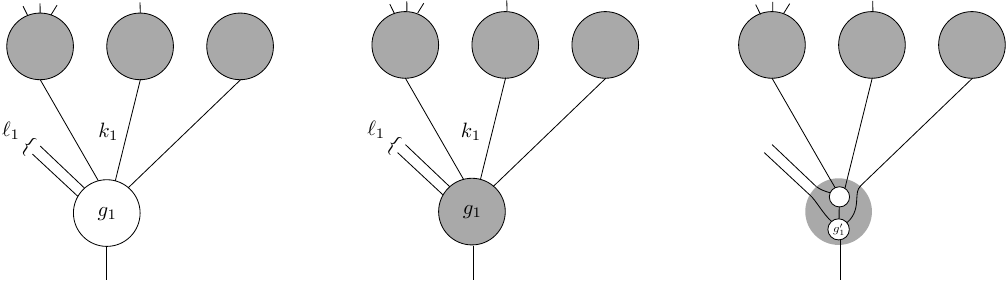}
\caption{\label{fig:Mtau} We depict $\M_{g,1+n}^{\tau}$ on the left, $\M_{g,1+n}^{\textnormal{ct},\tau}$ in the middle, and a non-empty $\M_{g,1+n}^{\tau'} \cap \M_{g,1+n}^{\textnormal{ct},\tau}$ for some $\tau \preceq \tau$ on the right. The moduli spaces $\M^{\textnormal{ct}}$ appear as grey vertices and $\M$ as white vertices. One should take the union over all possible topologies for vertices on the top so that the total genus is $g$ and total number of leaves is $n$, and eventually apply the gluing morphism to land in $\M_{g,1+n}^{\textnormal{ct}}$.}
\end{center}
\end{figure} 
\begin{figure}[h!]
\begin{center}
\includegraphics[width=0.9\textwidth]{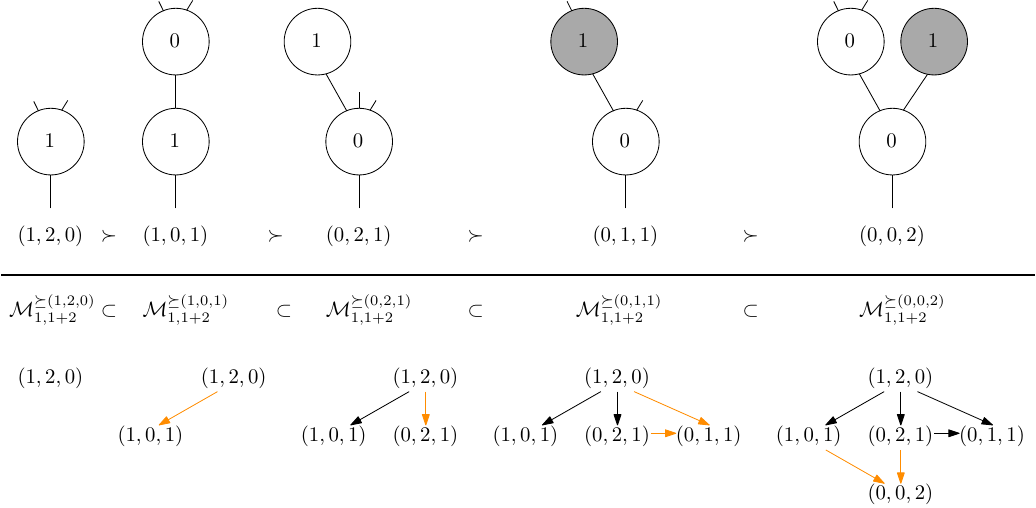}
\caption{\label{fig:Mtau13} Top: the stable trees representing the open strata $\mathcal{M}_{1,1+2}^{\tau}$. Taking into account leaf labels, there are two distinct trees for $(0,1,1)$. Bottom: strata of fixed root type in $\mathcal{M}_{1,1+2}^{\succeq \tau}$ and their incidences with the newly added stratum $\mathcal{M}_{1,1+2}^{\tau}$. The root type $(0,1,1)$ can be reached by degeneration from $(1,2,0)$, or from $(0,2,1)$ and the result lies in a codimension-one locus in $\mathcal{M}_{1,1+2}^{(0,1,1)}$).}
\end{center}
\end{figure}

\begin{lemma}
	\label{lemma: the euler class is not a 0 div}
Let $\tau = (g_1,\ell_1,k_1)$ be the root type of a stable tree in $T_{g,1+n}$, and $\phi \in H^k(\M_{g,1+n}^{\textnormal{ct}})$ such that $k \leq \frac{2g_1}{3}$. Assume that the restriction of $\phi$ to $\M_{g,1+n}^{\tau'}$ vanishes for every $\tau' \succeq \tau$. Then the restriction of $\phi$ to $\M_{g,1+n}^{\succeq \tau}$ vanishes.
\end{lemma}
\begin{proof}
We prove by descending induction on $\tau'$ until reaching $\tau' = \tau$ that the restriction of $\phi$ to $\M_{g,1+n}^{\succeq \tau'}$ vanishes. 
It is important to observe that  $g_1' \geq g_1$, so the smallest stability range we may encounter is controlled by $g_1$.  
The maximal root type of a stable tree in $T_{g,1+n}$ is $(g,n,0)$, and we have $\M_{g,1+n}^{(g,n,0)} = \M_{g,1+n}$, so the claim holds for $\tau' = (g,n,0)$ by assumption. 
Let $\tau' \succeq \tau$ and assume that the restriction of $\phi$ to $\M_{g,1+n}^{\succeq \tau''}$ vanishes for every $\tau'' \succ \tau'$, so we choose in particular $\tau''$ to be minimal. 
Let us call $\psi_j$ the classes associated to the half-edges of the special component, and $\psi'_j$ the ones associated to the opposite respective half-edge. 
We decompose
\begin{equation}
\label{unionM} \M_{g,1+n}^{\succeq \tau'} = \M_{g,1+n}^{\succeq \tau''} \sqcup \M_{g,1+n}^{\tau'},
\end{equation}
where $S = \M_{g,1+n}^{ \tau'}$ is a closed smooth suborbifold of $M = \M_{g,1+n}^{\succeq \tau'}$ of codimension $k_1'$. The initial assumption gives $\phi_{|S} = 0$ and the induction assumption gives $\phi_{|M \setminus S} = 0$. We then apply Lemma~\ref{MVlem} with  $c_{k_1'}(\mathcal{N})=(-1)^{k_1'}(\psi_1 +\psi_1')\cdots (\psi_{k_1'} + \psi'_{k_1'})$.  As the $\psi_j$ are free generators in the cohomology degree range $\leq \frac{2g_1'}{3}$ for each of the strata of $\M_{g,1+n}^{\succeq\tau'}$, they remain free in the same cohomology range for $\M_{g,1+n}^{\tau'}$.  Therefore $\textnormal{Ker}(c_{k_1'}(\mathcal{N})\cup-) = 0$ at least in degree $\leq \frac{2g_1'}{3}$ and we obtain $\phi_{|M}  = 0$. We conclude by induction.
 \end{proof}

\subsection{Proof of Theorem~\ref{thm:transfree}}
\label{sec:proofthmA}

Given an invertible ctF-CohFT $\Omega$, we have so far constructed $R(z) \in \End(V)\llbracket z \rrbracket$ and $T(z) \in z^2V\llbracket z \rrbracket$ such that\footnote{In the context of Proposition~\ref{afterbeforeprop} this is $T(z) = T_{\textnormal{B}}(z)$.} the two ctF-CohFTs $\Omega$ and $RT\omega$ coincide upon restriction to $\M_{g,1+n}$. We are now in position to upgrade the result and show they coincide on $\M_{g,1+n}^{\textnormal{ct}}$.

\begin{proposition}
\label{thm:recon}
Let $\Omega$ and $\Omega'$ be invertible ctF-CohFTs whose restriction to $\M$ coincide. Then, they coincide on the whole $\M^{\textnormal{ct}}$.
\end{proposition}
\begin{proof}
The compatibility axiom of ctF-CohFTs implies that the restrictions of $\Omega_{g,1+n}$ and $\Omega'_{g,1+n}$ to each strata of $\M^{\textnormal{ct}}_{g,1+n}$ coincide. The agreement on $\M_{g,1+n}^{\textnormal{ct}}$ does not yet follow from Section~\ref{sec:patch} which only applies in stability ranges. We will prove it differently, by induction on the complex dimension $d_{g,1+n} = 3g - 2 + n$ of the moduli spaces. The dimension $0$ case is obvious: it corresponds to $(g,n) = (0,2)$ and $\M_{0,3} = \M_{0,3}^{\textnormal{ct}}$.

Let $d > 0$ and suppose that the restrictions of $\Omega_{g',1+n'}$ and $\Omega'_{g',1+n'}$ to $\M_{g',1+n'}^{\textnormal{ct}}$ agree for $d_{g',1+n'} < d$. 
Let $(g,n)$ be such that $d = d_{g,1+n}$.
 We shall first express $\Omega$ and $\Omega'$ on $\M_{g,1+n}^{\textnormal{ct}}$ in terms of their value on a thickening of a stratum $\mathcal{S}$ in a moduli space involving curves of large genus. More precisely, we take $h \geq 1$ and $\mathcal{S}$ to be the image of the gluing morphism 
\[
\textnormal{gl} : \M_{h,1+1} \times \M_{g,1+n}^{\textnormal{ct}} \longrightarrow \M_{g+h,1+n}^{\textnormal{ct}}.
\]
Recall from Section~\ref{sec:variantFCohFT} that the restriction of $\Omega_{h,1+1}$ to the moduli space of smooth curves is equal to $\theta_*\Omega_{h,1+1}^{\circ}$, where $\theta : \M_{h,1+1}^{\circ} \rightarrow \M_{h,1+1}$ is an isomorphism. The compatibility axiom of ctF-CohFTs after restriction to $\mathcal{S}$ yields for any $w \in V^{\otimes n}$
\[
\textnormal{gl}^*\Omega_{g+h,1+n}(w) = \theta_*\Omega_{h,1+1}^{\circ}\big(\Omega_{g,1+n}(w)\big)
\]
and likewise for $\Omega'$. The element $\Omega_{h,1+1}^{\circ} \in \End(V) \otimes H^{\textnormal{even}}(\M_{h,1+1})$ coincides with $\Omega_{h,1+1}^{\circ\,'}$ by assumption. Its cohomology degree $0$ part is the $h$-th power of $\omega_{1,1} = \alpha$. As its invertibility is assumed, $\theta_*\Omega_{h,1+1}^{\circ}$  is also invertible and we can write
\begin{equation}
\label{undd}\Omega_{g,1+n}(w) = (\theta_*\Omega_{h,1+1}^{\circ})^{-1}\big(\textnormal{gl}^*\Omega_{g+h,1+n}(w)_{|\mathcal{S}}\big).
\end{equation}
The left-hand side should be understood as $\mathbf{1} \otimes \Omega_{g,1+n}(w) \in V \otimes H^{*}(\M_{h,1+1}) \otimes H^{*}(\M_{g,1+n}^{\textnormal{ct}})$ and it suffices to study the right-hand side in degree $\leq 6g - 4 + 2n$ to extract the left-hand side.

We are going to lift this relation to the $\mathbb{S}_1$-bundle $\rho : \partial \mathcal{N}_{\theta} \rightarrow \mathcal{S}$ specified by the boundary of a tubular neighborhood of $\mathcal{S}$. The first Chern class of this bundle is $- (\psi_2^{(1)} + \psi_1^{(2)})$, where the exponents refer to the first or second factor in $\mathcal{S}$. The cohomology of $\partial\mathcal{N}_{\theta}$ is computed by the Gysin--Leray sequence in Theorem~\ref{LerayGysin}. For fixed cohomology degree $k$, the multiplication by $(\psi_2^{(1)}+\psi_1^{(2)})$ is injective provided we choose $h$ large enough (repeat the justification of \eqref{delNHk} with $\M_{g,1+n}^{\textnormal{ct}}$ instead of $\M^{\circ}_{g_2,1+1}$), hence
\begin{equation}
\label{HkdNdS}
H^{k}(\partial\mathcal{N}_{\theta}) \simeq \frac{H^k(\mathcal{S})}{(\psi_2^{(1)}+\psi_1^{(2)})} \simeq \frac{H^{k}(\mathcal{M}_{h,1+1}\times \M_{g,1+n}^{\textnormal{ct}})}{(\psi_2^{(1)}+\psi_1^{(2)})}.
\end{equation}
The choice of $h$ can be made to depend only on $(g,n)$ so that \eqref{HkdNdS} holds for any $k \leq 6g - 4 + 2n$. Since in the stable range $H^*(\mathcal{M}_{h,1+1})$ is the free ring $\mathbb{C}[\psi_1^{(1)},\psi_2^{(1)},\kappa_1,\kappa_2,\ldots]$  (Theorem~\ref{Mumfordwithpsiclasses}) and we have $\boldsymbol{\kappa} = \boldsymbol{\kappa}^{(1)} + \boldsymbol{\kappa}^{(2)}$ (Lemma~\ref{lem39}), pulling back \eqref{undd} via $\rho$ amounts to specialising to $\psi_2^{(1)} = - \psi_1$, where $\psi_1 := \psi_1^{(2)}$ is the class associated to the first marked point in $\M_{g,1+n}^{\textnormal{ct}}$. We do not lose information on the left-hand side if we further specialise to $\boldsymbol{\kappa}^{(1)} = 0$: this has the advantage to equate the kappa classes associated to $\M_{g,1+n}^{\textnormal{ct}}$ with the ones associated to $\M_{g+h,1+n}^{\textnormal{ct}}$. In short
\begin{equation}
\label{Homcun}
\Omega_{g,1+n}(w) = (\theta_*\Omega_{h,1+1}^{\circ})^{-1}\big(\Omega_{g+h,1+n}(w)_{|\partial\mathcal{N}_{\theta}}\big)_{|\boldsymbol{\kappa}^{(1)} = 0,\psi_2^{(1)} = - \psi_1}. 
\end{equation}

Let now $\mathcal{U}$ be the union of strata of $\M_{g+h,1+n}^{\textnormal{ct}}$ meeting $\partial \mathcal{N}_{\theta}$, and $\mathfrak{t}$ the set of root types $\tau'$ such that $\M_{g+h,1+n}^{\tau'} \cap \partial\mathcal{N}_{\theta} \neq \emptyset$. Then we have a decomposition
\[
\partial\mathcal{N}_{\theta} \subseteq \mathcal{U} = \bigsqcup_{\tau' \in \mathfrak{t}} \M_{g+h,1+n}^{\tau'},
\]
and $\tau := (h,0,1)$ is a minimal element of $\mathfrak{t}$. For any $\tau' \succeq \tau$ the space $\M_{g+h,1+n}^{\tau'}$ is the disjoint union of (the image under a gluing morphism of) the product  (modulo a finite automorphism group) of $\M_{g_1',1+\ell_1' + k_1'}$ (with $g_1' \geq h$) with other $k_1'$ moduli spaces of compact type, each of them having dimension $< d_{g,n}$ because the large genus $h$ was concentrated in the root. The induction hypothesis guarantees that the restrictions of $\Omega$ and $\Omega'$ to $\M_{g,1+n}^{\tau'}$ coincide for every $\tau' \succeq \tau$. From Lemma~\ref{lemma: the euler class is not a 0 div} we deduce that the restriction of $\Omega_{g+h,1+n}$ and $\Omega'_{g+h,1+n}$ to $\M_{g+h,1+n}^{\succeq \tau}$ coincide in cohomology degree $k \leq \frac{2h}{3}$, and this space contains $\mathcal{U}$ and thus $\partial\mathcal{N}_{\theta}$. Therefore, \eqref{Homcun} implies that $\Omega_{g,1+n} = \Omega_{g,1+n}'$. By induction, this holds for any $(g,n)$.
\end{proof}

In particular, Theorem~\ref{thm:recon} proves the transitivity of the F-Givental group action on invertible ctF-CohFTs having a given F-TFT.  To complete the proof of Theorem~\ref{thm:transfree} it remains to justify freeness.

\begin{lemma}
       The F-Givental group acts freely on the set of invertible ctF-CohFTs.
\end{lemma}
\begin{proof} 
By transitivity, it is enough to check that invertible F-TFTs have trivial stabilisers. Assume we have $R(z) \in \End(V)\llbracket z \rrbracket$ such that $R(0) = \Id_V$ and $T(z) \in z^2V\llbracket z \rrbracket$ such that $RT\omega = \omega$. For any $g \geq 1$ and $v \in V$, this identity on $\M_{g,1}$ evaluated on the vector $\alpha^{-g} \cdot v$ yields
\begin{equation}
\label{RTfree1}
R(-\psi_1)\big[\hat{T}(\boldsymbol{\kappa}) \cdot v\big] = v
\end{equation}
where $\hat{T}(\boldsymbol{\kappa}) = \exp(\sum_{m \geq 1} \hat{t}_m \kappa_m)$ is determined from $T(z)$ as in Lemma~\ref{expkappa}. Since \eqref{RTfree1} is valid for any $g$, the same identity holds in the completed stable cohomology where $\psi_1,\kappa_1,\kappa_2,\ldots$ are free generators. Evaluating $\boldsymbol{\kappa}$ to zero gives $R(z) = \Id_V$, while evaluating $\psi_1$ and all $\kappa$s to zero except $\kappa_m$ yields $\hat{t}_m = 0$, for each $m \geq 1$. Hence $\hat{T}(\boldsymbol{\kappa}) = \1$ and $T(z) = 0$.
\end{proof}

\section{{\Large Reconstruction from flat F-manifolds}}

\label{S4}
\medskip

\subsection{Flat F-manifolds}

\label{flat F-manifold, def section}
Flat F-manifold were introduced by Hertling and Manin in \cite{hertling_weak_1999}, see also \cite{getzler_jet-space_2004,manin_f_2005}. We summarise here the basic definitions and properties. Upper indices indicate components of vectors, lower indices components of linear forms. In particular, $A^{\mu}_{\nu}$ are entries of a matrix with row index $\mu$ and column index $\nu$. Einstein summation convention for Greek indices appearing in lower and upper position will be assumed, but repeated Latin indices are not summed over unless the sum is explicitly written.

\begin{definition}
    A \emph{flat F-manifold} $(M, \nabla, \cdot)$ is the datum of an analytic manifold $M$, an analytic affine connection $\nabla$ on $M$, and for each $p \in M$ the structure a bilinear product $\mathop{\cdot}_p$ depending analytically on $p$, such that
    \begin{equation}
    \label{deformcon}
   \forall v,w \in \Gamma(TM)\qquad  \nabla^z_vw = \nabla_v w + z^{-1} v \cdot w
    \end{equation}
    defines a family of flat torsion-free connections parameterised by $z \in \widehat{\mathbb{C}}^*$.
    \end{definition}

  This last condition is equivalent to requiring that $\nabla$ is flat and torsion-free, that the product $\cdot$ is commutative and associative, and that $\nabla_u(\cdot)(v \otimes w) = \nabla_v(\cdot)(u \otimes w)$ for any $u,v,w \in \Gamma(TM)$. Call $N$ the complex dimension of $M$, denote $t = (t^{\mu})_{\mu = 1}^{N}$ flat coordinates for $M$ (with respect to the connection $\nabla$), $(\partial_{\mu})_{\mu = 1}^{N}$ the corresponding coordinate vector fields and $c^\tau_{\mu\nu}$ the structure constants of the product in these coordinates. Then
  \[
  \forall \mu,\nu,\rho,\sigma \in [N]\qquad \nabla_{\partial_{\mu}} \partial_{\nu} = 0 \qquad \textnormal{and} \qquad (\nabla_{\partial_\mu} c)_{\nu \rho}^{\tau} = (\nabla_{\partial_\nu} c)_{\mu \rho}^{\tau}.
 \]
This second property together with the commutativity of the product is equivalent to the local existence of a vector potential $F = F^{\mu} \partial_{\mu}$ such that
\begin{equation}
\label{strcons}
\forall \mu,\nu,\rho \in [N]\qquad c_{\mu\nu}^{\rho} = \frac{\partial^2 F^{\rho}}{\partial t^{\mu} \partial t^{\nu}}.
\end{equation}
Associativity then translates into the \emph{oriented WDVV equations}
\begin{equation}
\label{asso}
\forall \mu,\nu,\rho,\lambda \in [N]\qquad \frac{\partial^2 F^{\mu}}{\partial t^{\nu} \partial t^{\beta}} \frac{\partial^2 F^{\beta}}{\partial t^{\rho} \partial t^{\lambda}} = \frac{\partial^2 F^{\mu}}{\partial t^{\rho} \partial t^{\beta}} \frac{\partial^2 F^{\beta}}{\partial t^{\nu} \partial t^{\lambda}}. 
\end{equation}
Conversely, if we have an analytic function $F : U \rightarrow \mathbb{C}^N$ defined in some open $U \subset \mathbb{C}^N$ with standard coordinates $t = (t^1,\ldots,t^N)$ satisfying \eqref{asso}, it defines a structure of flat F-manifold on $U$ with structure constants \eqref{strcons}. For instance, if $\Omega$ is a ctF-CohFT on a vector space $V$, its genus $0$ part defines a germ of flat F-manifold structure near $0 \in V$ with
\begin{equation}
\label{FCohFTpotential}
F(t) = \sum_{n \geq 2} \frac{1}{n!} \int_{\Mbar_{0,1+n}} \Omega_{0,1+n}(t^{\otimes n}),
\end{equation}
recalling that in genus zero $\Mct_{0,1+n} = \Mbar_{0,1+n}$. We can also define the \emph{formal shift} of $\Omega$, which is a germ near $0$ of a family of ctF-CohFTs on $V$, given for $g \geq 0$, $n \geq 0$ and $w \in V^{\otimes n}$ by
\begin{equation}
\label{formalt}{}_{t}\Omega_{g,1+n}(w) = \sum_{m \geq 0} \frac{1}{m!} (f_m)_* \Omega_{g,1+n+m}(w \otimes t^{\otimes m}),
\end{equation}
where $f_m : \Mct_{g,1+n+m} \rightarrow \Mct_{g,1+n}$ forgets the $m$ last marked points. Unlike the translation of Definition~\ref{Taction}, in general the series \eqref{FCohFTpotential} and \eqref{formalt} do not truncate and may not converge (the reason why we only talk about germs). Nevertheless, ${}_{t} \Omega$ satisfy the ctF-CohFT axioms order by order in $t$, and the equations \eqref{asso} are satisfied.

\begin{definition}
	A flat F-manifold \emph{with flat unit} $(M, \nabla, \cdot, \1)$ is the datum of a flat F-manifold $(M, \nabla, \cdot)$ and a unit vector field $\1_p$ for $\cdot$ depending analytically on the point $p\in M$ such that  $\nabla \1 = 0$.
\end{definition}
The condition $\nabla \1 = 0$ implies that $\1 = \1^{\mu} \partial_{\mu}$, where $\1^{\mu}$ is a scalar constant for every $\mu \in [N]$, and that the vector potential $F$ satisfies
\begin{equation}
	\label{unit}
	\1^{\beta} \frac{\partial^2 F^{\mu}}{\partial t^{\beta} \partial t^{\nu}} = \delta^{\mu}_{\nu}. 
\end{equation}
If a ctF-CohFT admits a flat unit, then it is flat as well for the flat F-manifold \eqref{FCohFTpotential}.
  
\subsection{Semi-simplicity and further constructions}
\label{sec:semisimpleF} In the rest of the article we will mostly work under a semi-simplicity assumption, which make many computations possible. We start by reminding the reader of the known implications of semi-simplicity for the geometry of the flat F-manifold.

\begin{definition}
A flat F-manifold is semi-simple at a point $p$ if the corresponding algebra $(T_pM,\cdot_p)$ is semi-simple. It is semi-simple if it is semi-simple at any point.
\end{definition}
A semi-simple flat F-manifold admits canonical coordinates $u = (u^{i})_{i = 1}^{N}$: the corresponding coordinate vector fields $(\partial_i)_{i = 1}^{N}$ satisfy
\begin{equation}
\label{prodcan}
\forall i,j \in [N]\qquad \partial_i \cdot \partial_j =  \delta_{ij} \partial_i.
\end{equation}
Note that, when an algebra is semi-simple, there always exists a unit. In canonical coordinates the unit takes the form $\1 = \sum_{i = 1}^{N} \partial_i$ but it is may not be flat. We will reserve Latin indices for tensor components in canonical coordinates; as we see in \eqref{prodcan} and unlike the convention for Greek indices, we refrain from summing repeated Latin indices. 

Let us introduce the Christoffel symbols of the connection in canonical coordinates
\[
\nabla_{\partial_i} \partial_j = \sum_{k = 1}^{N} \Gamma_{ij}^{k} \partial_k
\]
The zero-torsion property translates into $\Gamma_{ij}^{k} = \Gamma_{ji}^{k}$, but further symmetries of Christoffel symbols are implied by the axioms of flat F-manifolds.
\begin{proposition} \cite{arsie_darbouxegorov_2013,arsie_semisimple_2023}
\label{prop:H}
There exists a matrix-valued function $\gamma = (\gamma^{i}_{j})_{1 \leq i,j \leq N}$ with zero diagonal and there exists locally scalar functions $H^1,\ldots,H^N$ such that, for any pairwise distinct $i,j,k \in [N]$ we have
\begin{equation}
\label{chissy}
\Gamma_{ij}^{k} = 0, \qquad \gamma^{i}_j =
 \frac{\partial \log H^i}{\partial u^j} = \Gamma_{ij}^{i} = \Gamma_{ji}^{i} = - \Gamma_{jj}^{i},\qquad \frac{\partial \log H^i}{\partial u^i} = \Gamma_{ii}^{i}.
 \end{equation}
Moreover, if our flat F-manifold admits a flat unit, then we also have: 
\[
\frac{\partial \log H^i}{\partial u^i} = \Gamma_{ii}^{i} = - \sum_{j = 1}^{N}\gamma^{i}_j.
\]
\end{proposition}
For each $i \in [N]$ the function $H^i$ is uniquely characterised by \eqref{chissy} up to multiplication by a non-zero constant. We define the metric\footnote{The metric is automatically compatible with the product since it is diagonal: for any $i,j,k \in [N]$ we have $\eta(\partial_i \cdot \partial_j,\partial_k) = \delta_{i,j,k} (H^i)^{-2} = \eta(\partial_i,\partial_j \cdot \partial_k)$. Hence $(V,\cdot,\eta)$ is a Frobenius algebra. However, it does not define a Frobenius manifold structure on $M$. Indeed, $\eta$ may not be flat for $\nabla$ and the properties \eqref{chissy} do not say that $\nabla$ is the Levi--Civita connection of $\eta$ (the structure of Christoffel symbols does not in general match the one of a diagonal metric).}
\begin{equation}
\label{metriceta}
\eta := \sum_{i = 1}^N (H^i\dd u^i)^{\otimes 2}
\end{equation}
The vector fields $\tilde{\partial}_i = (H^i)^{-1} \partial_i$ then form the orthonormal canonical basis, \textit{i.e.} satisfy
\[
\forall i,j \in [N]\qquad \eta(\tilde{\partial}_i,\tilde{\partial}_j) = \delta_{i,j} \qquad \tilde{\partial}_i \cdot \tilde{\partial}_j = \frac{\delta_{i,j}}{H^i}\, \tilde{\partial}_i.
\]
\begin{definition}
The change of basis from the canonical to the flat basis of vector fields is denoted $\Psi_{\beta}^{i} := \frac{\partial u^i}{\partial t^{\beta}}$. In other words
\[
\forall \beta \in [N]\qquad \partial_{\beta} = \sum_{i = 1}^{N} \Psi^i_{\beta} \partial_{i}.
\]
We also introduce $\tilde{\Psi}^i_{\beta} = H^i \Psi^i_{\beta}$, which is the change of basis from the orthonormal canonical basis to the flat basis\footnote{\label{cornot}Our convention is to consistently use tilde for objects involving the orthonormal canonical basis and no tilde for objects involving the canonical basis (we will carry all our computations with the latter). In particular, $\tilde{\Psi}$ and $\tilde{\Gamma}$ in \cite{arsie_semisimple_2023} are our $\Psi$ and $\gamma$, and vice versa.}.
\end{definition}

With these objects  we can write down the connection $\nabla^z$ on $TM$ from \eqref{deformcon}, as well as the dual connection $\nabla^{*,z}$ on $T^*M$. For this purpose we use the canonical basis to represent vector fields as column vectors, $1$-forms as line vectors, sections of $\End(TM)$ as matrices, and let the differential $\dd$ act entrywise on them. Introduce the matrices
\begin{equation}
\label{grago}
U = \textnormal{diag}(u^1,\ldots,u^N)\qquad \textnormal{and}\qquad  H = \textnormal{diag}(H^1,\ldots,H^N)
\end{equation}
representing sections of $\textnormal{End}(TM)$ that are diagonal on the canonical basis. Then, the deformed connection \eqref{deformcon} acting on a vector field $X$ reads
\begin{equation}
\label{conntriv}
\nabla^{z} X = \dd X - (\dd \Psi)\Psi^{-1} X + z^{-1}(\dd U) X,
\end{equation}
and the dual connection acting on a $1$-form $L$ is
\begin{equation}
\label{dualconntriv}
\nabla^{*,z} L = \dd L +  L (\dd \Psi) \Psi^{-1} - z^{-1} L(\dd U).
\end{equation}
\begin{lemma} \cite[Propositions 1.4 and 1.5]{arsie_semisimple_2023}
\label{lem:gadi} We have $(\dd \Psi) \Psi^{-1} = -(\dd H) H^{-1} + \textnormal{\textbf{[}}\gamma,\dd U\textnormal{\textbf{]}}$.
\end{lemma}

\subsection{Differential equations for semi-simple ctF-CohFTs: results}
\label{sec:resdiff}
Given an invertible ctF-CohFT $\Omega$ with associated F-TFT $\omega$, Theorem~\ref{thm:transfree} provides unique $R(z) \in \End(V)\llbracket z \rrbracket$ and $T(z) \in z^2 V\llbracket z \rrbracket$ such that $R(0) = \Id_V$ and $\Omega = RT\omega$. Assuming that the formal shift \eqref{formalt} of the ctF-CohFT has non-zero radius of convergence, we can find a contractible open neighborhood\footnote{This is unnecessary if we allow ourselves to work over $\textnormal{Spec}\,\mathbb{C}\llbracket t^1,\ldots,t^N\rrbracket$; the differential equations are then true order by order in $t^{\alpha}$.} $M \subset V$ of $0$ such that ${}_t \Omega$ remains invertible for $t \in M$ and $M$ is a flat $F$-manifold with vector potential \eqref{FCohFTpotential}. Therefore, we have $\omega$, $R(z)$ and $T(z)$ depending on $t$ as well (we use the flat connection to identify $T_pM$ to $V$ for any $p \in M$) and we can look for differential equations they may satisfy. To shorten notations, we will drop the subscript $t$.

On the one hand, by construction of the formal shift using the flat basis, we have for any $g,n \geq 0$ such that $2g - 1 + n > 0$ and $\nu,\mu_1,\ldots,\mu_n \in [N]$ that
\begin{equation}
\label{nablaomega}
\nabla_{\partial_{\nu}}\Omega_{g,1+n}(\partial_{\mu_1} \otimes \cdots \otimes \partial_{\mu_n}) = f_*\Omega_{g,1+n+1}(\partial_{\mu_1} \otimes \cdots \otimes \partial_{\mu_n} \otimes \partial_\nu),
\end{equation}
where $f : \Mct_{g,1+n+1} \rightarrow \Mct_{g,1+n}$ forget the last marked point. Since $f^{-1}(\M_{g,1+n}^{\textnormal{ct}}) \subseteq \M_{g,1+n+1}^{\textnormal{ct}}$, there is no harm in replacing $\Omega$ by $RT\omega$ in \eqref{nablaomega}. Independently, we can use a trivialisation of $\nabla$ to act on tensor fields like $RT\omega_{g,1+n}$. Here, assuming semi-simplicity, there is a clear advantage in using the canonical coordinates to trivialise $\nabla$, because the product (involved in the F-TFT $\omega$) has structure constants $0$ or $1$ independently of $t$. For instance, \eqref{conntriv} and \eqref{dualconntriv} gave the action of $\nabla$ on vector fields and $1$-forms represented in the canonical basis. Comparing the two approaches yields differential equations for $R$ and $T$, allowing to relate them to the constituents of the flat F-manifolds met in Section~\ref{sec:semisimpleF}.

Doing so for $\nabla \omega_{0,1+2}$ and $\nabla \omega_{1,1} = \nabla \alpha$, we first obtain formulae for the first-order coefficients in 
\[
R(z) = \Id_V + R_1 z + O(z^2),\qquad \hat{T}(z) = \1 + \hat{t}_1z + O(z^2).
\]

\begin{lemma}
\label{lem:alphaHR1}
Let $\Omega$ be an invertible semi-simple ctF-CohFT. The R- and T-elements of the F-Givental group associated to its formal shift by Theorem~\ref{thm:transfree} satisfy
\[
\forall i,j \in [N]\qquad (R_1)^j_i - (\hat{t}_1)^i \delta_i^j = \frac{\partial \log H^j}{\partial u^i},\qquad \dd \big(H^i \sqrt{\alpha^i}\big) = 0.
\]
In particular, in canonical coordinates $R_1 - \gamma$ is a diagonal matrix and $\textnormal{\textbf{[}}R_1,\dd U\textnormal{\textbf{]}} = (\dd \Psi)\Psi^{-1} + (\dd H)H^{-1}$.
\end{lemma}

This shows that, if we normalise $H^i$ to have $H^i \sqrt{\alpha^i} = 1$ at $t = 0$ for every $i \in [N]$, Lemma~\ref{lem:alphaHR1}, it remains so for all $t$. We give an interpretation of the corresponding metric $\eta$ in Appendix~\ref{appA}. Lemma~\ref{lem:alphaHR1} is instrumental in obtaining the sought-for differential equations for $R(z)$ and $T(z)$. Commutators of endomorphism will be denoted $\textbf{[}\,,\textbf{]}$ to distinguish them for $[ \cdot ]$ which are used for the evaluation of an endomorphism on a vector.

\begin{proposition}
\label{prop:diffTR}
Let $\Omega$ be an invertible semi-simple ctF-CohFT. Then $R(z)$ and $T(z)$ for its formal shift satisfy
\begin{equation*}
\begin{split}
& \quad \dd R(z) - R(z) (\dd H)H^{-1} - (\dd \Psi)\Psi^{-1} R(z) + z^{-1}\textnormal{\textbf{[}}R(z),\dd U \textnormal{\textbf{]}} = 0, \\
& \dd \hat{T}(z) - (\dd H)H^{-1}\hat{T}(z) + z^{-1}\hat{T}(z) \cdot \big(\Id - \hat{T}(z) \cdot R^{-1}(z)\big) \dd \overline{U} = 0.
 \\
\end{split}
\end{equation*}
In these formulae we used the canonical basis to consider $\hat{T}(z)$ as a column vector and $R(z)$ as a matrix. $\dd \overline{U}$ is the column vector $(\dd u^1,\ldots,\dd u^N)^{\textnormal{T}}$ and $\cdot$ is the F-TFT product of vectors\footnote{The differential equation for $R$ involves only usual matrix products, not the F-TFT product.}.
\end{proposition}

These equations can be transformed, giving a more geometric meaning to $R(z)$ and to the vacuum vector $\Upsilon(z) = R(z)[\hat{T}^{-1}(z)]$ that we already met in Proposition~\ref{prop:RTflatunit}.

\begin{corollary}\label{cor:RTfial}
Let $\Omega$ be an invertible semi-simple ctF-CohFT. Denote $\xi_j(z)$ the vector field whose components in the canonical basis are given by the $j$-th column of $R(z)H^{-1}e^{U/z}$, that is
\begin{equation}
\label{xijpsi} \xi_j(z) = (\Psi^{-1}R(z)H^{-1}e^{U/z})_j^{\mu} \partial_\mu.
\end{equation}
Then, $(\xi_j(z))_{j = 1}^N$ is a basis of flat sections for $\nabla^{-z}$, that is $\nabla^{-z} \xi_j = 0$ for any $j \in [N]$. Furthermore
\begin{equation}
\label{vaceqn}
\Upsilon(z) := R(z)[\hat{T}^{-1}(z)] = \sum_{m \geq 0} z^m \nabla_{\1}^{m}(\1).
\end{equation}
\end{corollary}

Note the minus sign in front of $z$. The $H^{-1}$ on the right-hand side in \eqref{xijpsi} would be absent if we had read $R(z)$ in the orthonormal canonical basis, and in that case there would be have been $\tilde{\Psi}^{-1}$ instead of $\Psi^{-1}$ on the left. All results will be proved in the next Section~\ref{sec:proofnu} and establish Theorem~\ref{thm:rec}, but a few comments are in order. In Lemma~\ref{lem:alphaHR1} we see that only the off-diagonal part of $R_1$ is determined from the flat F-manifold. In Proposition~\ref{prop:diffTR}, the differential equation for $R(z)$ only determines it uniquely up to right-multiplication by $\exp(\sum_{k \geq 1} D_k z^k)$, where $D_k$ are constant diagonal matrices (this covers the diagonal ambiguity in $R_1$). The odd part of this ambiguity comes from the vanishing of Hodge classes in genus $0$, while the even part comes from another vanishing relation in cohomology mentioned in Section~\ref{Sec1}. In contrast, $\Upsilon(z)$ is uniquely determined by the flat F-manifold.

Given a semi-simple flat F-manifold, \cite[Proposition 1.6]{arsie_semisimple_2023} 
rather considers\footnote{See footnote \ref{cornot} for the correspondence of notations.\cite[Proposition 1.6]{arsie_semisimple_2023} is announced assuming a flat unit, but this assumption is not used in the proof.}
 a matrix $\mathsf{R}(z)$ solving the differential equation
\begin{equation}
\label{dRRossi}
\dd \mathsf{R}(z) + \mathsf{R}(z) (\dd \tilde{\Psi}) \tilde{\Psi}^{-1} + z^{-1}\textbf{[} \dd U,\mathsf{R}(z)\textbf{]} = 0
\end{equation}
The ambiguity on $\mathsf{R}(z)$ is now by left-multiplication. Since $\tilde{\Psi} = H\Psi$ this can be rewritten
\begin{equation}
\label{mathsfR}
\dd(\mathsf{R}(z)H) + \mathsf{R}(z)H (\dd \Psi)\Psi^{-1} + z^{-1}\textbf{[}\dd U,\mathsf{R}(z)\textbf{]} H= 0,
\end{equation}
which allows a more direct comparison with \eqref{dualconntriv}. This equation amounts to saying that the $1$-forms whose components in the dual canonical basis give the lines of the matrix $e^{U/z}\mathsf{R}(z)H$, that is
\begin{equation}
\label{dualflatsec}
L^{i}(z) =  (e^{U/z}\mathsf{R}(z)H\Psi)^i_{\mu} \dd t^{\mu} \qquad i \in [N].
\end{equation}
form a basis of flat sections of $\nabla^{z,*}$. An equivalent statement (compare \eqref{conntriv} and \eqref{dualconntriv}) is that the vector fields whose components in the canonical basis give the columns of $(e^{U/z}\mathsf{R}(z)H)^{-1}$ form a basis of flat sections of $\nabla^{z}$.  Accordingly, the correspondence 
\begin{equation}
\label{Rcorres}
R(z) = H^{-1}\mathsf{R}(-z)^{-1}H
\end{equation}
transforms \eqref{dRRossi} of \cite{arsie_semisimple_2023} into the equation for $R(z)$ found in Proposition~\ref{prop:diffTR}.

\subsection{Differential equations for semi-simple ctF-CohFTs: proofs}

\label{sec:proofnu}

The derivation of differential equations combine the geometry of tautological classes in the moduli space of curves (the strategy was outlined at the beginning of Section~\ref{sec:resdiff}) with the geometry of the flat F-manifold (reviewed in Section~\ref{sec:semisimpleF}).  Although nothing prevents us applying this strategy to any invertible ctF-CohFT, there are many simplifications when we can use a basis of vector fields making the product constant. This is where the semi-simplicity assumption becomes handy.

\begin{proof}[Proof of Lemma~\ref{lem:alphaHR1}]
Since $\Psi$ is the change of basis from canonical to flat, we have
\[
\forall \beta,\gamma \in [N]\qquad \partial_{\beta} \cdot \partial_{\gamma} = \sum_{j = 1}^{N} \Psi_{\beta}^j \Psi_{\gamma}^j \partial_j = \sum_{j = 1}^{N} \Psi_{\beta}^j \Psi_{\gamma}^{j} \Psi_{j}^{\epsilon}\partial_{\epsilon},
\]
where we wrote $\Psi_{i}^{\beta}$ for the matrix elements of $\Psi^{-1}$. Differentiating this equation we get
\[
\forall i,\beta,\gamma \in [N]\qquad \nabla_{\partial_i}\big(\partial_{\beta} \cdot \partial_{\gamma}\big) = \sum_{j = 1}^{N} \left( \frac{\partial \Psi_{\beta}^{j}}{\partial u^i} \Psi_{\gamma}^j  + \Psi_\beta^j  \frac{\partial \Psi_{\gamma}^j}{\partial u^i} - \sum_{k = 1}^{N} \Psi_{\beta}^k \Psi_{\gamma}^k  \frac{\partial \Psi_\rho^j}{\partial u^i} \Psi^{\rho}_k \right) \partial_{j}.
\]
Multiplying by $\Psi_{b}^{\beta} \Psi_{c}^{\gamma}$ and summing over $\beta,\gamma,i$, we arrive for every $i,b,c \in [N]$ to the identity
\begin{equation}
\label{equatelet}
\begin{split}
\Psi_{b}^{\beta} \Psi_{c}^{\gamma}\nabla_{\partial_{i}}\big(\partial_{\beta} \cdot \partial_{\gamma}\big) & = \frac{\partial \Psi_{\beta}^c}{\partial u^i} \Psi_{b}^{\beta} \partial_c  + \frac{\partial \Psi^b_{\gamma}}{\partial u^i} \Psi_{c}^{\gamma}\partial_b - \sum_{j = 1}^{N} \delta_{b,c} \frac{\partial \Psi_{\rho}^j}{\partial u^i} \Psi_{b}^{\rho} \partial_j \\
& = \left(\frac{\partial \Psi}{\partial u^i} \Psi^{-1}\right)_{b}^{c} \partial_c + \left(\frac{\partial \Psi}{\partial u^i} \Psi^{-1}\right)_c^b \partial_b -  \sum_{j = 1}^{N} \delta_{b,c} \left(\frac{\partial \Psi}{\partial u^i} \Psi^{-1}\right)_{b}^{j} \partial_j.
\end{split}
\end{equation}

\begin{figure}[h!]
\begin{center}
\includegraphics[width=0.7\textwidth]{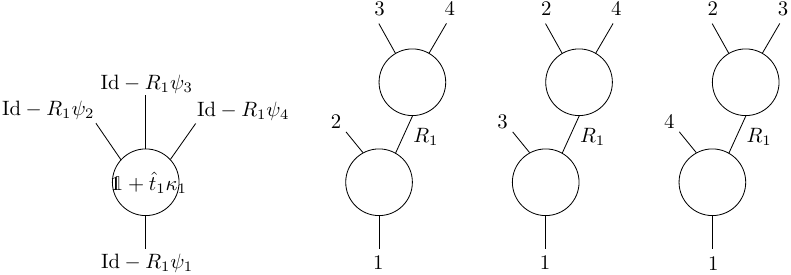}
\caption{\label{fig:F04strat} Contributions to $(RT\Omega)_{0,1+3}$. The stratum on the left has complex dimension $1$ so we just need to linearise in $\psi$ and $\kappa$. The strata on the right have dimension $0$ so only remains the F-TFT product at vertices and $E_R(0,0) = R_1$ on the edge (see \eqref{Wweight}).}
\end{center}
\end{figure}

On the other hand, from the definition of the formal shift we can compute the covariant derivative in the direction of a flat basis vector
\[
\forall \mu,\beta,\gamma \in [N]\qquad \nabla_{\partial_{\mu}} (\partial_{\beta} \cdot \partial_{\gamma}) =  f_*\big({}_{t}\Omega_{0,1+3}(\partial_{\mu} \otimes \partial_{\beta} \otimes \partial_{\gamma})\big),
\]
where $f : \Mbar_{0,1+3} \rightarrow \Mbar_{0,1+2}$ forgets the last marked point. Inserting the change of basis we get
\[
\forall i,b,c \in [N]\qquad \Psi_{b}^{\beta} \Psi_c^{\gamma} \nabla_{\partial_{i}} (\partial_{\beta} \cdot \partial_{\gamma})  = f_*\big({}_{t} \Omega_{0,1+3}(\partial_i \otimes \partial_b \otimes \partial_c)\big).
\]
We can compute $\Omega_{0,1+3} = (RT\omega)_{0,1+3}$ from the F-Givental group action. There are four stable trees in $T_{0,1+3}$ (Figure~\ref{fig:F04strat}). The tree with a single vertex  has contribution
\[
\omega_{0,1+3} + \kappa_1 \hat{t}_1 \cdot \omega_{0,1+3} - \psi_1 R_1 \circ \omega_{0,1+3} - \omega_{0,1+3} \circ (\psi_2 R_1 \otimes \Id^{\otimes 2} + \psi_3 \Id \otimes R_1 \otimes \Id + \psi_4 \Id^{\otimes 2} \otimes R_1).
\]
Applying $f_*$ kills the degree $0$ term and replace the degree $2$ classes with $\int_{\Mbar_{0,1+3}} \kappa_1 = \int_{\Mbar_{0,1+3}} \psi_l = 1$ for $l \in [4]$.  The three trees with two vertices related by an edge differ by the choice of a label $2,3$ or $4$ of the ingoing leaf on the root vertex. If $2$ is chosen, the contribution to $\Omega_{0,1+3}$ is $\textnormal{gl}_*\omega_{0,1+2} \circ (\textnormal{Id} \otimes R_1) \circ \omega_{0,1+2}$; the contributions from the two other trees are obtained by suitable permutation of inputs. For any $i,b,c \in [N]$ we arrive to
\begin{equation*}
\begin{split}
\Psi_{b}^{\beta} \Psi_c^{\gamma} \nabla_{\partial_{i}} (\partial_{\beta} \cdot \partial_{\gamma}) & =  \delta_{b,c,i}\big(\hat{t}_1 \cdot \partial_i - R_1(\partial_i)\big) + \partial_i \cdot R_1(\partial_b \cdot \partial_c)- R_1(\partial_i) \cdot \partial_b \cdot \partial_c \\
& \quad + \partial_{b} \cdot R_1(\partial_i \cdot \partial_c)  - R_1(\partial_b) \cdot \partial_i\cdot \partial_c + \partial_c \cdot R_1(\partial_i \cdot \partial_b) - R_1(\partial_c) \cdot \partial_i \cdot \partial_b.
\end{split}
\end{equation*}
Equating this to \eqref{equatelet}, specialising to $i = b = c$ and extracting the coefficient of $\partial_j$, we obtain
\[
\forall i,j \in [N]\qquad \delta_{i}^j \bigg(\frac{\partial \Psi}{\partial u^i} \Psi^{-1}\bigg)_{i}^{j} - (1 - \delta_i^j)\bigg(\frac{\partial \Psi}{\partial u^i} \Psi^{-1}\bigg)_i^j = (\hat{t}_1)^i \delta_i^j - (R_1)_i^j.
\]
As the left-hand side is computed by Lemma~\ref{lem:gadi} in terms of $H$, we find
\begin{equation}
\label{watwe}
\begin{split}
(i = j) & \qquad \frac{\partial \log H^i}{\partial u^i} = (R_1)_i^i - (\hat{t}_1)^i, \\
(i \neq j) & \qquad \frac{\partial \log H^j}{\partial u^i} = \gamma_i^j = (R_1)_i^j .
\end{split}
\end{equation}
We conclude that $R_1 - \gamma$ is diagonal, hence commutes with $\dd U$. This allows rewriting Lemma~\ref{lem:gadi} as
\begin{equation}
\label{dPsidH}
(\dd \Psi)\Psi^{-1} = -(\dd H)H^{-1} + \textbf{[}\gamma,\dd U \textbf{]} = -(\dd H)H^{-1} + \textbf{[}R_1,\dd U \textbf{]}.
\end{equation}

We now turn to the covariant derivative of $\alpha$. On the one hand, representing $\alpha$ in the canonical basis as a column vector (see \eqref{conntriv} without the $z^{-1}$ term) and using \eqref{dPsidH} we have
\begin{equation}
\label{alphaun}
\nabla \alpha = \dd \alpha - (\dd \Psi) \Psi^{-1}\alpha = \dd \alpha + (\dd H)H^{-1} \alpha - \textbf{[}\gamma,\dd U\textbf{]}\alpha = \dd \alpha + (\dd H)H^{-1}\alpha - \textbf{[}R_1,\dd U\textbf{]}\alpha.
\end{equation}
On the other hand, we have
\[
\forall \nu \in [N]\qquad \nabla_{\partial_{\nu}} \alpha = \big(f_* (RT\omega)_{1,1+1}(\partial_{\nu})\big)^{\textnormal{deg}\,0} = f_*\big((RT\omega)_{1,1+1}^{\textnormal{deg}\,2}(\partial_\nu)\big)
\]
as the forgetful map $f : \Mct_{1,1+1} \rightarrow \Mct_{1,1}$ has fibers of complex dimension $1$. We compute from the F-Givental group action (Figure~\ref{fig:alphastrat})
\begin{equation}
\label{regun11}
(RT\omega)_{1,1+1}(\partial_\nu) = R(-\tilde{\psi}_1)\big[\alpha \cdot \hat{T}(\tilde{\boldsymbol{\kappa}}) \cdot R^{-1}(\tilde{\psi}_2)[\partial_\nu]\big] + \textnormal{gl}_*\big(\partial_{\nu} \cdot E_R(0,\psi')[\alpha \cdot \hat{T}(\boldsymbol{\kappa})]\big).
\end{equation}
with the edge weight given by the specialisation of \eqref{Wweight}
\[
E_R(0,\psi') = \frac{\Id_V - R(-\psi')}{\psi'} = R_1 + O(\psi').
\]
Tilde refer to classes on $\Mct_{1,1+1}$, no tilde to classes on $\Mct_{1,1}$, and they are related by Lemma~\ref{lem39}. We only need the relations in degree $2$
\[
\tilde{\psi}_1 = f^*\psi_1 + p_{1*}\mathbf{1},\qquad \qquad \tilde{\kappa}_1 = f^*\kappa_1 + \tilde{\psi}_2.
\]
Here $\textbf{1}$ is the fundamental class. Extracting the degree $2$ part of \eqref{regun11} yields
\begin{equation}
(RT\omega)_{1,1+1}^{\textnormal{deg}\,2}(\partial_\nu) = \alpha \cdot \big((f^*\kappa_1 + \tilde{\psi}_2)\hat{t}_1 \cdot  \partial_\nu - \tilde{\psi}_2 R_1[\partial_\nu]\big)  - (f^* \psi_1 + p_{1*}\textbf{1})R_1[\alpha \cdot \partial_\nu] + \textnormal{gl}_*\big(\partial_{\nu} \cdot R_1[\alpha]\big).
\end{equation}
Applying $f_*$ kills the two $f^*$ terms, and noting that $f_* \tilde{\psi}_2 = \kappa_0 = \textbf{1}$ on $\Mct_{1,1}$ and $f \circ p_1 = \textnormal{id}$, we get
\[
\nabla_{\partial_{\nu}} \alpha = \alpha \cdot \big( \hat{t}_1 \cdot \partial_\nu - R_1(\partial_{\nu})\big) - R_1[\alpha \cdot \partial_\nu] + \partial_\nu \cdot R_1[\alpha].
\]
Rewriting this in the canonical basis and inserting \eqref{watwe} for the first term, we arrive to
\[
\nabla \alpha = -(\dd H)H^{-1} \alpha -\textbf{[}R_1,\dd U\textbf{]} \alpha.
\]
Comparing with \eqref{alphaun} we find $(\dd \alpha)\cdot \alpha^{-1} + 2(\dd H)H^{-1} = 0$, that is $\dd (H^i \sqrt{\alpha^i}) = 0$ for any $i \in [N]$.
\end{proof}

\begin{figure}
\begin{center}
\includegraphics[width=0.3\textwidth]{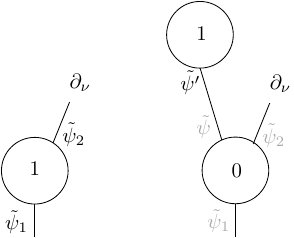}
\caption{\label{fig:alphastrat} The stable trees in $T_{1,1+1}$.} 
\end{center}
\end{figure}

\begin{proof}[Proof of Proposition~\ref{prop:diffTR}]
We can access $R(z)$ and $T(z) = z(\1 - \hat{T}^{-1}(z))$ simultaneously\footnote{Another route leading to the same result for $T(z)$ is to  come back to its definition in Proposition~\ref{prop:Zkappa}, \textit{i.e.} look at the large genus behavior of $\Omega_{g,1}{}_{|\M_{g,1}}$.} by examining for large $g$ the $\End(V)$-valued cohomology class $\Omega_{g,1+1}^{\circ}$, or equivalently the restriction of $\Omega_{g,1+1}$ to $\M_{g,1+1}$. For a fixed $g \geq 1$, Corollary~\ref{coformsmooth} says that
\begin{equation}
\label{starung}
\Omega_{g,1+1}{}_{|\M_{g,1+1}}  = R(-\psi_1) \circ (\hat{T}_g(\boldsymbol{\kappa}) \cdot) \circ R^{-1}(\psi_2),
\end{equation}
where $\hat{T}_g(\boldsymbol{\kappa}) = \alpha^g \cdot \hat{T}(\boldsymbol{\kappa})$. If we truncate up to a given cohomological degree, taking $g$ large enough psi and kappa classes become free (Theorem~\ref{Mumfordwithpsiclasses}). So, we can define two specialisations:
\begin{itemize}
\item $\textnormal{sp}_R$ takes $\boldsymbol{\kappa} = 0$, $\psi_1 = -z$ and $\psi_2 = 0$;
\item $\textnormal{sp}_T$ takes $\psi_1 = 0$, $\kappa_m = z^m$ and $\psi_2 = 0$.
\end{itemize}
They are such that
\begin{equation}
\label{starung2}
\textnormal{sp}_R\,\Omega_{g,1+1}{}_{|\M_{g,1+1}} = R(z) \circ (\alpha^{g} \cdot),\qquad \textnormal{sp}_T\,\Omega_{g,1+1}{}_{|\M_{g,1+1}} = \alpha^g \cdot \hat{T}(z) \cdot.
\end{equation}
These equations hold as well for the formal shift of the ctF-CohFT, and formal shifting commutes with specialising. We want to compute the covariant derivative of \eqref{starung2} in two independent ways.

\noindent \textsc{Step 1: Direct computation.} Representing endomorphisms by matrices in the canonical basis, in particular letting $A = \textnormal{diag}(\alpha^1,\ldots,\alpha^N)$ represent the multiplication by $\alpha$, we have
\begin{equation}
\label{ysho}
\begin{split}
\textnormal{sp}_R \,\nabla \Omega_{g,1+1}{}_{|\M_{g,1+1}} & = (\dd R(z)) A^{g} + g R(z)A^g(\dd A)A^{-1} + \textbf{[}R(z)A^g,(\dd \Psi)\Psi^{-1}\textbf{]} \\
& = (\dd R(z))A^g - 2g R(z)A^g (\dd H)H^{-1} + \textbf{[}R(z)A^g,(\dd \Psi)\Psi^{-1}\textbf{]} .
\end{split}
\end{equation}
The second equality comes from Lemma~\ref{lem:alphaHR1}. Likewise
\begin{equation}
\label{ysho2}\textnormal{sp}_T \,\nabla \Omega_{g,1+1}{}_{|\M_{g,1+1}} =  A^g \dd (\hat{T}(z) \cdot) -2g(\dd H)H^{-1} A^{g}\hat{T}(z)\cdot +\, \,\textbf{[}A^g \hat{T}(z) \cdot,(\dd \Psi)\Psi^{-1}\textbf{]}.
\end{equation} 
We stress that $\hat{T}(z)$ is a column vector, while $\hat{T}(z) \cdot$ is an endomorphism hence represented by a matrix.

\textsc{Step 2: Geometric computation (before specialisation).} By definition of the formal shift
\begin{equation}
\label{derOg11}
\forall \mu,\nu \in [N]\qquad \nabla_{\partial_{\nu}}\big( \Omega_{g,1+1}(\partial_{\mu})_{|\mathcal{M}_{g,1+1}}\big) = f_*\big(\Omega_{g,1+2}(\partial_{\mu} \otimes \partial_{\nu})\big)_{|\mathcal{M}_{g,1+1}}.
\end{equation}
We already carried out a similar computation in the proof of Lemma~\ref{lem:alphaHR1} for $\Omega_{1,1+1}$ instead of $\Omega_{g,1+2}$, so the ingredients will be familiar. First, we can replace $\Omega_{g,1+2}(\partial_{\mu} \otimes \partial_{\nu})$ by $(RT\omega)_{g,1+2}$, and when we apply $f_*$ only the stable trees corresponding to the strata in
\begin{equation}
\label{f1Mg1}
f^{-1}(\M_{g,1+1}) = \M_{g,1+2} \sqcup \mathcal{S} \sqcup \mathcal{S}' \subseteq \M_{g,1+2}^{\textnormal{ct}}
\end{equation}
contribute. If classes in $\M_{g,1+2}$ have tilde and those in $\M_{g,1+1}$ do not, using the canonical identification $\M_{1,1+1}\times \M_{0,3} \cong \M_{1,1+1}$ and setting $\hat{T}_g = \alpha^g \cdot \hat{T}$ we find
\begin{equation}
\label{uoohu}
\begin{split}
& \quad \Omega_{g,1+2}(\partial_{\mu} \otimes \partial_{\nu}) \\
& = R(-\tilde{\psi}_1)\big[\hat{T}_{g}(\tilde{\boldsymbol{\kappa}} ) \cdot R^{-1}(\tilde{\psi}_2)[\partial_{\mu}] \cdot R^{-1}(\tilde{\psi}_3)[\partial_{\nu}]\big] + \cdots \\
& \quad +\textnormal{gl}_*\Big(R(-\psi_1)\big[\hat{T}_g(\boldsymbol{\kappa}) \cdot E_R(\psi,0)[\partial_{\mu} \cdot \partial_{\nu}]\big]\Big) + \textnormal{gl}_*\Big(\partial_{\nu} \cdot E_R(0,\psi')\big[\hat{T}_g(\boldsymbol{\kappa}) \cdot R^{-1}(\psi_2)[\partial_{\mu}]\big]\Big),
\end{split}  
\end{equation} 
where $\cdots$ will eventually be projected to zero. The relevant specialisations of the edge weight \eqref{Wweight} are
\begin{equation}
\label{WWeight2}
E_R(\psi,0) = \frac{\Id_V - R^{-1}(\psi)}{\psi} \qquad \textnormal{and} \qquad E_R(0,\psi') = \frac{\Id_V - R(-\psi')}{\psi'}.
\end{equation}
Comparing the tilde and non-tilde classes with help of Lemma~\ref{lem39}, we find
\begin{equation*}
\begin{split}
R(-\tilde{\psi}_1)  & = R(-f^*\psi_1) - p_{1*} E_{R}(0,\psi_1),\\
R^{-1}(\tilde{\psi}_2) & = R^{-1}(f^*\psi_2) - p_{2*} E_R(\psi_2,0), \\
\hat{T}_g(\tilde{\boldsymbol{\kappa}})  & = \alpha^g \cdot \hat{T}(f^*\boldsymbol{\kappa}) \cdot \hat{T}(\tilde{\psi}_3).
\end{split}
\end{equation*}
In $f^{-1}(\M_{g,1+1})$ the section $p_1$ is supported on $\mathcal{S}'$  (where $\tilde{\psi}_1,\tilde{\psi},\tilde{\psi}_3$ restrict to zero) while the section $p_2$ is supported on $\mathcal{S}$ (where $\tilde{\psi}',\tilde{\psi}_2,\tilde{\psi}_3$ restrict to zero), and these two strata are disjoint. Applying $f_*$ to the first term in \eqref{uoohu} then yields
\begin{equation}
\label{nterm}
\begin{split}
& \quad  \quad f_*\Big(R(-f^*\psi_1)\big[\hat{T}_g(f^*\boldsymbol{\kappa}) \cdot R^{-1}(f^*\psi_2)[\partial_\mu] \cdot \hat{T}(\tilde{\psi}_3) \cdot R^{-1}(\tilde{\psi}_3)[\partial_\nu]\big]\Big) \\
& \quad - f_*\Big(p_{1*}E_R(0,\psi_1)\big[ \hat{T}_g(\boldsymbol{\kappa}) \cdot R^{-1}(\tilde{\psi}_2)[\partial_\mu] \cdot \partial_\nu\big] \Big) - f_*\Big(R(-\tilde{\psi}_1)\big[\hat{T}_g(\boldsymbol{\kappa}) \cdot  p_{2*} E_R(\psi_2,0)[\partial_\mu] \cdot \partial_{\nu}\big]\Big) \\
& = \quad R(-\psi_1)\big[\hat{T}_g(\boldsymbol{\kappa}) \cdot R^{-1}(\psi_2)[\partial_\mu] \cdot f_*\big(\hat{T}(\tilde{\psi}_3) \cdot R^{-1}(\tilde{\psi}_3)[\partial_\nu]\big)\big] \\
& \quad - E_R(0,\psi_1)\big[\hat{T}_g(\boldsymbol{\kappa}) \cdot R^{-1}(\psi_2)[\partial_\mu] \cdot \partial_\nu\big] - R(-\psi_1)\big[\hat{T}_g(\boldsymbol{\kappa}) \cdot E_R(\psi_2,0)[\partial_{\mu}] \cdot \partial_\nu\big].
\end{split}
\end{equation} 
In the last equality we used that $f_*(\phi_1 f^* \phi_2) =  f_*(\phi_1) \phi_2$ holds for any cohomology classes $\phi_1,\phi_2$.

For the second term in \eqref{uoohu} we can use that $f \circ \textnormal{gl}: \M_{1,1+1} \times \M_{0,3} \to \M_{1,1+2}\to \M_{1,1+1}$ is the identity under the canonical identification $\M_{1,1+1}\times \M_{0,3} \cong \M_{1,1+1}$, so $(f \circ \textnormal{gl})_* \psi_1 = \psi_1$ and $(f \circ \textnormal{gl})_* \psi = \psi_2$, while for the third term we can use $(f \circ \textnormal{gl})_* \psi' = \psi_1$ and $(f \circ \textnormal{gl})_* \psi_2 = \psi_2$, and $(f \circ \textnormal{gl})_* \boldsymbol{\kappa} = \boldsymbol{\kappa}$ in both cases.  Therefore, applying $f_*$ to these two terms yields 
\begin{equation}
\label{n2eterm} R(-\psi_1)\big[\hat{T}_g(\boldsymbol{\kappa}) \cdot E_R(\psi_2,0)[\partial_\mu \cdot \partial_\nu]\big] + \partial_\nu \cdot E_R(0,\psi_1)\big[\hat{T}_g(\boldsymbol{\kappa}) \cdot R^{-1}(\psi_2)[\partial_{\mu}]\big].
\end{equation}
The formula we need is $f_*\big(\Omega_{g,1+2}(\partial_{\mu} \otimes \partial_{\nu})\big)_{|\mathcal{M}_{g,1+1}} = \eqref{nterm} + \eqref{n2eterm}$.

\medskip

\textsc{Step 3: Applying $\textnormal{sp}_R$.} In the first line of \eqref{nterm}, it sends the classes $f_*\tilde{\psi}_3^{m} = \kappa_{m - 1}$ for $m \geq 2$  to $0$, while $f_* \tilde{\psi}_3 = \kappa_0 = 2g$ on $\Mct_{g,1+1}$, and $f_*\tilde{\psi}_3^0 = 0$; in the third line, we can use $E_R(0,0) = R_1$. Applying $\textnormal{sp}_R$ to \eqref{nterm} yields
\begin{equation}
\label{contr1}
2g R(z)\big[\alpha^{g} \cdot \partial_\mu \cdot (\hat{t}_1 \cdot \partial_\nu  - R_1[\partial_\nu])\big] - E_R(0,-z)[\alpha^g \cdot \partial_{\mu} \cdot \partial_{\nu}] - R(z)\big[\alpha^g \cdot R_1[\partial_\mu] \cdot \partial_\nu\big]. 
\end{equation}
while applying it to \eqref{n2eterm} results in
\begin{equation}
\label{contr2}
R(z)\big[\alpha^g \cdot R_1[\partial_\mu \cdot \partial_\nu]\big] + \partial_\nu \cdot E_R(0,-z)\big[\alpha^g \cdot \partial_\mu].
\end{equation}
In the sum \eqref{contr1} and \eqref{contr2} appears the commutator of $E_R(0,-z)$ with the operator of multiplication by $\partial_\nu$. Due to \eqref{WWeight2} this is also the commutator of $z^{-1}R(z)$ with the multiplication by $\partial_\nu$. Hence
\begin{equation*}
\begin{split}
\textnormal{sp}_R\,f_*\big(\Omega_{g,1+2}(\partial_\mu,\partial_\nu)\big)_{|\M_{g,1+1}} & = 2gR(z)\big[\alpha^g  \cdot (\hat{t}_1 \cdot \partial_{\nu} - R_1[\partial_\nu]) \cdot \partial_\mu\big] - z^{-1}\textbf{[}R(z),\partial_\nu \cdot \textbf{]}\big[\alpha^g \cdot \partial_\mu\big] \\
& \quad +  R(z)\big[\alpha^g \cdot \textbf{[}R_1,\partial_\nu \cdot \textbf{]}[\partial_\mu]\big],
\end{split}
\end{equation*}
We transformed it with help of Lemma~\ref{lem:alphaHR1}. In matrix form in the canonical basis, we can rewrite
\begin{equation*}
\begin{split}
\textnormal{sp}_R\,f_*\Omega_{g,1+2}{}_{|\M_{g,1+1}} & = - 2gR(z)A^g (\dd H)H^{-1} - z^{-1}\textbf{[}R(z),\dd U\textbf{]}A^g + R(z)A^g \textbf{[}R_1,\dd U\textbf{]} \\
& = -2gR(z)A^g (\dd H)H^{-1} - z^{-1}\textbf{[}R(z),\dd U\textbf{]}A^g + R(z)A^g\big((\dd H)H^{-1} + (\dd \Psi)\Psi^{-1}\big).
\end{split}
\end{equation*}
Equating this to \eqref{ysho}, we observe simplifications and $A^g$ appears as prefactor on the right. Since $A^g$ is invertible we conclude that
\[
\dd R(z) - (\dd \Psi)\Psi^{-1} R(z) = - z^{-1} \textbf{[}R(z),\dd U\textbf{]} + R(\dd H)H^{-1}.
\]
This is an equivalent form of the claimed differential equation for $R(z)$.

\medskip

\noindent \textsc{Step 4: Applying $\textnormal{sp}_T$.} In this case $f_* \tilde{\psi}_3^m = \kappa_{m - 1}$ is sent to $z^{m - 1}$ for $m \geq 2$, while $f_*\tilde{\psi}_3 = \kappa_0 = 2g$ and $f_*\textbf{1} = 0$. Then, applying $\textnormal{sp}_T$ to \eqref{nterm} yields
\begin{equation}
\label{Tcont1}
\begin{split}
& \quad z^{-1} \cdot \hat{T}_g(z) \cdot \partial_\mu \cdot (\hat{T}(z) \cdot R^{-1}(z) - \Id)[\partial_\nu]  + (2g-1) \hat{T}_g(z) \cdot \partial_\mu \cdot (\hat{t}_1 \cdot \partial_\nu - R_1[\partial_\nu])  \\
& \quad - R_1[\hat{T}_g(z) \cdot \partial_\mu \cdot \partial_\nu] -  \hat{T}_g(z) \cdot R_1[\partial_\mu] \cdot \partial_\nu,
\end{split}
\end{equation}
while applying it to \eqref{n2eterm} results in
\begin{equation}
\label{Tcont2}
\hat{T}_g(z) \cdot R_1[\partial_\mu \cdot \partial_\nu] + \partial_\nu \cdot R_1\big[\hat{T}_g(z) \cdot \partial_\mu\big].
\end{equation}
The sum of \eqref{Tcont1} and \eqref{Tcont2} reconstructs 
\begin{equation*}
\begin{split}
 \textnormal{sp}_T f_*\big(\Omega_{g,1+2}(\partial_\mu,\partial_\nu)\big)_{|\M_{g,1+1}} & = z^{-1}\hat{T}_g(z) \cdot \partial_\mu \cdot (\hat{T}(z) \cdot R^{-1}(z) - \Id)[\partial_\nu]   \\
& \quad + (2g-1) \hat{T}_g(z) \cdot \partial_\mu \cdot (\hat{t}_1 \cdot \partial_\nu - R_1[\partial_\nu]) + \textbf{[} \hat{T}_g(z) \cdot, \textbf{[}R_1,\partial_\nu \cdot\textbf{]}\textbf{]}[\partial_\mu].
\end{split}
\end{equation*}
In the canonical basis, this reads
\begin{equation}
\begin{split}
\textnormal{sp}_T f_*(\Omega_{g,1+2})_{|\M_{g,1+1}} & = z^{-1} A^g \hat{T}(z) \cdot \big(\hat{T}(z) \cdot R^{-1}(z) - \Id\big)\dd \overline{U} \cdot\,- (2g - 1)A^g \hat{T}(z) \cdot (\dd H)H^{-1} \\
& \quad + \textbf{[}A^g\hat{T}(z) \cdot ,(\dd\Psi)\Psi^{-1}\textbf{]},
\end{split}
\end{equation}
Here we used again Lemma~\ref{lem:alphaHR1} and the fact that $A^g \hat{T}(z) \cdot$ and $H$ are diagonal matrices, as well as the notation $\dd \overline{U}$ for the column vector whose $i$-th component is $\dd u^i$ for $i \in [N]$. Equating this to \eqref{ysho2} we find that commutator and the terms containing $2g$ cancel out and $A^g$ factors out. The result is an equation for diagonal matrices, which can also be written for the vector made out of the diagonal entries. This replaces the matrix $\hat{T}(z) \cdot$ with the vector $\hat{T}(z)$ and takes the announced form.
\end{proof}

\begin{figure}[h!]
\begin{center}
\includegraphics[width=0.5\textwidth]{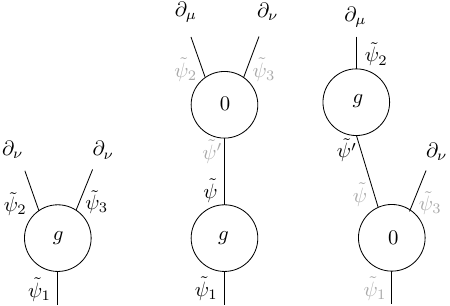}
\caption{\label{fig:DiffRstrat} The stable trees associated to the strata $\M_{g,1+1}$, $\mathcal{S}$, $\mathcal{S}'$. Psi-classes in grey vanish.}
\end{center}
\end{figure}

\begin{proof}[Proof of Corollary~\ref{cor:RTfial}] The differential equation for $R(z)$ in Proposition~\ref{prop:diffTR} can be rewritten
\[
\dd(R(z)H^{-1}e^{U/z}) - (\dd \Psi)\Psi^{-1}R(z)H^{-1}e^{U/z} - z^{-1}\dd U R(z)H^{-1}e^{U/z} = 0.
\]
Comparing with \eqref{conntriv} relates the columns of $R(z)H^{-1}e^{U/z}$ to vector fields that are flat for $\nabla^{-z}$.

Next, we differentiate $\Upsilon(z) = R(z)[\hat{T}^{-1}(z)]$ represented as a column vector in canonical basis. Inserting Proposition~\ref{prop:diffTR} and using commutativity of the product $\cdot$ we find
\begin{equation*}
\begin{split}
\dd \Upsilon(z) & = (\dd R(z))[\hat{T}^{-1}(z)] - R(z)\big[(\dd\hat{T}(z)) \cdot \hat{T}^{-2}(z)\big] \\
 & = \big(R(z)(\dd H)H^{-1} + (\dd \Psi)\Psi^{-1} R(z) - z^{-1}R(z)\dd U + z^{-1}(\dd U) R(z)\big)[\hat{T}^{-1}(z)] \\
 & \quad + R(z)\big[-(\dd H)H^{-1}\hat{T}^{-1}(z) + z^{-1}\dd \overline{U} \cdot \hat{T}^{-1}(z)  - z^{-1}R^{-1}(z)\dd \overline{U}\big] \\
 & = (\dd \Psi)\Psi^{-1} \Upsilon(z)  + z^{-1}(\dd U)\Upsilon(z)- z^{-1}\dd \overline{U}.
\end{split}
\end{equation*}
From \eqref{conntriv} we recognise the equation $\nabla_{X}^{-z}\Upsilon(z)= -z^{-1} X$ for any vector field $X$.  In particular for the unit vector field we get $\nabla_{\1} \Upsilon(z) = z^{-1}(\Upsilon(z) - \1)$, which gives a recursion for the coefficient $\Upsilon_m$ of  $z^m$ in $\Upsilon(z)$, namely $\nabla_{\1} \Upsilon_m = \Upsilon_{m + 1}$. Since $\Upsilon_0 = \1$, we deduce $\Upsilon(z) = \sum_{m \geq 0} z^m \nabla_\1^{m}(\1)$.
\end{proof}

\subsection{Unique reconstruction for conformal ctF-CohFTs}
\label{sec:recHomo}

The diagonal ambiguity in the definition of $R(z)$ and $T(z)$ via the differential equations in Proposition~\ref{prop:diffTR} can be fixed if we have conformality assumptions. Before discussing this, we shall review the general properties of conformal ctF-CohFTs.

Given $L \in \End(V)$ we let it act on $\Phi \in \End(V^{\otimes n},V)$ as a ``derivation''
\[
(d_L \Phi)(v_1 \otimes \cdots \otimes v_n) = L\big(\Phi(v_1 \otimes \cdots \otimes v_n)\big) - \sum_{i = 1}^{n} \Phi(v_1 \otimes \cdots \otimes v_{i - 1} \otimes L(v_i) \otimes v_{i + 1} \otimes \cdots \otimes v_n).
\]
We introduce $\textnormal{deg} \in \End(H^{\textnormal{even}}(\Mct_{g,1+n}))$ given by multiplication by $k$ on the subspace $H^{2k}(\Mct_{g,1+n})$. 

\begin{definition}
\label{def:Fconf} A ctF-CohFT $\Omega$ on $V$ is conformal  of dimension $\Delta \in \mathbb{C}$ if there exist $K \in V$ and $L \in \End(V)$ such that for $2g - 1 + n > 0$
\begin{equation}
\label{confOmega}
f_*\Omega_{g,1+n+1}(- \otimes K) = (d_{L} - \textnormal{deg} + g\Delta + n - 1) \Omega_{g,1+n}.
\end{equation}

\end{definition}
This can be reformulated in terms of the formal shift of $\Omega$, which is a family of $t$-dependent ctF-CohFTs. Indeed, introduce the \emph{Euler vector field} in the flat basis
\[
E := K + L[t] = K^{\mu} \partial_\mu + L_{\nu}^{\mu} t^{\nu} \partial_{\mu}.
\]
Then, denoting $\mathcal{L}_{X}$ the Lie derivative on tensor fields along a vector field $X$, \eqref{confOmega} is equivalent to
\begin{equation}
(\mathcal{L}_{E} + \textnormal{deg})\Omega_{g,1+n} = (g\Delta + n - 1)\Omega_{g,1+n}
\end{equation}
at $t = 0$. The definition of the formal shift immediately implies that it remains valid to all order\footnote{Remark that at a point $p$ in the flat F-manifold, the vector $K$ in \eqref{confOmega} should be replaced by the vector field $E$ at $p$.}. In particular for $(g,1+n) = (1,1)$ in degree $0$, the conformality property implies
\begin{equation}
\label{LEA}
\mathcal{L}_{E} \alpha = (\Delta - 1)\alpha.
\end{equation}
And, if there is a (non-necessarily flat) unit $\1$, it implies $\mathcal{L}_{E} \1 = -\1$.
\begin{proposition}
\label{prop:c1conf}
Let $\Omega$ be an invertible conformal ctF-CohFT. Then, the R- and T-elements of the F-Givental group that Theorem~\ref{thm:transfree} associates to its formal shift obey
\begin{equation}
\label{LERT}
(\mathcal{L}_E + z\partial_z) R(z) = 0,\qquad (\mathcal{L}_{E} + z \partial_z + 1)\hat{T}(z) = 0.
\end{equation}
In particular $(\mathcal{L}_{E} + z\partial_z)T(z) = 0$ and $(\mathcal{L}_{E} + z\partial_z + 1)\Upsilon(z) = 0$. If furthermore $\Omega$ is semi-simple, then the equation \eqref{LERT} for $R(z)$ together with Proposition~\ref{prop:diffTR} determines uniquely $R(z)$ and $T(z)$.
\end{proposition}

\begin{proof}
First remark that the properties $\mathcal{L}_{E} \cdot = \cdot$ and $\mathcal{L}_{E} \alpha = (\Delta - 1)\alpha$ imply that $\mathcal{L}_{E}(\alpha^g \cdot) = g\Delta \alpha^g \cdot$. Then, we come back to \eqref{starung}, that is
\[
\textnormal{sp}_R \Omega_{g,1+1}{}_{|\M_{g,1+1}} = R(z) \circ (\alpha^g \cdot),\qquad \textnormal{sp}_T \Omega_{g,1+1}{}_{|\M_{g,1+1}} =  \hat{T}(z) \cdot \alpha^g \cdot.
\]
and apply $\mathcal{L}_E$. This gives
\begin{equation}
\begin{split}
 \textnormal{sp}_{R} (\mathcal{L}_{E} \Omega_{g,1+1})_{|\M_{g,1+1}} & = (\mathcal{L}_{E} R(z)) \circ (\alpha^g \cdot) + g\Delta  R(z) \circ (\alpha^g \cdot), \\
  \textnormal{sp}_T(\mathcal{L}_{E} \Omega_{g,1+1})_{|\M_{g,1+1}} & = (\mathcal{L}_{E}\hat{T}(z) \cdot) \circ (\alpha^g \cdot) + g \Delta \hat{T}(z) \cdot \alpha^g \cdot.
 \end{split}
\end{equation}
On the other hand
\[
(\mathcal{L}_{E} \Omega_{g,1+1})_{|\M_{g,1+1}} = \big(f_* \Omega_{g,1+2}(-\otimes E)-  d_{L} \Omega_{g,1+1}\big)_{|\M_{g,1+1}} = (-\textnormal{deg} + g \Delta)\Omega_{g,1+1}{}_{|\M_{g,1+1}}
\]
due to the conformality assumption \eqref{confOmega}. As $z$ appears from specialisation of cohomological degree $2$ classes, $\textnormal{deg}$ is realised by $z\partial_z$ after the specialisation. Comparing the two equations and using invertibility of $\alpha^g \cdot$ we get the claim for $R(z)$, and the relation $\mathcal{L}_{E}( \hat{T}(z) \cdot) + z\partial_z \hat{T}(z)\cdot = 0$ between endomorphism. If we apply it to the vector $\1$ we find
\[
\mathcal{L}_{E} \hat{T}(z) = \mathcal{L}_{E}(\hat{T}(z) \cdot )[\1] + \hat{T}(z) \cdot \mathcal{L}_{E}\1 = -z\partial_z \hat{T}(z) - \hat{T}(z).
\]
The last two claims follow from $T(z) = z(\1 - \hat{T}^{-1}(z))$ and $\Upsilon(z) = R(z)[\hat{T}^{-1}(z)]$.

Assume now semi-simplicity. If $R(z)$ and $R'(z)$ (not a derivative) are solutions of \eqref{LERT} and the equation in Proposition~\ref{prop:diffTR}, then $R''(z) := R'(z)R^{-1}(z) = \exp(\sum_{m \geq 1} D_m z^m)$ for some constant matrices $D_m$ which are diagonal in the canonical basis and we have $(\mathcal{L}_{E} + z\partial_z)R''(z) = 0$. This implies $D_m = 0$ for every $m$, making $R(z)$ unique. As $\Upsilon(z)$ is specified by \eqref{vaceqn}, $T(z)$ is unique as well.
 \end{proof}

We now turn to the flat F-manifold side.

\begin{definition}
\label{flathomodef} A flat F-manifold is \emph{homogeneous} if there exists a vector field $E$, called \emph{Euler vector field}, \st $\nabla\nabla E = 0$ and $\mathcal{L}_{E} \cdot = \cdot$, where the product $\cdot$ is seen as a section of $\End(TM^{\otimes 2},TM)$.
\end{definition}
The first property is equivalent to  having $E = (K^\mu + L_{\nu}^{\mu}t^{\nu}) \partial_\mu$ for some $K \in \Gamma(TM)$ and $L \in \Gamma(\End(TM))$. If the flat F-manifold has a (non-necessarily flat) unit $\1$, then $\mathcal{L}_{E} \1 = -\1$ holds automatically. If it is semi-simple, up to translation we can choose canonical coordinates such that $E = \sum_{i = 1}^{N} u^i \partial_i$, and there exist \cite{Lore} scalar constants $\Delta^1,\ldots,\Delta^N$ such that $\mathcal{L}_{E}(H^i) = -\frac{1}{2}\Delta^i H^i$ for any $i \in [N]$. In particular, if $\Delta^i =: \Delta$ does not depend on $i$, then the metric of \eqref{metriceta} satisfies $\mathcal{L}_{E} \eta = (2 - \Delta)\eta$.

\begin{definition}
\label{confhomodef}A flat F-manifold is conformal of dimension $\Delta \in \mathbb{C}$ if it is homogeneous, semi-simple and $\Delta^i = \Delta$ for any $i \in [N]$.
\end{definition}

We did not include the semi-simplicity assumption for conformal ctF-CohFTs (Definition~\ref{def:Fconf} makes sense without it), but we do include it for conformal flat F-manifolds (Definition~\ref{confhomodef} does not make sense without it). 
\begin{proposition}
\label{prop:invertif}
	 If $\Omega$ is an invertible semi-simple conformal ctF-CohFT, then the underlying flat F-manifold is conformal with same Euler vector field and dimension.
\end{proposition}
\begin{proof}
It is clear from the definitions that a conformal ctF-CohFT gives rise to a homogeneous flat F-manifold. If the ctF-CohFT is also invertible and semi-simple, then for any $i \in [N]$ we have $\alpha^i = (H^i)^{-2}$ up to a multiplicative constant (Lemma~\ref{lem:alphaHR1}). We compute
\begin{equation}
	\mathcal{L}_{E} \alpha = \sum_{i = 1}^{N} \big((H^i)^{-2} \mathcal{L}_{E} \partial_i - 2(H^i)^{-2} \mathcal{L}_{E}(\log H^i) \partial_i\big) = \sum_{i = 1}^{N} (H^i)^{-2}\big(-1 - 2 \mathcal{L}_{E}(\log H^i)\big) \partial_i,
\end{equation}
which is equal to $(\Delta - 1)\alpha$ due to \eqref{LEA}. Hence $\mathcal{L}_{E} H^i = -\frac{\Delta}{2} H^i$.
\end{proof}
\begin{corollary}
	A homogeneous semi-simple flat F-manifold which is not conformal cannot be produced by an invertible conformal ctF-CohFT.
\end{corollary}

For any semi-simple homogeneous flat F-manifold with $\Delta^i$ possibly distinct, \cite[Proposition 1.16]{arsie_semisimple_2023} produces a unique $\mathsf{R}(z)$ associated to a basis of flat sections of the dual connection like in \eqref{dualflatsec} and satisfying (their $\delta^i$ is our $-\frac{\Delta^i}{2}$)
\begin{equation}
	\label{LEDel}
	(\mathcal{L}_{E} + z\partial_z)\mathsf{R}(z)  + \frac{1}{2} \textbf{[}\Delta,\mathsf{R}(z) \textbf{]} = 0.
\end{equation}

Then $R(z) = H^{-1}\mathsf{R}(-z)^{-1}H$ (see \eqref{Rcorres}) does satisfy the equation $(\mathcal{L}_{E}+z\partial_z) R(z)=0 $; in particular, in the case of conformal semi-simple invertible ctF-CohFTs, the uniqueness statements in Proposition~\ref{prop:c1conf} and in \cite{arsie_semisimple_2023} agree. In the spirit of \cite[Section 8.4]{teleman_structure_2012} for CohFTs, the unique solution can be constructed more directly in the following way.

\begin{proposition}
\label{prop:c2conf} Besides, $R(z)$ and $T(z)$ from Theorem~\ref{thm:transfree} are uniquely determined by
\begin{equation}
\label{dzdzR} \boldsymbol{\mu}R(z) + z\partial_zR(z) = z^{-1}\textnormal{\textbf{[}}R(z),K\cdot \textnormal{\textbf{]}},\qquad 
\big(\boldsymbol{\mu} + \tfrac{\Delta}{2} + z\partial_z\big) \Upsilon(z) = -z^{-1} K \cdot (\Upsilon(z) - \1),
\end{equation}
where $\boldsymbol{\mu}  = -L + (1 - \frac{\Delta}{2}) \Id_V$ is the Hodge grading operator and $\Upsilon(z) = R(z)[\1 - \frac{T(z)}{z}]$.
\end{proposition}

The difference with Proposition~\ref{prop:c1conf} is that \eqref{dzdzR} does not involve differentiation in $t$, it holds pointwise in the flat F-manifold. Clearly, these equations determine uniquely $R(z)$ and $T(z)$ order by order in $z$. This proves Theorem~\ref{thm:recHom} since the F-TFT is determined by the product (data contained in the flat F-manifold data) and by the vector field $\alpha$ (which is determined by the value of $\alpha$ at a point and by the flat F-manifold data due to Lemma~\ref{lem:alphaHR1}). 

\begin{proof}
We want to translate the equations \eqref{LERT} at $t = 0$ in the canonical basis. For any $i,\mu \in [N]$ we have
\begin{equation*}
\begin{split}
\mathcal{L}_E (\Psi_{\mu}^i) & = E^{\nu} \frac{\partial}{\partial t^{\nu}}\bigg( \frac{\partial u^i}{\partial t^{\mu}}\bigg) = E^{\nu} \frac{\partial}{\partial t^{\mu}} \bigg(\frac{\partial u^i}{\partial t^{\nu}}\bigg) \\
& = \frac{\partial}{\partial t^{\mu}}\bigg(E^{\nu}\frac{\partial u^i}{\partial t^{\nu}}\bigg) - L_\mu^{\nu} \Psi_{\nu}^i  = \frac{\partial}{\partial t^{\mu}}\bigg(\sum_{j = 1}^{N} u^j \frac{\partial u^i}{\partial u^j}\bigg) - L_\mu^{\nu} \Psi_{\nu}^i \\
& =  \frac{\partial u^i}{\partial t^{\mu}} - L_{\mu}^{\nu} \Psi_{\nu}^i = \Psi^i_{\mu} - L_{\mu}^{\nu} \Psi_{\nu}^i,
\end{split}
\end{equation*}
Therefore $\iota_E(\dd \Psi)\Psi^{-1} = \Id_V - L$, where $\iota_E$ is the interior product of a $1$-form with $E$. Since $\mathcal{L}_{E} \dd u^i = \dd u^i$ and $\mathcal{L}_{E}\partial_i = -\partial_i$ for any $i \in [N]$, there is no distinction between the matrix of $\mathcal{L}_{E}R(z)$ in the canonical basis and $\mathcal{L}_{E}$ applied entrywise to the matrix of $R(z)$. We compute it thanks to Proposition~\ref{prop:diffTR}
\begin{equation}
\begin{split}
-z\partial_z R(z) = \mathcal{L}_ER(z) & = R(z) \iota_E(\dd H)H^{-1} + \iota_E(\dd \Psi)\Psi^{-1} R(z) - z^{-1}\textbf{[} R(z),\iota_E \dd U \textbf{]} \\
& = - \frac{\Delta}{2}R(z) + (\textnormal{Id} - L) R(z) - z^{-1} \textbf{[} R(z), E \cdot \textbf{]}.
\end{split}
\end{equation}
This relation is still valid for the formal shift (at any $t$), and specialising at $t = 0$  replaces $E$ by $K$ in the last term, while in the first three terms we recognise the Hodge grading operator. This is a relation between matrices which can also be directly interpreted as relation between endomorphisms and which retrieves \eqref{dzdzR}. We still have to justify that it determines completely $R(z)$. In canonical coordinates, the equation reads
\[
(R_{k})_{ij}(K_j-K_i)=\sum_l\boldsymbol{\mu}_{il}(R_{k-1})_{lj}+(k-1)(R_{k-1})_{ij}.
\]
When $K_i\neq K_j$, then $(R_{k})_{ij}$ is controlled by $R_{k-1}$. When $K_i = K_j$, the equation for $k=1$ shows that $\boldsymbol{\mu}_{ij}=0$, and this implies $k(R_{k})_{ij}=-\sum_l \boldsymbol{\mu}_{il}(R_{k})_{lj}$ for $k \geq 1$, where the sum runs only over those $l$ such that $K_l\neq K_i = K_j$. Therefore the remaining entries of $R_k$ are determined by the already fully determined entries. 

Besides, we have $-(z\partial_z + 1)\Upsilon(z)^i = (\mathcal{L}_E\Upsilon(z))^i = \mathcal{L}_E(\Upsilon(z)^i) - \Upsilon(z)^i$ for any $i \in [N]$, and by Proposition~\ref{prop:diffTR} the column vector $(\mathcal{L}_E(\Upsilon(z)^i))_{i = 1}^{N}$ is equal to $(\Id - L) \Upsilon(z) + z^{-1} K \cdot (\Upsilon(z) - \1)$. Substituting $\boldsymbol{\mu} = (1-\frac{\Delta}{2})\Id - L$ we get the result for $\Upsilon(z)$.
\end{proof}

\subsection{Final remarks}
\label{badex}
If we had allowed $\Delta$ in Definition~\ref{def:Fconf} to be an endomorphism instead of a scalar, Proposition \ref{prop:c1conf} and the unique reconstruction in Proposition~\ref{prop:c2conf} would be valid if we further assume that $\Delta$ (a) commutes with multiplication by $\alpha$ and $\hat{T}$, and (b) commutes with $R(z)$. Property (a) can be achieved in the semi-simple case by assuming that $\Delta$ is diagonal in the canonical basis, which is indeed what happens in general for semi-simple homogeneous flat F-manifolds. But then, (b) is equivalent to $R(z)$ having a block diagonal structure with respect to the decomposition of $V$ into eigenspaces of $\Delta$. A necessary condition for (b) is that the ctF-CohFT  does not couple different eigenspaces. Such a setting does not bring anything new as the unique reconstruction separately in each eigenspace would already reconstruct the full ctF-CohFT.

There are interesting examples of semi-simple ctF-CohFTs $\Omega$ for which the underlying flat F-manifold is homogeneous with non-scalar $\Delta$ such that the ctF-CohFT couples different eigenspaces and for which we have $\mathcal{L}_{E} \alpha = (\Delta - 1)\alpha$. If $\alpha$ is invertible, Theorems~\ref{thm:transfree} and \ref{thm:rec} tell us that $\Omega=  RT\omega$, but the flat F-manifold determines $R(z)$ only up to the diagonal ambiguity and we do not know a priori if it satisfies an additional equation like $(\mathcal{L}_{E} + z\partial_z)R(z) = 0$ which would kill the ambiguity. In concrete geometric examples, such an equation (or a variant of it) could perhaps be proved by ad hoc methods. If $\alpha$ is non-invertible in such a way that it is invertible on a single eigenspace and zero on the others, it is not impossible for $\Omega$ to be conformal, but in that case we do not know whether $\Omega$ is of the form $RT\omega$, and even if it were we still would not be able to derive an additional equation for $R(z)$ like we did in the proof of Proposition~\ref{prop:c2conf}.

Yet, an interesting situation occurs when a F-CohFT couples different eigenspaces but its restriction to $\M^{\textnormal{ct}}$ does not. Then, restricting/projecting the F-CohFT to an eigenspace gives a conformal ctF-CohFT, whose formal shift has a chance to be generically invertible and semi-simple, and to which we can apply the reconstruction theorem. We demonstrate this for the extended $r$-spin theory in the next Section.

\section{Application: extended \texorpdfstring{$r$}{r}-spin class in the extended direction}

\subsection{Conformal flat F-manifolds in dimension \texorpdfstring{$1$}{1} and compact-type vanishing}

\label{1dimFman}
Let us consider the example of any flat F-manifold structure in an open $M \subseteq \mathbb{C}$ having a conformal dimension $\Delta \neq 2$. For reasons that will appear later, we parametrise it $\Delta = 2 - \frac{2}{r}$ with $r \in \mathbb{C} \setminus \{0\}$.

The unique R-element specified by Proposition~\ref{prop:c2conf}  is $R(z) = \textnormal{Id}_{\mathbb{C}}$ and we must have $\boldsymbol{\mu} = 0$, therefore $L = (1 - \frac{\Delta}{2})\textnormal{Id}_{\mathbb{C}} = r^{-1} \textnormal{Id}_{\mathbb{C}}$. Up to translation, we can take a flat coordinate such that the Euler vector field vanishes at $t = 0$. Then:
\[
E = r^{-1}t \partial_t.
\]
Denoting $F(t)\partial_t$ the vector potential, the product is
\begin{equation}
\label{product1}
\partial_t \cdot \partial_t = F''(t) \partial_t.
\end{equation}
The property $\mathcal{L}_{E} \cdot = \cdot$ imposes that $F''(t)$ is homogeneous of degree $r - 1$, that is
\begin{equation}
\label{prepro}
F(t) = \frac{A t^{r + 1}}{r(r + 1)}
\end{equation}
for some constant $A \neq 0$, up to linear terms that we can set to zero. In particular $ \1 = A^{-1} t^{-(r - 1)}\partial_t$ is a unit at $t \neq 0$ and we can write
\[
E = r^{-1}A t^{r} \1.
\]
Solving the second differential equation of Proposition~\ref{prop:c2conf} we deduce
\[
\hat{T}^{-1}(z) = \Upsilon(z) = \sum_{m \geq 0}   \Big[\prod_{j = 1}^{m} (jr - 1)\Big]  (-A^{-1}t^{-r}z)^m \1.
\]
If $r$ were an integer (as it will be later), the product in the formula above would simply be the $r$-fold factorial $(rm - 1)!^{(r)}$.
\begin{lemma} 
Equivalently, we can derive the presentation
\begin{equation}
\label{TQT}
\hat{T}(z) = \exp\Bigg[-\sum_{m \geq 1} (-A^{-1}t^{-r}z)^m s_m\Bigg] \1,
\end{equation}
where $s_0 = 0$, $s_1 = r - 1$ and
\[
\forall m \geq 1 \qquad (m + 1)s_{m + 1} = m\big(r(m + 2) - 1\big) s_{m} + r \sum_{\ell + \ell' = m} \ell\ell' s_{\ell} s_{\ell'}. 
\]
\end{lemma} 
\begin{proof}
Plugging $\hat{T}^{-1}(z) = e^{s(z)} \1$ in the second differential equation of Proposition~\ref{prop:c2conf} yields
\[
1 - r^{-1} + z s'(z) = z^{-1} r^{-1}At^{r} (e^{-s(z)} - 1)
\]
Differentiating once more with respect to $z$ yields
\[
z s''(z) + s'(z) = -z^{-2} r^{-1}At^{r}(e^{-s(z)} - 1) -  z^{-1} r^{-1}At^{r} s'(z)e^{-s(z)}
\]
We then eliminate $e^{\hat{t}(z)}$ between the two equations, and arrive to
\[
zs''(z) + (3 - r^{-1} + zs'(z) + z^{-1}r^{-1}At^r)s'(z) = z^{-1}(r^{-1} - 1)
\]
Inserting an expansion $s(z) = \sum_{m \geq 0} (-A^{-1}t^{-r}z)^m s_m$ readily gives the announced recursion. 
\end{proof}

\begin{definition}
We define uniquely classes $P_m^{(r)} \in H^{2m}(\M_{g,1+n}^{\textnormal{ct}})$ by writing
\[
\exp\bigg(-\sum_{m \geq 1} s_m \kappa_m\bigg) := 1+\sum_{m \geq 1} P_{m}^{(r)}(\boldsymbol{\kappa}).
\]
\end{definition}
For $r=1$, $P_{m}^{(1)}=0$ for all $m\geq 1$. For $r \neq 1$, Writing $P_m^{(r)}= -\frac{(r - 1)}{m!}\tilde{P}_m^{(r)}$, we find in low degrees
\begin{equation*}
\begin{split}
\tilde{P}_1^{(r)}(\boldsymbol{\kappa}) & = \kappa_1, \\
\tilde{P}_2^{(r)}(\boldsymbol{\kappa}) & = (3r - 1)\kappa_2 - (r-1)\kappa_1^2, \\
\tilde{P}_3^{(r)}(\boldsymbol{\kappa}) & = 2(13r^2 - 8r + 1) \kappa_3 - 3(3r^2 - 4r + 1)\kappa_2\kappa_1 + (r-1)^2 \kappa_1^3, \\
\tilde{P}_4^{(r)}(\boldsymbol{\kappa}) & = 6(71r^3 - 61r^2 + 15r - 1)\kappa_4 - 8(13r^3 - 21r^2 + 9r - 1) \kappa_3\kappa_1 - 3(3r - 1)^2(r-1)\kappa_2^2 \\
& \quad + 6(3r - 1)(r -1)^2 \kappa_2\kappa_1^2 - (r-1)^3\kappa_1^4.
\end{split}
\end{equation*}

\begin{proposition}
\label{prop:1FCoh} Suppose that there exists a ctF-CohFT $\Omega$ on $V = \mathbb{C}$, which is conformal of dimension $\Delta \notin \{0,2\}$ but not invertible and not semi-simple, but whose formal shift exists for $t$ in a small neighborhood of $0 \in \mathbb{C}$ and is invertible and semi-simple for $t \neq 0$. Then, there exists an integer $r \geq 2$ and constants $A,a \neq 0$ such that $\Delta = 2 - \frac{2}{r}$ and the formal shift at $t$ is
\begin{equation}
\label{eqnOmegunasg}
\Omega_{g, 1+n}(\partial_{t}^{\otimes n}) = a^{g} A^{g + n - 1} t^{(r - 1)(2g + n - 1)}\bigg(\sum_{m \geq 0} (-A)^{-m} t^{-rm} P_m^{(r)}(\boldsymbol{\kappa})\bigg) \partial_t.
\end{equation}
\end{proposition}
For every $r \geq 2$ we will exhibit in Section~\ref{Sec:extrspin} a ctF-CohFT  (which is not a F-CohFT) satisfying these assumptions (Lemma~\ref{crstarlema}). The reconstruction formula \eqref{eqnOmegunasg} is obtained for $t \neq 0$, while regularity of the formal shift at $t = 0$ forces the negative Laurent coefficients on the right-hand side to vanish, allowing us to deduce compact-type vanishing relations.
\begin{corollary}
\label{vanishingcor}
For any integer $r \geq 2$ and $g,n \geq 0$ such that $2g - 1 + n > 0$ we have
\[
(r - 1)(2g - 1 + n) < rm \qquad \Longrightarrow \qquad P_m^{(r)}(\boldsymbol{\kappa}) = 0 \in H^{2m}(\M_{g,1+n}^{\textnormal{ct}}).
\]
\end{corollary}
Equation \eqref{eqnOmegunasg} also gives a formula for the original ctF-CohFT (\textit{i.e.} pick the coefficient of $t^{0}$) in terms of $\kappa$-classes.  This proves Theorem~\ref{thm:vani}.

\begin{proof}[Proof of Proposition~\ref{prop:1FCoh}]
The assumptions turn the formal shift into a family of ctF-CohFTs over an open neighborhood $M \subseteq \mathbb{C}$ of $0$, with $F'' \neq 0$ and $\alpha \neq 0$ for $t \neq 0$. Then, it must be conformal of dimension $\Delta = 2 - \frac{2}{r}$ with $r \neq 0,1$ and the underlying flat F-manifold is also conformal away from $t = 0$. The vector potential must be given by \eqref{prepro} up to a shift of origin of $t$, but non semi-simplicity at $t = 0$ and semi-simplicity at $t \neq 0$ impose that origins agree and that $A \neq 0$ and $\Delta \neq 0$. The vector potential should also be analytic near $t = 0$, forcing $r$ to be an integer $r \geq 2$. Conformality imposes that $\mathcal{L}_{E}\alpha = (\Delta - 1)\alpha = \frac{r - 2}{r} \alpha$, whose solution is
\begin{equation}
\label{formealpha}
\alpha = a t^{r - 1} \partial_t = aA t^{2r - 2} \1
\end{equation}
for some constant $a$. As we assume $\alpha \neq 0$ for $t \neq 0$ we have $a \neq 0$. Then, the reconstruction Theorem~\ref{thm:recHom} for $t \neq 0$ yields $\Omega_{g,1+n}(\partial_t^{\otimes n}) = \hat{T}(\boldsymbol{\kappa}) \cdot \alpha^{g} \cdot \partial_t^{\cdot n}$ (notice that having $R(z) = 1$, the edge weight is zero so that only the tree with one vertex contributes to the reconstruction formula). This is the announced formula after we take into account  \eqref{product1}, \eqref{prepro} and \eqref{formealpha}.\end{proof}

\subsection{Construction from the extended \texorpdfstring{$r$}{r}-spin class}
\label{Sec:extrspin}

The extended $r$-spin class $c^{r,\textnormal{ext}}$ was first constructed in genus $0$ in \cite{JKV} and further studied in \cite{BCT}. The definition in all genera was proposed in \cite{buryak_extended_2021} and shown to be a F-CohFT, see especially \cite[Theorem 3.9]{buryak_extended_2021} for its properties. The underlying vector space is
\[
V^{\textnormal{ext}} = V \oplus V',\qquad V = \mathbb{C}^{r - 1},\qquad V' = \mathbb{C}
\]
We denote $(\partial_\mu)_{\mu = 1}^{r - 1}$ the standard basis in $V$, while $\partial_t$ is the basis in $V'$ and $t$ the corresponding coordinate.

\medskip

\noindent \textsc{Homogeneity.} For any $\nu,\mu_1,\ldots,\mu_n \in [r]$ the component of $c^{r,\textnormal{ext}}_{g,1+n}(\partial_{\mu_1} \otimes \cdots \otimes \partial_{\mu_n})$ along $\partial_{\nu}$ has half-cohomological degree
\[
\frac{2g(r - 1) - (\nu - 1) + \sum_{i = 1}^{n} (\mu_i - 1)}{r}.
\]

\medskip

\noindent \textsc{Projection.} Call $\pi_V$ the projection onto the subspace $V$. For $n \geq 1$ and any $w \in (V^{\textnormal{ext}})^{\otimes (n - 1)}$ we have the Ramond vanishing
\[
\pi_V\big[c_{g,1+n}^{r,\textnormal{ext}}(\partial_t \otimes w)\big]= 0.
\]
Besides, for any $n \geq 0$ and $w \in (V^{\textnormal{ext}})^{\otimes n}$ we have
\begin{equation}
\label{projsepc}
\pi_V\big[c^{r,\textnormal{ext}}_{g,1+n}(w)\big] = \sum_{\mu = 1}^{r - 1} \lambda_g c^r_{g,1+n}(\partial_{r - \mu} \otimes w) \partial_\mu,
\end{equation}
where $c^{r}$ is the Witten $r$-spin class \cite{PV01,ChiodoWitten,FSZ}.

\medskip

\noindent \textsc{Compatibility with self-gluing.} On top of the F-CohFT axioms, the extended $r$-spin classes are compatible with the self-gluing morphism $\textnormal{gl} : \Mbar_{g-1,1+n+2} \rightarrow \Mbar_{g,1+n}$ up to a prefactor:
\begin{equation}
\label{compaself}
\forall w \in (V^{\textnormal{ext}})^{\otimes n} \qquad \textnormal{gl}^* c^{r,\textnormal{ext}}_{g,1+n}(w) = -r c^{r,\textnormal{ext}}_{g-1,1+n+2}(w \otimes \partial_t^{\otimes 2}).
\end{equation}

\medskip

\noindent \textsc{Flat F-manifold.} The underlying flat F-manifold is semi-simple away from the origin and homogeneous with
\[
\{\Delta^1,\ldots,\Delta^{r-1},\Delta^r\} = \big\{\tfrac{r - 2}{r},\ldots,\tfrac{r - 2}{r} ,\tfrac{2(r - 1)}{r}\}.
\]
Indeed, if we restrict this flat F-manifold to the $(r-1)$-dimensional subspace $V \subset V^{\textnormal{ext}}$, we obtain the Frobenius manifold for the Witten $r$-spin CohFT, for which the conformal dimension $\Delta=\tfrac{r - 2}{r}=\Delta^i$, $i \in [r-1]$, see e.g. \cite{pandharipande_relations_2015}. If we restrict to the $1$-dimensional subspace $V' \subset V^{\textnormal{ext}}$, we are exactly in the situation described at the end of Section~\ref{badex}.

\begin{definition} Let $c^{r,\star}$ be the restriction/projection of $c^{r,\textnormal{ext}}$ to $V' = \mathbb{C}.\partial_t$.
\end{definition}
\begin{lemma}
\label{crstarlema}
$c^{r,\star}_{|\textnormal{ct}}$ is a ctF-CohFT on $V'$. Its formal shift exists for any $t \in \mathbb{C}$ and is an invertible, conformal ctF-CohFT for $t \neq 0$, with conformal dimension $\Delta$ and vector potential $F(t)\partial_t$ given by
\[
\Delta = 2 - \frac{2}{r},\qquad F(t) = \frac{(-t)^{r + 1}}{(r + 1) r^{r}}.
\]
It fulfills the assumption of Proposition~\ref{prop:1FCoh} with
\[
A = (-r)^{1 - r},\qquad a = (-r)^{2 - r}
\]
\end{lemma}
\begin{proof}
Comparing to \eqref{projsepc}, the definition means
\[
c_{g,1+n}^{r,\textnormal{ext}}(\partial_t^{\otimes n}) = \sum_{\mu = 1}^{r - 1} \lambda_g c^{r}_{g,1+n}(\partial_{r-\mu} \otimes \partial_t^{\otimes n}) \partial_{\mu} + c^{r,\star}_{g,1+n}(\partial_t^{\otimes n}).
\]
Due to Ramond vanishing, for $n \geq 1$ we simply have $c_{g,1+n}^{r,\textnormal{ext}}(\partial_t^{\otimes n}) = c_{g,1+n}^{r,\star}(\partial_t^{\otimes n})$, but for $n = 0$ the first term is not always zero, preventing $c^{r,\star}$ from being a F-CohFT. Yet, the vanishing\footnote{The Hodge bundle on $\Mct_{g,1+n}$ is a pullback from the moduli space of principally polarised abelian varieties $\mathcal{A}_g$, and $\lambda_g$ vanishes there due to \cite[1.2]{vdG}.} of $\lambda_g$ on $\M^{\textnormal{ct}}$ makes $c^{r,\star}$ a ctF-CohFT. The homogeneity properties of $c^{r,\textnormal{ext}}$ show that $c^{r,\star}$ is conformal with
\[
K = 0,\qquad L = r^{-1} \textnormal{Id}_{V'},\qquad \Delta = \frac{2(r - 1)}{r}.
\]
The vector potential for $c^{r,\star}$ can be computed from the values of $X_{\alpha}$ given in  \cite[Proof of Theorem 4.6]{BCT}. We see that it corresponds to a conformal unital flat F-manifold away from $t = 0$ of the given conformal dimension. Since the formal shift of $c^{r,\star}$ remains conformal of the same dimension, it must have --- \textit{cf.} \eqref{formealpha}  --- $\alpha = a t^{r - 1}\partial_t$ for some constant $a$, which we now compute. This $a$ is the coefficient of $t^{r - 1}$ of the cohomological degree-zero part of the formal shift of $c^{r,\star}$ for $(g,1+n) = (1,1)$, that is
\begin{equation}
\label{fromtehts}
(r - 1)!a \partial_t = \Big[(f_{r - 1})_*\big(c^{r,\star}_{1,r}(\partial_t^{\otimes (r - 1)})\big)\big]^{\textnormal{deg}\,0},
\end{equation}
where $f_{r - 1} : \Mbar_{1,r} \rightarrow \Mbar_{1,1}$ is the forgetful morphism. As $\Mbar_{1,1}$ is one-dimensional, it is enough to understand how it pairs with the class $[\delta_{\textnormal{irr}}]$ of the (irreducible) boundary divisor, which has top-degree $2$. The pullback of $[\delta_{\textnormal{irr}}]$ via $f_{r - 1}$ is again the class of the irreducible boundary divisor $[\delta_{\textnormal{irr}}] \in H^{2}(\Mbar_{1,r})$. Thus 
\[
[\delta_{\textnormal{irr}}] \cup (f_{r - 1})_*\big(c^{r,\star}_{1,r}(\partial_t^{\otimes (r - 1)})\big)  = (f_{r - 1})_*\big(c^{r,\star}_{1,r}(\partial_t^{\otimes (r - 1)}) \cup [\delta_{\textnormal{irr}}]\big).
\]
By compatibility with the self-gluing morphism \eqref{compaself}, we have
\begin{equation*}
\int_{\Mbar_{1,r}} c^{r,\star}_{1,r}(\partial_t^{\otimes (r - 1)}) \cup [\delta_{\textnormal{irr}}] = \frac{1}{2} \int_{\Mbar_{0,r+2}} \!\!\!\!\!\textnormal{gl}^* c^{r,\star}_{1,r}(\partial_t^{\otimes(r - 1)}) = - \frac{r}{2} \int_{\Mbar_{0,r+2}}  \!\!\!\!\!c^{r,\star}_{0,r+2}(\partial_t^{\otimes (r + 1)}) = - \frac{r! A}{2} \partial_t.
\end{equation*}
By \eqref{fromtehts} this is equal to $(r - 1)!a \int_{\Mbar_{1,1}} [\delta_{\textnormal{irr}}] = \frac{(r - 1)!a}{2}$. Thus $a = - rA = (-r)^{2 - r}$.  In particular $a \neq 0$ and the formal shift of $c^{r,\textnormal{ext}}$ is invertible for $t \neq 0$.
\end{proof}

  \appendix
  
  \section{Metric for invertible semi-simple ctF-CohFTs}
 
 \label{appA}
 
For invertible semi-simple ctF-CohFTs, the metric $\eta$ of \eqref{metriceta} together with Proposition can be interpreted as coming from a Frobenius algebra structure on $TM$, although it does not in general corresponds to a Frobenius manifold because $\eta$ may not be flat. For Frobenius manifolds and CohFTs, $\eta$ is flat but this does not play a role in the proof of the reconstruction theorem of \cite{teleman_structure_2012}.

  \begin{lemma}
\label{comparmet}
Let $\Omega$ be an invertible semi-simple ctF-CohFT on $V$ with associated F-TFT $\omega$ and $\alpha := \omega_{1,1} = \sum_{i \in I} \alpha^i \partial_i$. Keep the same notations (with implicit $t$-dependence) for its formal shift, and define the metric
\[
\eta =\sum_{i=1}^N \frac{(\dd u^i)^{\otimes 2}}{\alpha^i}.
\]
Then, $(V,\cdot,\eta)$ is a $t$-dependent Frobenius algebra. Calling $\omega^{\eta}$ the associated TFT, we have $\omega^{\eta}_{g,1+n} = \# \circ \omega_{g,1+n}$ for any $g,n \geq 0$ such that $2g - 1 + n > 0$, where $\# : V \rightarrow V^*$ is the isomorphism specified by $\eta$.
\end{lemma}
\begin{proof}
The metric is automatically compatible with the product since it is diagonal: for any $i,j,k \in [N]$ we have $\eta(\partial_i \cdot \partial_j,\partial_k) = \delta_{i,j,k} \alpha_i^{-1} = \eta(\partial_i,\partial_j \cdot \partial_k)$. Hence $(V,\cdot,\eta)$ is a Frobenius algebra. Let $\omega^{\eta}$ be the associated TFT. A standard computation yields
\[
\forall i_1,\ldots,i_n \in [N]\qquad \omega_{g,n}(\partial_{i_1} \otimes \cdots \otimes \partial_{i_n}) = \sum_{i = 1}^{N} (\alpha^i)^{g - 1} \delta_{i,i_1,\ldots,i_n}
\]
We can compare the result to Lemma~\ref{lem:FTFT} in the canonical basis
\[
\forall i_1,\ldots,i_n \in [N] \qquad \omega_{g,1+n}(\partial_{i_1} \otimes \cdots \otimes \partial_{i_n}) = \sum_{i = 1}^{N} (\alpha^i)^{g} \delta_{i,i_1,\ldots,i_n} \partial_i.
\]
Since $\sharp(\partial_i) := \eta(\partial_i \otimes -) = \alpha_i^{-1} \dd u^i$ for any $i \in [N]$, we indeed have $\sharp \circ \omega_{g,1+n} = \omega_{g,1+n}^{\eta}$.
\end{proof}
v

\newpage
\bibliographystyle{amsalpha}
\bibliography{BibliF}

@article{SM,
author = {Khoroshkin, Anton and Markarian, Nikita and Shadrin, Sergey},
title = {Hypercommutative operad as a homotopy quotient of {BV}},
journal = {Commun. Math. Phys.},
volume = {322},
pages = {697--729},
year = {2013},
note = {\href{https://arxiv.org/abs/1206.3749}{math.QA/1206.3749}}}

@article{lu_GW_for_non_compact_targets_2006,
author = {Lu, Guangcun},
title = {Virtual moduli cycles and {G}romov--{W}itten Invariants of noncompact symplectic manifolds},
journal = {Commun. Math. Phys.},
volume = {43},
number = {131},
year = {2006},
pages = {43--131},
note = {\href{https://arxiv.org/abs/math/0306255}{math.DG/0306255}}}

@article{buryak_xu_yang_2026,
	title = {Bihamiltonian tests for integrable systems associated to rank-$1$ {F}-{C}oh{F}{T}s},
	author = {Buryak, Alexandr Y. and Xu, Jianghao and Yang, Di},
	year = {2026},
	note = {\href{https://arxiv.org/abs/2601.13203}{nlin/2601.13203}}}

@article{buryak_rossi_classification_2025,
	title = {Deformations of the {R}iemann hierarchy and the geometry of {$\overline{\mathcal{M}}_{g,n}$}},
	author = {Buryak, Alexandr Y. and Rossi, Paolo},
	year = {2025},
	journal = {to appear in Int. Math. Res. Not.},
	note = {\href{https://arxiv.org/abs/2504.02079}{math.ph/2504.02079}}}

@inbook{vdG,
author = {van der Geer, Gerard},
booktitle = {Moduli of curves and abelian varieties},
title = {Cycles on the moduli space of abelian varieties},
series = {Aspects of Mathematics},
volume = {33},
publisher = {Vieweg+Teubner Verlag},
address = {Braunschweig},
year = {1999},
pages = {65--89},
note = {\href{https://arxiv.org/pdf/alg-geom/9605011}{math.AG/9605011}}}

@article{Solomonopen,
author = {Solomon, Jake P. and Tukachinsky, Sara B.},
title = {Relative quantum cohomology},
journal = {J. Eur. Math. Soc.},
volume = {26},
number = {9},
year = {2024},
pages = {3497--3573},
note = {\href{https://arxiv.org/abs/1906.04795}{math.SG/1906.04795}}}

@article{Zongopen,
author = {Yu, Song and Zong, Zhengyu},
title = {Open {WDVV} equations and {F}robenius structures for toric
{C}alabi--{Y}au {$3$}-folds},
journal = {Forum Math. Sigma},
year = {2025},
volume = {13},
number = {e76},
pages = {1--29},
note = {\href{https://arxiv.org/abs/2312.06160}{math.AG/2312.06160}}}

@article{BCT,
author = {Buryak, Alexandr Y. and Clader, Emily and Tessler, Ran J.},
title = {Closed extended {$r$}-spin theory and the {G}elfand--{D}ickey wave function},
journal = {J. Geom. Phys.},
volume = {137},
year = {2019},
pages = {132--153},
note = {\href{https://arxiv.org/abs/1710.04829}{math.AG/1710.04829}}}

@inproceedings{JKV,
author = {Jarvis, Tyler J. and Kimura, Takashi and Vaintrob, Arkady},
title = {Gravitational descendants and the moduli space of higher spin curves},
booktitle = {Advanced in algebraic geometry motivated by physics (Lowell, MA, 2000)},
address = {Providence, RI},
year = {2001},
series = {Contemp. Math.},
publisher = {Amer. Math. Soc.},
note = {\href{https://arxiv.org/abs/math/0009066}{math.AG/0009066}}}

@article{MW,
author = {Madsen, Ib and Weiss, Michael},
title = {The stable moduli space of {R}iemann surfaces: {M}umford's conjecture},
journal = {Ann. Math.},
volume = {165},
pages = {843--941},
year = {2007},
note = {\href{https://arxiv.org/abs/math/0212321}{math.AT/0212321}}}

@article{borot_symmetries_2024,
	title = {Symmetries of {F}-cohomological field theories and {F}-topological recursion},
	author = {Borot, Gaëtan and Giacchetto, Alessandro and Umer, Giacomo},
	journal = {Commun. Math. Phys.},
	volume = {406},
	number = {248},
	year = {2025},
	note = {\href{http://arxiv.org/abs/2406.06304}{math-ph/2406.06304}}}

@article{abrams_two-dimensional_1996,
	title = {Two-dimensional topological quantum field theories and {F}robenius algebras},
	volume = {05},
	number = {05},
	journal = {J. Knot Theory Ramifications},
	author = {Abrams, Lowell},
	year = {1996},
	pages = {569--587}}

@article{BCFG,
author = {Liu, Si-Qi and Ruan, Yongbin and Zhang, Youjin},
title = {{BCFG} {D}rinfeld--{S}okolov hierarchies and {FJRW}-theory},
journal = {Invent. math.},
volume = {201},
pages = {711--772},
year = {2015},
note = {\href{https://arxiv.org/abs/1312.7227}{math.AG/1312.7227}}}

@article{buryak_new_2017,
	title = {New approaches to integrable hierarchies of topological type},
	volume = {72},
	number = {5},
	journal = {Russ. Math. Surv.},
	author = {Buryak, Alexandr Y.},
	year = {2017},
	pages = {841}}

@article{FSZ,
author = {Faber, Carel and Shadrin, Sergey and Zvonkine, Dimitri},
title = {Tautological relations and the {$r$}-spin {W}itten conjecture},
journal = {Ann. Sci. {\'{E}}c. {N}orm. {S}up.},
volume = {43},
number = {4},
pages = {621--665},
year = {2010},
note = {\href{https://arxiv.org/abs/math/0612510}{math.AG/0612510}}}

@incollection{Wahl,
author = {Wahl, Nathalie},
title = {Homological stability for mapping class groups of surfaces},
booktitle = {Handbook of moduli, Vol. III},
series = {Advanced Lectures in Mathematics},
volume = {26},
year = {2012},
pages = {547--583},
publisher = {International Press},
editor = {Farkas, Gavril and Morrison, Ian},
note = {\href{https://arxiv.org/abs/1006.4476}{math.GT/1006.4476}}}

@incollection{laptev_stable_2005,
	address = {Zürich, Switzerland},
	title = {The stable mapping class group and stable homotopy theory},
	booktitle = {European {Congress} of {Mathematics} {Stockholm}, {June} 27 -- {July} 2, 2004},
	publisher = {Eur. Math. Soc.},
	author = {Andersen, Jørgen Ellegaard and Weiss, Michael},
	editor = {Laptev, Ari},
	year = {2005},
	pages = {283--308}}

@article{MVpaper,
author = {Andersen, Jørgen Ellegaard and Borot, Ga\"etan and Charbonnier, S\'everin and Delecroix, Vincent and Giacchetto, Alessandro and Lewa\'nski, Danilo and Wheeler, Campbell},
title = {Topological recursion for {M}asur--{V}eech volumes},
journal = {J. Lond. Math. Soc.},
volume = {107},
number = {2},
pages = {254--332},
year = {2023},
note = {\href{https://arxiv.org/abs/1905.10352}{math.GT/1905.10352}}}

@article{chiarello_telemans_2016,
	title = {Teleman's classification of {2D} semisimple cohomological field theories},
	author = {Chiarello, Simone M.},
	year = {2016},
	note = {\href{http://arxiv.org/abs/1610.04368}{math.AG/1610.04368}}}

@article{buryak_extended_2021,
	title = {Extended {$r$}-spin theory in all genera and the discrete {KdV} hierarchy},
	volume = {386},
	journal = {Adv. Math.},
	author = {Buryak, Alexandr Y. and Rossi, Paolo},
	year = {2021},
	number = {6},
	pages = {107794},
	note = {\href{https://arxiv.org/abs/1806.09825}{math.AG/1806.09825}}}

@article{arsie_semisimple_2023,
	title = {Semisimple flat {F}-manifolds in higher genus},
	volume = {397},
	number = {1},
	journal = {Commun. Math. Phys.},
	note = {\href{https://arxiv.org/abs/2001.05599}{math.AG/2001.05599}},
	author = {Arsie, Alessandro and Buryak, Alexandr Y. and Lorenzoni, Paolo and Rossi, Paolo},
	year = {2023},
	pages = {141--197}}

@book{Eynardbook,
author = {Eynard, Bertrand},
title = {Counting surfaces},
publisher = {Birkh\"auser},
year = {2016},
series = {Progress in Mathematics}}

@article{AC94,
author = {Arbarello, Enrico and Cornalba, Maurizio},
title = {Combinatorial and algebro-geometric cohomology classes on the moduli spaces of curves},
journal = {J. Alg. Geom.},
volume = {5},
year = {1996},
number = {4},
pages = {705--749},
note = {\href{https://arxiv.org/abs/alg-geom/9406008}{math.AG/9406008}}}

@article{Lore,
author = {Lorenzoni, Paolo},
title = {Darboux--{E}gorov system, bi-flat {F}-manifolds and {P}ainlevé {VI}},
journal = {Int. Math. Res. Not.},
volume = {2014},
year = {2014},
number = {12},
pages = {3279--3302},
note = {\href{https://arxiv.org/abs/1207.5979}{math-ph/1207.5979}}}

@article{ABLR2,
author = {Arsie, Alessandro and Buryak, Alexandr Y. and Lorenzoni, Paolo and Rossi, Paolo},
title = {Riemannian {F}-manifolds, bi-flat {F}-manifolds, and flat pencils of metrics},
journal = {Int. Math. Res. Not.},
year = {2022},
volume = {2022},
number = {21},
pages = {16730--16778},
note = {\href{https://arxiv.org/abs/2104.09380}{math.DG/2104.09380}}}

@article{Basa,
author = {Basalaev, Alexey and Buryak, Alexandr},
title = {Open {WDVV} equations and {V}irasoro constraints},
journal = {Arnold Math. J.},
volume = {5},
pages = {145--186},
year = {2019},
note = {\href{https://arxiv.org/abs/1901.10393}{math-ph/1901.10393}}}

@article{ABLRint,
author = {Arsie, Alessandro and Buryak, Alexandr Y. and Lorenzoni, Paolo and Rossi, Paolo},
title = {Flat {F}-manifolds, {F}-{C}oh{FT}s, and integrable hierarchies},
journal = {Commun. Math. Phys.},
year = {2021},
volume = {388},
pages = {291--328},
note = {\href{https://arxiv.org/abs/2012.05332}{math-ph/2012.05332}}}

@article{hertling_weak_1999,
	title = {Weak {F}robenius manifolds},
	volume = {1999},
	number = {6},
	journal = {Int. Math. Res. Not.},
	author = {Hertling, Claus and Manin, Yuri I.},
	year = {1999},
	note = {\href{https://arxiv.org/abs/math/9810132}{math.OA/9810132}},
	pages = {277--286}}

@article{Lee1,
author = {Lee, Yuan-Pin},
title = {Invariance of tautological equations {I}: conjectures and applications},
journal = {J. Eur. Math. Soc.},
volume = {10},
number = {2},
pages = {399--413},
year = {2008},
note = {\href{https://arxiv.org/abs/math/0604318}{math.AG/0604318}}}

@article{Lee2,
author = {Lee, Yuan-Pin},
title = {Invariance of tautological equations {II}: {G}romov--{W}itten theory},
journal = {J. Amer. Math. Soc.},
volume = {22},
number = {2},
year = {2009},
pages = {331--352},
note = {\href{https://arxiv.org/abs/math/0605708}{math.AG/0605708}}}

@article{manin_f_2005,
	title = {{F}-manifolds with flat structure and {D}ubrovin's duality},
	volume = {198},
	number = {1},
	journal = {Adv. Math.},
	author = {Manin, Yuri I.},
	year = {2005},
	note = {\href{https://arxiv.org/abs/math/0402451}{math.DG/0402451}},
	pages = {5--26}}

@book{bott_differential_2010,
	address = {New York, NY},
	edition = {Reprint of 1st edition 1982},
	series = {Graduate texts in {Mathematics}},
	title = {Differential {forms} in {algebraic} {topology}},
	number = {82},
	publisher = {Springer},
	author = {Bott, Raoul and Tu, Loring W.},
	year = {2010}}

@article{harer_stability_1985,
	title = {Stability of the {homology} of the {mapping} {class} {groups} of {orientable} {surfaces}},
	volume = {121},
	number = {2},
	journal = {Ann. Math.},
	author = {Harer, John L.},
	year = {1985},
	pages = {215}}

@book{ivanov_subgroups_1992,
	address = {Providence, R.I},
	title = {Subgroups of {Teichmüller} modular groups},
	publisher = {Amer. Math. Soc.},
	author = {Ivanov, Nikolai V.},
	volume = {115},
	year = {1992}}

@article{arsie_darbouxegorov_2013,
	title = {From the {Darboux}--{Egorov} system to bi-flat {F}-manifolds},
	volume = {70},
	journal = {J. Geom. Phys.},
	author = {Arsie, Alessandro and Lorenzoni, Paolo},
	year = {2013},
	pages = {98--116},
	note = {\href{https://arxiv.org/abs/1205.2468}{math-ph/1205.2468}}}

@article{kontsevich_gromov-witten_1994,
	title = {Gromov--{Witten} classes, quantum cohomology, and enumerative geometry},
	volume = {164},
	number = {3},
	journal = {Commun. Math. Phys.},
	author = {Kontsevich, Maxim and Manin, Yuri I.},
	year = {1994},
	note = {\href{https://arxiv.org/abs/hep-th/9402147}{hep-th/9402147}},
	pages = {525--562}}

@incollection{dubrovin_geometry_1996,
	address = {Berlin, Heidelberg},
	title = {Geometry of {2D} topological field theories},
	booktitle = {Integrable {systems} and {quantum} {groups}: {Lectures} given at the 1st {session} of the {Centro} {Internazionale} {Matematico} {Estivo} ({C}.{I}.{M}.{E}.) held in {Montecatini} {Terme}, {Italy}, {June} 14–22, 1993},
	publisher = {Springer},
	author = {Dubrovin, Boris},
	editor = {Donagi, Ron and Dubrovin, Boris and Frenkel, Edward and Previato, Emma and Francaviglia, Mauro and Greco, Silvio},
	year = {1996},
	pages = {120--348},
	note = {\href{https://arxiv.org/abs/hep-th/9407018}{hep-th/9407018}}}

@article{pandharipande_relations_2015,
	title = {Relations on {$\overline{\mathcal{M}}_{g,n}$} via {$3$}-spin structures},
	volume = {28},
	number = {1},
	author = {Pandharipande, Rahul and Pixton, Aaron and Zvonkine, Dimitri},
	journal = {J. Amer. Math. Soc.},
	year = {2015},
	note = {\href{https://arxiv.org/abs/1303.1043}{math.AG/1303.1043}},
	pages = {279--309}}

@article{givental_semisimple_2001,
	title = {Semisimple {Frobenius} structures at higher genus},
	volume = {2001},
	number = {23},
	journal = {Int. Math. Res. Not.},
	author = {Givental, Alexander},
	year = {2001},
	note = {\href{https://arxiv.org/abs/math/0008067}{math.AG/0008067}},
	pages = {1265--1286}}

@article{looijenga_stable_1994,
	title = {Stable cohomology of the mapping class group with symplectic coefficients and of the universal {Abel}--{Jacobi} map},
	volume = {5},
	journal = {J. Alg. Geom.},
	author = {Looijenga, Eduard},
	year = {1994},
	note = {\href{https://arxiv.org/abs/alg-geom/9401005}{math.AG/9401005}}}

@incollection{getzler_jet-space_2004,
	address = {Wiesbaden},
	title = {The jet-space of a {Frobenius} manifold and higher-genus {Gromov}--{Witten} invariants},
	booktitle = {Frobenius {manifolds}: {quantum} {cohomology} and {singularities}},
	publisher = {Vieweg+Teubner Verlag},
	author = {Getzler, Ezra},
	editor = {Hertling, Claus and Marcolli, Matilde},
	year = {2004},
	note = {\href{https://arxiv.org/abs/math/0211338}{math.AG/0211338}},
	pages = {45--89}}

@article{teleman_structure_2012,
	title = {The structure of {2D} semi-simple field theories},
	volume = {188},
	number = {3},
	journal = {Invent. math.},
	author = {Teleman, Constantin},
	year = {2012},
	pages = {525--588},
	note = {\href{https://arxiv.org/abs/0712.0160}{math.AT/0712.0160}}}

@article{LPR09,
author = {Lorenzoni, Paolo and Pedroni, Marco and Raimondo, Andrea},
title = {{F}-manifolds and integrable systems of hydrodynamic type},
journal = {Arch. Math.},
volume = {47},
year = {2011},
number = {3},
pages = {163--180},
note = {\href{https://arxiv.org/abs/0905.4054}{math.DG/0905.4054}}}

@article{Pixtonthesis,
title = {The tautological ring of the moduli space of curves},
author = {Pixton, Aaron},
note = {PhD Thesis, Princeton University, 2013},
year = {2013}}

@article{DRJanda,
author = {Janda, Felix and Pandharipande, Rahul and Pixton, Aaron and Zvonkine, Dimitri},
title = {Double ramification cycles on the moduli spaces of curves},
journal = {Publ. Math. Inst. Hautes \'Etudes Sci.},
volume = {125},
year = {2017},
pages = {221--266},
note = {\href{https://arxiv.org/abs/1602.04705}{math.AG/1602.04705}}}

@article{BuryakDR,
author = {Buryak, Alexandr Y.},
title = {Double ramification cycles and integrable hierarchies},
journal = {Commun. Math. Phys.},
volume = {336},
year = {2015},
number = {3},
pages = {1085--1107},
note = {\href{https://arxiv.org/abs/1403.1719}{math-ph/1403.1719}}}

@inproceedings{PV01,
author = {Polishchuk, Alexander and Vaintrob, Arkady},
title = {Algebraic construction of {W}itten's top {C}hern class},
booktitle = {Advances in algebraic geometry motivated by physics (Lowell, MA, 2000)},
volume = {276},
year = {2001},
pages = {229--249},
series = {Contemp. Math.},
publisher = {AMS},
address = {Providence},
note = {\href{https://arxiv.org/abs/math/0011032}{math.AG/0011032}}}

@article{ChiodoWitten,
author = {Chiodo, Alessandro},
title = {The {W}itten top {C}hern class via {K}-theory},
journal = {J. Alg. Geom.},
volume = {15},
number = {4},
year = {2006},
pages = {681--707},
note = {\href{https://arxiv.org/abs/math/0210398}{math.AG/0210398}}}

@article{Pan-kappa,
author = {Pandharipande, Rahul},
title = {The {$\kappa$} ring of the moduli of curves of compact type},
journal = {Acta Math.},
volume = {208},
year = {2012},
pages = {335--388},
note = {\href{https://arxiv.org/abs/0906.2657}{math.AG/0906.2657} and \href{https://arxiv.org/abs/0906.2658}{math.AG/0906.2658}}}
\end{document}